\documentclass[10pt]{amsart}

\usepackage[top=1in,bottom=1in,left=1in,right=1in,heightrounded]{geometry}
\usepackage{setspace}
\usepackage[parfill]{parskip}
\usepackage{xspace}
\allowdisplaybreaks
\usepackage{quiver}
\usepackage{circledsteps}
\usepackage{longtable}
\usepackage{amsmath,amsthm,amssymb} 
\usepackage{mathtools}
\usepackage{mathrsfs}
\usepackage{stmaryrd}
\usepackage{thmtools}
\usepackage{thm-restate}
\usepackage{comment}
\usepackage{tikz}
\usepackage{tikz-cd}
\usetikzlibrary{decorations.pathmorphing,decorations.markings,patterns,hobby,arrows.meta}
\tikzset{>={Straight Barb[scale=0.75]}}
\tikzcdset{arrow style=tikz,
           diagrams={>=Straight Barb}
           }
\usepackage{pgfplots}
    \pgfplotsset{compat=1.18}
\usepackage{graphicx}
\usepackage{array}
\usepackage{makecell}

\usepackage{diagbox}
\usepackage{enumitem}
\usepackage{caption} 
\usepackage{float}
\usepackage{subcaption}
\usepackage{fancyhdr}
\usepackage{xcolor,colortbl}
\usepackage{soul}
\definecolor{cb-brown}{RGB}{106,74,60}
\definecolor{cb-blue}{RGB}{15,101,161}
\usepackage{hyperref}
\usepackage{url}
\hypersetup{colorlinks,citecolor=blue,filecolor=blue,linkcolor=blue,urlcolor=blue}
\usepackage[capitalise,noabbrev]{cleveref}
\usepackage[numbers]{natbib}
\usepackage{lmodern}
\usepackage{mathpazo}
\usepackage[cal=boondox,scr=boondox]{mathalfa}
\AtBeginDocument{
  \DeclareSymbolFont{AMSb}{U}{msb}{m}{n}
  \DeclareSymbolFontAlphabet{\mathbb}{AMSb}}
\DeclareMathAlphabet{\mathbx}{U}{BOONDOX-ds}{m}{n}
\SetMathAlphabet{\mathbx}{bold}{U}{BOONDOX-ds}{b}{n}
\DeclareMathAlphabet{\mathbbx}{U}{BOONDOX-ds}{b}{n}
\SetMathAlphabet{\mathcal}{bold}{U}{dutchcal}{b}{n}
\DeclareMathAlphabet{\mathbcal}{U}{dutchcal}{b}{n}
\newtheorem{theorem}{Theorem}[section]
\newtheorem{lemma}[theorem]{Lemma}
\newtheorem{corollary}[theorem]{Corollary}
\newtheorem{proposition}[theorem]{Proposition}
\newtheorem*{theproblem*}{Problem Statement}

\theoremstyle{definition}
\newtheorem{example}[theorem]{Example}
\newtheorem{remark}[theorem]{Remark}
\newtheorem{definition}[theorem]{Definition}
\newtheorem{question}[theorem]{Question}
\newcommand{\zz}{\mathbx Z}   
\newcommand{\qq}{\mathbx Q}   
\newcommand{\ff}{\mathbx F}   
\newcommand{\kk}{\mathbx k}   
\newcommand{\M}{\mathcal M}   
\renewcommand{\aa}{\mathcal A}   
\newcommand{\bt}{\mathrm{BT}_1}

\newcommand{\word}[1]{\textnormal{\textbf{\texttt{#1}}}}
\newcommand{\id}{\operatorname{id}}          
\newcommand{\im}{\operatorname{im}}          
\newcommand{\set}[1]{\left\{#1\right\}}     
\newcommand{\setp}[2]{\left\{#1\ \middle|\ #2\right\}} 
\newcommand{\Dd}{Dieudonn\'{e}\xspace}
\newcommand{\xM}{\mathbx{M}} 
\newcommand{\xW}{\mathbx{W}} 
\newcommand{\xe}{\mathbx{e}}
\newcommand{\xF}{\mathbx{F}}
\newcommand{\xV}{\mathbx{V}}
\newcommand{\TL}{\mathrm{TL}}
\newcommand{\NP}{\mathrm{NP}}
\newcommand{\STC}{\mathsf{STC}}
\newcommand{\xx}{\mathbx{x}}
\newcommand{\xy}{\mathbx{y}}

\newcommand{\frob}{\mathrm{\mathsf{Frob}}}
\newcommand{\ind}{\mathsf{Ind}}
\newcommand{\uni}{\mathsf{Uni}}
\newcommand{\bi}{\mathsf{Bi}}
\newcommand{\serre}{\textnormal{Serre}}
\newcommand{\odd}{\textnormal{odd}}
\newcommand{\even}{\textnormal{even}}
\newcommand{\none}{\cellcolor{gray!40}}

\renewcommand{\epsilon}{\varepsilon}
\renewcommand{\phi}{\varphi}
\renewcommand{\emptyset}{\varnothing}
\renewcommand{\geq}{\geqslant}
\renewcommand{\leq}{\leqslant}

\renewcommand{\preceq}{\preccurlyeq}

\numberwithin{equation}{section}

\colorlet{deewangcolor}{cyan!50}

\colorlet{emeraldcolor}{blue!30}

\colorlet{heidicolor}{magenta!50}

\colorlet{miacolor}{yellow!90}

\colorlet{sandracolor}{green!23!yellow}

\colorlet{stevencolor}{orange!55}

\title{On the classification of indecomposable Ekedahl--Oort strata in unitary Shimura varieties, and related Newton polygons}

\author{Emerald Andrews} 
\address{Department of Mathematics and Computer Science, Washington College, Chestertown, Maryland 21620, USA}
\email{emerald.andrews@washcoll.edu}

\author{Deewang Bhamidipati} 
\address{Department of Mathematics and Statistics, Carleton College, Northfield, MN 55057, USA} 
\email{bdeewang@carleton.edu}

\author{Maria Fox} 
\address{Department of Mathematics, Oklahoma State University, Stillwater, OK 74078, USA}
\email{maria.fox@okstate.edu}

\author{Heidi Goodson} 
\address{Department of Mathematics, Brooklyn College, City University of New York, Brooklyn, NY 11210 USA}
\email{heidi.goodson@brooklyn.cuny.edu}

\author{\\ Steven R. Groen} 
\address{Korteweg-de Vries Institute for Mathematics, University of Amsterdam, Amsterdam, Netherlands}
\email{s.r.groen@uva.nl}

\author{Sandra Nair} 
\address{Department of Mathematics, Colorado State University, Fort Collins, CO 80523, USA}
\email{sandra.nair@colostate.edu}

\begin{document}

\begin{abstract}
In this paper, we give a complete classification of indecomposable Ekedahl--Oort strata of Shimura varieties associated to the unitary group $\mathsf{GU}(a, b)$ over an odd inert prime. We show that each indecomposable stratum is one of four types: unitary unicycle, unitary bicycle, Serre unicycle, or Serre bicycle; the latter two types are named for a tensor construction of abelian varieties developed by Serre. We provide an algorithm that translates the description of a stratum in terms of words in the alphabet $\{\word{f},\word{v}\}$ to the corresponding Weyl group coset representative. Finally, using a $p$-adic lift, we construct a `tautological' point in each Ekedahl--Oort stratum, and compute its Newton polygon. As an application, we show that the indecomposable Ekedahl--Oort strata corresponding to unitary unicycles and Serre unicycles always intersect the supersingular locus. 
\end{abstract}

\maketitle

\section{Introduction}

{
Let $A$ be an abelian variety of dimension $q$ defined over $\kk$, an algebraically closed field of characteristic $p > 0$. In this positive characteristic setting, an interesting invariant is its $p$-torsion group subscheme $A[p]$, which is a finite commutative $\kk$-group scheme of rank $p^{2q}$. Such group schemes may be studied using linear algebraic methods via Dieudonn{\'e} theory, which associates to a finite commutative $\kk$-group scheme annihilated by $p$ a finite-dimensional $\kk$-vector space equipped with operators $F$ and $V$ (Frobenius and Verschiebung), called a mod-$p$ \Dd{} module. 

In an unpublished manuscript \cite{KraftpGruppen}, Kraft establishes a structure theorem for mod-$p$ \Dd{} modules by describing \emph{indecomposable} mod-$p$ \Dd{} modules, those that cannot be decomposed as a direct sum of (lower-dimensional) mod-$p$ \Dd{} modules. This allows for an elegant description of the (finitely many) possible isomorphism types of mod-$p$ \Dd{} modules of a fixed $\kk$-dimension.

The mod-$p$ \Dd{} modules that arise from $p$-torsion group schemes of abelian varieties possess a constraint in terms of the operators $F$ and $V$; they therefore belong to a sub-class of mod-$p$ \Dd{} modules called $\bt$ modules. Beginning with the work of Ekedahl and Oort \cite{OortStrat}, isomorphism types of $\bt$ modules have been used to establish a stratification, called the Ekedahl--Oort stratification, of characteristic $p$ fibers of Shimura varieties of PEL type, which have a moduli interpretation as moduli spaces of abelian varieties over $\kk$ with additional structure. 

The additional structure on an abelian variety parameterized by a Shimura varieties of PEL type imparts analogous additional structure on the associated $\bt$ modules. For example, in the case of the characteristic $p$ Siegel modular variety $\aa_q$, which parameterizes polarized abelian varieties over $\kk$ of dimension $q$, the associated $\bt$ modules admit a symplectic form and are called \emph{polarized} $\bt$ modules. 

In this paper, we are interested in the $\bt$ modules arising from characteristic $p$ fibers of unitary Shimura varieties for $p>2$. The characteristic $p$ unitary Shimura variety $\M(a,b)$ is the moduli space of $(a+b)$-dimensional abelian varieties over $\kk$ with a prime-to-$p$ polarization, an action by an order in an imaginary quadratic field wherein $p$ is inert, and level structure. The action satisfies a ``signature $(a,b)$'' condition. The associated $\bt$ modules admit additional structure and are named \emph{unitary} $\bt$ modules.

One of the main results of the present paper is the classification of \emph{indecomposable} unitary $\bt$ modules in arbitrary signature  (see \cref{thm: structure of unitary BT1s}). 
Geometrically, an indecomposable unitary $\bt$ module corresponds to an Ekedahl--Oort stratum of a unitary Shimura variety $\M(a,b)$ which is disjoint from the images of the \emph{product maps}
\[\M(a',b') \times \M(a - a',b - b') \to \M(a,b),\]
from unitary Shimura varieties of lower signatures, given by taking the product of the parameterized abelian varieties and their additional structure.

For the case of unitary Shimura varieties with signature $(q-1,1)$, \cite{bultel2006congruence} gives a complete description of the structure of indecomposable Ekedahl--Oort strata, using mod-$p$ Dieudonn\'e modules and a ``braid construction''. This is further expanded in \cite{VollaardWedhorn}. In \cite[Section 4]{EOstrata1}, we give a way to generate all decomposable Ekedahl--Oort strata of unitary Shimura varieties with signature $(q-2,2)$; we identify decomposable strata by exhibiting that they are in the images of product maps from unitary Shimura varieties of lower signature.

In this paper, we use the classical construction of Kraft diagrams \cite{KraftpGruppen} and words in the letters $\set{\text{\texttt{f}, \texttt{v}}}$ to classify the indecomposable unitary $\bt$ modules. As a starting point, we expand on results of \cite{OortStrat, PriesUlmerBT1Fermat} through a study of indecomposable polarized $\bt$ modules (see \cref{sec: words and indecomposable bt1 modules}). In \cref{sec:unitary bt1 modules}, we turn our attention to unitary $\bt$ modules and show that there are four indecomposable types: unitary unicycles and bicycles (named for the number of ``wheels'' in the Kraft diagrams) and Serre unicycles and bicycles (named for their relation to the Serre tensor construction for abelian varieties). Moreover, every unitary $\bt$ module decomposes uniquely as a direct sum of these indecomposable types. In \cref{appendix:indecomp-count}, we count how many of each of these  indecomposable types occur for a given signature.

Results of \cite{Moonengsas} establish a correspondence between Ekedahl--Oort strata of unitary Shimura varieties and cosets in certain Weyl groups. One advantage of this description is that the topological ordering between Ekedahl--Oort strata based on closure relations is captured directly by the partial ordering first developed by Bruhat and later generalized by He (see, for example, \cite{EOstrata1,bjorner2006combinatorics, HeOrder2007,PinkWedhornZiegler2011, ViehmannWedhorn2013,wedhorn2005specialization, wooding_2016}). 
In this present work, we give a method for going from words 
in the letters $\set{\text{\texttt{f}, \texttt{v}}}$
to Weyl group coset representatives expressed using the index set defined in \cite[Section 3]{EOstrata1}. This process is described in \cref{sec: word Weyl}.

Another invariant of interest for an abelian variety $A$ over $\kk$ is its $p$-divisible group $A[p^\infty]$. The isogeny-types of $p$-divisible groups yield a different stratification of Shimura varieties of PEL type, called the Newton stratification. In addition to studying the Ekedahl--Oort stratification by itself, one may also study how it interacts with the Newton stratification. More specifically, one can ask when a given Ekedahl--Oort stratum and a given Newton stratum intersect.

This is studied for general PEL type Shimura varieties in \cite{ViehmannWedhorn2013}, but more can be said if one restricts to a smaller class of Shimura varieties. Such interactions between the stratifications are well-known for Shimura varieties of Coxeter type \cite{goertzhe}. This includes unitary Shimura varieties in signature $(q,0)$ and  $(q-1,1)$ \cite{bultel2006congruence,VollaardWedhorn}, as well as the unitary Shimura variety in signature $(2,2)$ \cite{goertzhe,ghn-coxeter, howard2014supersingular}. In these cases, every Newton stratum is a union of Ekedahl--Oort strata.
Beyond these signatures, the situation is more nuanced since the varieties are no longer of Coxeter type, making it challenging to explicitly describe the intersection behavior.

The two stratifications for signature $(q-2,2)$ unitary Shimura varieties were studied in \cite{EOstrata1,shimada2024supersingular}, with the techniques in \cite{shimada2024supersingular} primarily focusing on the interaction of the supersingular locus, the unique closed stratum  of the Newton stratification, with the Ekedahl--Oort stratification. Studying Newton strata beyond the supersingular locus provides additional challenges, but has recently been done in some specific cases. For example, complete descriptions of the interactions between the two stratifications were given for the unitary Shimura variety with signature $(3,2)$ \cite{EOstrata2} and the Siegel modular variety $\aa_5$ \cite{GroenLupoianParker2026}, which both parameterize certain abelian varieties of dimension $5$. 

In \cref{sec: Tautological lifts and Newton strata} of this paper, we expand on the literature through a study of the interaction between the Ekedahl--Oort and Newton stratifications of unitary Shimura varieties in arbitrary signature. Using an extension of the tautological lift construction developed in \cite{Oort05simple}, we construct a point in each Ekedahl--Oort stratum and compute its Newton polygon (see \cref{thm: intersection}), implying that the given Ekedahl--Oort stratum and Newton stratum intersect. As a special case, we show that Ekedahl--Oort strata that decompose into unitary unicycles and Serre unicycles always intersect the supersingular locus of the Newton stratification (see \cref{corollary:balancedssintersection1,corollary:balancedssintersection2}).
}

\section*{Acknowledgments}

This collaboration was supported by the American Institute of Mathematics (AIM) SQuaRE program. The authors thank AIM and the NSF for their support through this valuable program. 

The authors thank Haley Covington for translating and typesetting the Kraft manuscript \cite{KraftpGruppen}. Her work was funded by the Toll Summer Research Program at Washington College.

H.G. was supported by NSF grant DMS-2201085 and the Max Planck Institute for Mathematics (MPIM) in Bonn, Germany. She is grateful to MPIM for its hospitality and financial support. S.R.G. was supported by grant VI.Vidi.223.028 of the Dutch Research Council (NWO).

\section{Preliminaries}

Throughout this paper, $\kk$ is an algebraically closed field of prime characteristic $p>2$. We study invariants associated with abelian varieties parameterized by certain Shimura varieties. 

\subsection{Shimura varieties} 
Let $K$ be an imaginary quadratic field such that $p$ is inert in $K$, with residue field $\ff_{p^2}$ at $p$.  The unitary Shimura variety $\M(a,b)$ arises as the characteristic $p$ fiber of an integral unitary Shimura variety associated to a PEL datum, constructed by Kottwitz \cite{Ko}; see \cite{vollaard} or \cite[Section 2.1]{EOstrata1} for a detailed definition. It is a quasi-projective variety over $\ff_{p^2}$ of dimension $ab$.

The $\kk$-points of $\M(a,b)$ correspond to isomorphism classes of tuples $(A, \lambda, \iota, \xi)$, where 
\begin{enumerate}[label=$\bullet$]
    \item $A$ is an abelian variety over $\kk$ of dimension $a+b$;
    \item $\lambda: A \overset{\sim}{\to} A^\vee$ is a prime-to-$p$ polarization;
    \item $\iota: \mathcal{O}_K \hookrightarrow \mathrm{End}(A)$ is an action of signature $(a,b)$ satisfying the Rosati involution condition; and
    \item $\xi:A[N](\kk) \to (\zz /N\zz)^{2(a+b)}$ is a level structure on $A$ for some $N$ coprime to $p$;
\end{enumerate}
subject to some compatibility conditions.

The characteristic $p$ {Siegel modular variety}, denoted by $\aa_q$, is a quasi-projective variety over $\ff_p$ of dimension $q(q+1)/{2}$. Its $\kk$-points correspond to isomorphism classes of $(A, \lambda, \xi)$, where 
\begin{enumerate}[label=$\bullet$]
    \item $A$ is an abelian variety over $\kk$ of dimension $q$;
    \item $\lambda: A \overset{\sim}{\to} A^\vee$ is a prime-to-$p$ polarization; and
    \item $\xi:A[N](\kk) \to (\zz /N\zz)^{2q}$ is a level structure on $A$ for some $N$ coprime to $p$,
\end{enumerate}
subject to some compatibility conditions.

\subsection{The Ekedahl--Oort and Newton stratification}\label{subsec:eo-newton-def}
An effective method of studying unitary Shimura varieties and Siegel modular varieties or, equivalently, the abelian varieties they parameterize, is via the \emph{Ekedahl--Oort} and  \emph{Newton stratifications} of these Shimura varieties. We briefly give the moduli description of these stratifications for the unitary Shimura variety $\M(a,b)$. These can be analogously described for the Siegel modular variety $\aa_q$. 

The Ekedahl--Oort stratification is based on the isomorphism class of the $p$-torsion group scheme of the parameterized abelian varieties. Two $\kk$-points $(A, \lambda, \iota, \xi)$ and $(A', \lambda', \iota', \xi')$ of $\M(a,b)$ are in the same \emph{Ekedahl--Oort stratum} if and only if the $p$-torsion group schemes equipped with induced action and polarization, $(A[p], \lambda, \iota)$ and $(A'[p], \lambda', \iota')$, are isomorphic as group schemes over  $\kk$. The Ekedahl--Oort strata are locally closed, and the closure of each Ekedahl--Oort stratum is a union of Ekedahl--Oort strata. For more details on Ekedahl--Oort stratifications, see \cite{EOstrata1,ViehmannWedhorn2013}.

The Newton stratification is based on the isogeny class of the $p$-divisible group of the parameterized abelian varieties. Two $\kk$-points $(A, \lambda, \iota, \xi)$ and $(A', \lambda', \iota', \xi')$ of $\M(a,b)$ are in the same \emph{Newton stratum} if and only if the $p$-divisible groups equipped with induced action and polarization, $(A[p^\infty], \lambda, \iota)$ and $(A'[p^\infty], \lambda', \iota')$, are isogenous over $\kk$. The Newton strata are also locally closed, and the closure of each Newton stratum is a union of Newton strata. The unique closed Newton stratum of a Shimura variety is called the \emph{basic locus}; in the case of $\M(a,b)$ this is the supersingular locus, which we denote as $\M(a,b)^{\textnormal{ss}}$. A $\kk$-point $(A, \lambda, \iota, \xi)$  of $\M(a,b)$ lies in $\M(a,b)^{\textnormal{ss}}$ if and only if $A$ is a supersingular abelian variety.

\subsection{\Dd{} Theory}\label{sec:preliminary DD theory}
The Ekedahl--Oort and Newton stratifications are defined in terms of finite commutative $p$-torsion group schemes and $p$-divisible groups. There is an equivalence between these objects and certain modules over $\xW \coloneqq \xW(\kk)$, the ring of Witt vectors of $\kk$, by \Dd{} theory. We briefly recall this equivalence (more detail can be found in \cite{DemazureGroups1972}).

Define the (non-commutative) \emph{\Dd{} ring} to be the following quotient of the free $\xW$-algebra generated by letters $\xF$ and $\xV$:
\[\xW\set{\xF,\xV} \coloneqq \xW\langle \xF,\xV\rangle /\langle \xF a - a^\sigma \xF,\, \xV a^\sigma - a \xV, \xF\xV - p, \xV\xF - p\ :\ a \in \xW\rangle,\]
where $\sigma = \xW(\frob)$ is the lift of the Frobenius automorphism $\frob\colon x \mapsto x^p$ on $\kk$ to $\xW$.

A \Dd{} module over $\kk$ is a left $\xW\set{\xF,\xV}$-module which is finitely generated as a $\xW$-module. There is a canonical categorical anti-equivalence, called the \emph{\Dd{} equivalence} and denoted $\mathbx{D}$, from the category of finite commutative $p$-torsion group schemes (resp. $p$-divisible groups) over $\kk$ to the category of \Dd{} modules over $\kk$ that are of finite $\xW$-length and annihilated by $p$ (resp. free and of finite rank as $\xW$-modules). Note that this is the statement of \emph{contravariant \Dd{} theory}; in this we differ here from \cite{DemazureGroups1972}, which discusses \emph{covariant \Dd{} theory}. 

In practice, it is useful to work with the following equivalent descriptions of the objects in the essential image(s) of $\mathbx{D}$.

\begin{definition} \label{def: p-adic DD}
A \emph{$p$-adic \Dd module} over $\kk$ is a triple $\overline{\xM} = (\xM ,\xF, \xV)$ such that
\begin{enumerate}
    \item $\xM$ is a free $\xW$-module with finite rank;
    \item $\xF: \xM \to \xM$ is a $\sigma$-linear map;
    \item $\xV: \xM \to \xM$ is a $\sigma^{-1}$-linear map; and
    \item $\xF \xV = \xV \xF = p$.
\end{enumerate}
A $p$-adic \Dd{} module over $\kk$ is equivalent to a \Dd{} module over $\kk$ that is free and of finite rank as a $\xW$-module. Therefore, $p$-adic \Dd{} modules are precisely those $\xW$-modules that arise from $p$-divisible groups.
\end{definition}

Reducing the \Dd{} ring modulo $p$, one obtains the $\kk$-algebra $\kk\set{F,V} = \xW\set{\xF,\xV}/p\xW\set{\xF,\xV}$; here $F$ and $V$ are the images of $\xF$ and $\xV$ in the quotient respectively. \Dd{} modules over $\kk$ that are annihilated by $p$ are canonically modules over $\kk\set{F,V}$. 

\begin{definition}
    A \emph{mod-$p$ \Dd{} module} over $\kk$ is a triple $(M,F,V)$ such that
\begin{enumerate}
    \item $M$ is a finite-dimensional $\kk$-vector space;
    \item $F:M \to M$ is a $\frob$-linear map;
    \item $V: M \to M$ is a $\frob^{-1}$-linear map; and
    \item $FV = VF = 0$. 
\end{enumerate}
A mod-$p$ \Dd{} module over $\kk$ is equivalent to a \Dd{} module over $\kk$ which is of finite $\xW$-length and annihilated by $p$. Therefore, mod-$p$ \Dd{} modules are precisely those $\xW$-modules that arise from finite commutative $p$-torsion group schemes.
\end{definition}

$\bt$ group schemes are those finite commutative $p$-torsion group schemes that arise as $p$-kernels of $p$-divisible groups. The name ``$\bt$'' stands for Barsotti-Tate groups truncated at level 1, where ``Barsotti-Tate groups'' is another name for $p$-divisible groups. Under the \Dd{} equivalence, the $\bt$ group schemes carve out the full subcategory of {$\bt$ modules} in the category of mod-$p$ \Dd{} modules.

\begin{definition} \label{def: BT1}
A \emph{$\bt$ module} over $\kk$ is a triple $\overline{M} = (M,F,V)$ such that
\begin{enumerate}
    \item $(M,F,V)$ is a mod-$p$ \Dd{} module; and
    \item $\ker (F)=\im(V)$ and $\ker(V)=\im(F)$.
\end{enumerate}
\end{definition}

Via the \Dd{} equivalence, $\bt$ modules are precisely the reductions modulo $p$ of $p$-adic \Dd{} modules. 
In the case of an abelian variety $A$ over $\kk$, its $p$-torsion group scheme $A[p]$ is the $p$-kernel of its $p$-divisible group $A[p^\infty]$. Therefore $\mathbx{D}(A[p])$ is a $\bt$ module. Hence the Ekedahl--Oort stratification is determined by isomorphism types of $\bt$ modules.

The additional structures on the abelian varieties parameterized by the unitary Shimura variety and the Siegel modular variety impart additional structures on the associated $p$-adic \Dd{} modules and thereby on the $\bt$ modules, as their reductions modulo $p$.

The $\bt$ modules equipped with additional structure arising from the Siegel modular variety are called \emph{polarized} $\bt$ modules; these are discussed in \cref{subsec:pol-bt1}. Those arising from the unitary Shimura variety are called \emph{unitary} $\bt$ modules. These modules, explored in \cref{subsec:unitary bt1 modules}, will be our primary objects of study.  In \cref{subsec:unitary-p-adic-mod}, we work with unitary $p$-adic \Dd{} modules and investigate certain lifts of unitary $\bt$ modules to \emph{unitary} $p$-adic \Dd{} modules, which are $p$-adic \Dd{} modules with additional structure arising from the unitary Shimura variety.

\subsection{Weyl group cosets}\label{sec:background Weyl} 
Moonen \cite{Moonengsas} relates Ekedahl--Oort strata of Shimura varieties of PEL type to cosets in certain Weyl groups. The Weyl group that is relevant for the study of $\M(a,b)$ is $\mathfrak{S}_q$, the symmetric group on $q = a+b$ elements. We consider $\mathfrak{S}_q$ as a Coxeter group with a set of \emph{simple reflections}
$S = \{s_1, \dots, s_{q-1} \}, \text{where }s_i = (i, \ i+1 ).$

Let $\mathbf{W}_{(a,b)}$ denote the subgroup of $\mathfrak{S}_q$ generated by the subset $\{s_1, ..., s_{q-1} \} \setminus \{s_b\}$ of $S$, and let $\mathbf{W}(a,b)$ denote the set of minimal-length coset representatives of $\mathbf{W}_{(a,b)}\setminus \mathfrak{S}_q$. The following theorem is a paraphrasing of \cite[Theorem~6.7]{Moonengsas}:

\begin{theorem}[Moonen]\label{thm:Moonenbij}
There is a bijection of sets:
\[\set{\text{Ekedahl--Oort Strata of $\M(a,b)$}} \longleftrightarrow \mathbf{W}(a,b).\]
\end{theorem}

We expand on Weyl group cosets associated to unitary Shimura varieties in \cref{sec: Weyl}. Such a correspondence also exists for Ekedahl--Oort strata of Siegel modular varieties $\aa_q$, where the Weyl group is a certain subgroup of $\mathfrak{S}_{2q}$. C.f. \cite[Section 3]{Moonengsas} or \cite[Section 5]{EOstrata1} for details. 

\section{Words and Decomposing \  (Polarized) \texorpdfstring{${\bt}$}{} Modules}\label{sec: words and indecomposable bt1 modules}

In this section, we recall the classical description of $\bt$ modules in terms of words in the alphabet $\{\word{f},\word{v}\}$. The main results are \cref{prop: Oort} and \cref{cor:Oort}, which summarize Oort's results \cite{OortStrat} on the decomposition of \emph{polarized} $\bt$ modules. See also \cite[Section 3]{PriesUlmerBT1Fermat} for an excellent exposition. In \cref{subsec:unitary bt1 modules} we will prove analogues of these results for \emph{unitary} $\bt$ modules.

\subsection{Words and \texorpdfstring{$\boldsymbol{\bt}$}{} modules}

In the manuscript \cite{KraftpGruppen}, Kraft proves a structure theorem for $\bt$ modules. Let $w=u_{\ell-1} \cdots u_0$ be a word in the alphabet $\{\word{f},\word{v}\}$ where $\ell \coloneqq \ell(w)$ is the \emph{length} of $w$.
The following construction associates a $\bt$ module of $\kk$-dimension $\ell$ to the word $w$. Let $M$ be the $\kk$-vector space spanned by $\{a_0, \ldots ,a_{\ell-1}\}$. The actions of $F$ and $V$ are defined on the basis $\{a_0, \ldots, a_{\ell-1}\}$ as follows: 
\[F(a_j) = \begin{cases} a_{j+1} &\hbox{if $u_j=\word{f}$,} \\ 0 &\hbox{if $u_j=\word{v}$,} \end{cases} \hspace{1cm} \text{and} \hspace{1cm}
    V(a_{j+1}) = \begin{cases} a_{j} &\hbox{if $u_j=\word{v}$,} \\ 0 &\hbox{if $u_j=\word{f}$.} \end{cases} \]

When $j=\ell-1$, it is understood that $a_{j+1}$ denotes $a_0$. The actions of $F$ and $V$ are extended $\frob$-linearly and $\frob^{-1}$-linearly respectively. This defines a $\bt$ module $\overline{M}(w)$.

This construction gives rise to a directed graph called a \emph{Kraft diagram}. The vertices correspond to the basis elements $\{a_0, \ldots , a_{\ell-1}\}$ and the edges between $a_j$ and $a_{j+1}$ represent either $F(a_j)=a_{j+1}$ or $V(a_{j+1})=a_j$. The letter $u_j$ of the word $w=u_{\ell-1} \cdots u_0$ determines the choice between these two possibilities. This is visualized in \cref{fig:cyclicKraft}.
\begin{figure}[H]
\begin{tikzcd}
& a_0 \arrow[r, "u_0", no head]      & a_1 \arrow[rd, "u_1", no head, bend left] &                                           \\
a_{\ell-1} \arrow[rd, "u_{\ell-2}"', no head, bend right] \arrow[ru, "u_{\ell-1}", no head, bend left] &                                    &                                           & a_2 \arrow[ld, "u_2", no head, bend left] \\
                & a_{\ell-2} \arrow[r, no head, dotted] & a_3  & 
\end{tikzcd}
\caption{A Kraft diagram}\label{fig:cyclicKraft}
\end{figure}
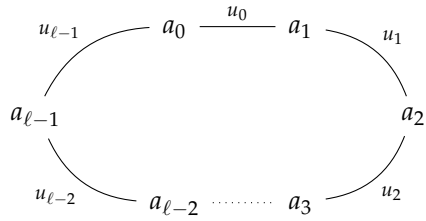

\begin{example}
Consider the word $w=\word{vfff}$. The Kraft diagram of $\overline{M}(\word{vfff})$ is given by
$$
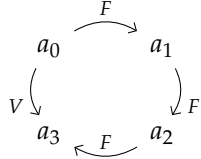
\begin{tikzcd}
a_0 \arrow[r, "F",bend left] \arrow[d, "V"',bend right]                              & a_1 \arrow[d, "F",bend left] \\
a_3 & a_2 \arrow[l, "F"', bend left]
\end{tikzcd}
$$
Note that $u_3=\word{v}$ indicates that $V(a_0) = a_3$ and $F(a_3)=0$. Similarly, for $i\in \{0,1,2\}$, we have $u_i = \word{f},$ indicating that $ F(a_i)=a_{i+1}$ and $V(a_{i+1})=0$.
\end{example}

Let $\tau$ be the rotation operator, i.e. $\tau(u_{\ell-1} \cdots u_0) = u_0 u_{\ell-1} \cdots u_1$. Then it is clear that, for any $n$, we have $\overline{M}(\tau^n(w)) \cong \overline{M}(w)$ by relabeling the basis vectors. Conversely, $\overline{M}(w) \cong \overline{M}(w')$ implies that $w' = \tau^n(w)$ for some $n$.

For a word $w$ and an integer $n>0$, we denote by $w^n$ the word that is obtained by concatenating $n$ copies of $w$. With this setup we have $\overline{M}(w^n) \cong \overline{M}(w)^n$. A word $w$ is called \emph{primitive} if it is not of the form $z^n$ for some word $z$ and $n>1$. Equivalently, $w$ is primitive if its rotation orbit has maximal size, i.e., the rotation orbit has size equaling $\ell(w)$. 

\begin{remark} \label{rmk: M((w)^n) = M(w)^n}
For later purposes, it is useful to make the isomorphism $\overline{M}(w^n) \cong \overline{M}(w)^n$ explicit. A precise description of the correspondence relies on the fact that $F$ and $V$ are semilinear operators. Let $w=u_{\ell-1} \cdots u_0$ and let $\overline{M}(w^n)$ be spanned by $\{a_0, \ldots, a_{n\ell-1}\}$, so that, for every $0 \leq j \leq \ell-1$ and $0\leq i \leq n-1$, we have 
\[F(a_{j+i\ell}) = \begin{cases} a_{j+1+i\ell} &\hbox{if $u_j=\word{f}$,} \\ 0 &\hbox{if $u_j=\word{v}$,} \end{cases} \hspace{2cm}
    V(a_{j+1+i\ell}) = \begin{cases} a_{j+i\ell} &\hbox{if $u_j=\word{v}$,} \\ 0 &\hbox{if $u_j=\word{f}$.} \end{cases} \]
Now let $\alpha_1, \ldots , \alpha_n$ be elements of  $\ff_{p^{n\ell}}$ that are linearly independent over $\ff_p$. 
We define the following change of basis: for $1\leq k \leq n$ and $0\leq j \leq \ell-1$, set
\[ b_{k,j}  \coloneqq  \sum_{i=0}^{n-1} \alpha_k^{p^{j+i\ell}} a_{j+i\ell}.\]
We claim that, for each $k$, the module spanned by $\{b_{k,0}, \ldots ,b_{k,\ell-1}\}$ is isomorphic to $\overline{M}(w)$. Indeed, we have
\begin{align*}
    F(b_{k,j})&= \sum_{i=0}^{n-1} \alpha_k^{p^{j+1+i\ell}} a_{j+1+i\ell}= b_{k,j+1} &\hbox{if $u_j=\word{f}$,} \\
    V(b_{k,j+1})&= \sum_{i=0}^{n-1} \alpha_k^{p^{j+i\ell}} a_{j+i\ell}=b_{k,j} &\hbox{if $u_j=\word{v}$},
\end{align*}
and both equal 0 otherwise.
Crucially, we have $\alpha_k^{p^{n\ell}}=\alpha_k$ since $\alpha_k \in \ff_{p^{n\ell}}$, which is needed to ensure $F(b_{k,\ell-1})=F(b_{k,0})$ for every $k$. Finally, the fact that $b_{k,i}$ are linearly independent is implied by the condition that $\alpha_1, \ldots ,\alpha_n$ are linearly independent over $\ff_p$ (see \cite[page 196]{MooreMatrix}).  This isomorphism is unique up to the choice of $\alpha_1, \ldots, \alpha_n$.

An isomorphism in the opposite direction can be obtained in a similar fashion. We may pick $\beta_1, \ldots, \beta_n \in \ff_{p^{n\ell}}$ that are linearly independent over $\ff_p$ and define
\[a_{j+i\ell}= \sum_{k=1}^{n} \beta_k^{j+i\ell} b_{k,j}.\]
\end{remark}

\begin{proposition}[Kraft, {\cite[Section 5]{KraftpGruppen}}]\label{prop: Kraft}
Any $\bt$ module $\overline{M}$ decomposes as follows:
\begin{equation*}
    \overline{M}=\bigoplus_{w \in W} \overline{M}(w).
\end{equation*}
Here $W$ is a multiset of primitive words in the alphabet $\{\word{f,v}\}$. This decomposition is unique up to rotating the words and permuting the factors.
\end{proposition}

\begin{corollary} \label{cor: Kraft}
For any $\ell \in \zz_{\geq 0}$, the assignment $w \mapsto \overline{M}(w)$ constitutes a bijection from the set of $\tau$-orbits of primitive words of length $\ell$ to the set of isomorphism classes of indecomposable $\bt$ modules of dimension $\ell$ as a $\kk$-vector space. 
\end{corollary}

\subsection{Polarized \texorpdfstring{$\boldsymbol{\bt}$}{} modules}\label{subsec:pol-bt1}

Our next step is to consider \emph{polarized} $\bt$ modules. These are $\bt$ modules that can occur as the \Dd module of the $p$-torsion group scheme of an abelian variety that admits a prime-to-$p$ polarization. For this reason, isomorphism classes of polarized $\bt$ modules of $\kk$-dimension $2q$ correspond to Ekedahl--Oort strata of $\aa_q$. In what follows, we work towards stating Oort's structure theorem (\cref{prop: Oort}) and include additional details that will be used later in the paper.

\begin{definition} \label{def:polBTmod}
A \emph{polarized $\bt$ module} is a quadruple $\widehat{M}=(M,F,V,\lambda)$ such that
\begin{enumerate} 
    \item $(M,F,V)$ is a $\bt$ module;
    \item $\lambda: M\times M \to k$ is a symplectic (bilinear, alternating, non-degenerate) pairing; and
    \item $\lambda( F(x), y ) = \lambda ( x,V(y) )^p$ for every $x,y \in M$. \label{property:FVcomp}
\end{enumerate}
\end{definition}

We say a $\bt$ module $(M,F,V)$ is \emph{polarizable} if there exists a polarized $\bt$ module $(M,F,V,\lambda)$. By \cite[Corollary 4.2]{PriesUlmerBT1Fermat}, the pairing $\lambda$ is unique, up to isomorphism, if it exists. Two polarized $\bt$ modules are isomorphic if and only if their underlying $\bt$ modules are isomorphic. It is a consequence of \Dd theory that a $\bt$ module $\overline{M}$ is polarizable if and only if it is isomorphic to its dual $\kk\set{F,V}$-module. 

Given a word $w=u_{\ell-1} \cdots u_0$, we define its \emph{dual word} $w^{\star}$ to be the word obtained from $w$ by changing every $\word{f}$ to $\word{v}$ and vice versa. A word $w$ is called \emph{self-dual} if $w^{\star}=\tau^n(w)$ for some $n$. 

For example, $(\word{ffffv})^\star =\word{vvvvf}$ , and $\word{fffvvv}$ is self-dual, since $(\word{fffvvv})^\star = \word{vvvfff}= \tau^3({\word{fffvvv}})$.

Since the dual module of $\overline{M}(w)$ is $\overline{M}(w^{\star})$ by \Dd theory, it follows that $\overline{M}(w) \cong \overline{M}(w^{\star})$ if and only if $w$ is self-dual. 

\subsubsection{Unicycles}

We now focus on $\bt$ modules $\overline{M}(w)$ that are self-dual (thus polarizable). This is equivalent to the word $w$ being self-dual. We begin by classifying what such words look like.

\begin{lemma} \label{lem: self-dual word}
Let $w$ be a self-dual word. Then $w=(yy^{\star})^n$ for some word $y$ such that $yy^{\star}$ is primitive, and for some integer $n>0$.
\end{lemma}
\begin{proof}
First, note that $w=z^n$ for some primitive word $z$ and some integer $n>0$, and $w$ is self-dual if and only if $z$ is as well. Thus, it suffices to show that $z$ is of the desired form.

Since $z=u_{\ell-1} \cdots u_0$ is self-dual, for some $0<m<\ell$, we have that $\tau^m(z)=z^{\star}$. Then for every $0\leq i < \ell-m$, we have that $u_{i+m}=u_i^{\star}$  and for $\ell-m \leq  i \leq \ell-1$, we have that $u_{i+m-\ell}=u_i^{\star}$. As a consequence, we have $\tau^{2m}(z)=z$. Now, since $z$ is primitive, this implies $2m=\ell$. Thus the lemma follows by letting $y$ be the first $m$ letters of $z$. 
\end{proof}

Note that the word $y$ here need not itself be primitive. For instance, taking $y=\word{ff}$ results in the primitive word $w=\word{ffvv}$.

We now describe explicitly the pairing $\lambda$ on $\overline{M}(w)$ when $w$ is self-dual.

\begin{lemma} \label{lem: pairing unicycle}
Let $w=(u_{\ell-1} \cdots u_0 u_{\ell-1}^{\star} \cdots u_0^{\star})^n$ be a self-dual word and consider the $\bt$ module $\overline{M}(w)=(M,F,V)$ with basis $\{a_0, \ldots, a_{2n\ell-1}\}$. Fix $c \in \kk^\times$ satisfying $c^{p^{\ell}}=-c$. Define a bilinear pairing by
$$\lambda(a_i,a_j) = \begin{cases}
    c^{p^i} &\hbox{if $j\equiv i+\ell \pmod{2n\ell}$,} \\
    0 &\hbox{otherwise.}
\end{cases}$$
Then $\widehat{M}=(M,F,V,\lambda)$ is a polarized $\bt$ module.
\end{lemma}

\begin{proof}
By construction, $(M,F,V)$ is a $\bt$ module and $\lambda$ is a non-degenerate bilinear pairing. We first show $\lambda$ is alternating. If $j=i+\ell$, then
\begin{equation}\label{eq:polarizedBT1_lambdacomputation1}
    \lambda(a_j,a_i) = c^{p^j}= c^{p^{i+\ell}}=\left(c^{p^{ \ell}}\right)^{p^i}=(-c)^{p^i} = -c^{p^i} = -\lambda(a_i,a_j).
\end{equation}
Furthermore, if $j=i+\ell-2n\ell,$ note that $c^{p^{2\ell}}=c$, so that
\begin{equation*}\label{eq:polarizedBT1_lambdacomputation2}
\lambda(a_j,a_i) = c^{p^j}= c^{p^{i+\ell-2n\ell}}=\left( c^{p^{-2n\ell}}\right)^{p^{i+\ell}}=c^{p^{i+\ell}} = -\lambda(a_i,a_j),
\end{equation*}
where the last equality holds by Equation \eqref{eq:polarizedBT1_lambdacomputation1}.

It remains to check the condition 
\begin{equation} \label{eq: pairing FV}
\lambda(F(x),y)=\lambda(x,V(y))^p.
\end{equation}
It suffices to check this on the basis $\{a_0, \ldots ,a_{2n\ell-1}\}$, so let $x=a_i$ and $y=a_j$. By the definition of $\lambda$, both sides of Equation~\eqref{eq: pairing FV} vanish unless $j=i+1+\ell$ or $j=i+1+\ell-2n\ell$, both of which imply $u_i=u_{j-1}^{\star}$. We consider the two options for $u_i$.
\begin{enumerate}
    \item Suppose $u_i=\word{f}$. Then $u_{j-1}=\word{v}$, and the construction of $\overline{M}(w)$ yields $F(a_i)=a_{i+1}$ and $V(a_j)=a_{j-1}$. Hence we obtain
    $$\lambda(F(a_i),a_j) = \lambda(a_{i+1},a_j)=c^{p^{i+1}}= \lambda(a_i,a_{j-1})^p = \lambda(a_i,V(a_j))^p,$$
    as desired. Note that this also holds when $i=2\ell-1$, since $c^{p^{2\ell}}=c$.
    \item Suppose $u_i=\word{v}$. Then $u_{j-1}=\word{f}$, so $F(a_i)=V(a_j)=0$.\qedhere
\end{enumerate}
\end{proof}

\begin{definition} \label{def: unicycle}

For a primitive, self-dual word $w$, the polarized $\bt$ module  $\widehat{M}(w)$ in \cref{lem: pairing unicycle} is called a \emph{unicycle}. The underlying $\bt$ module of $\widehat{M}(w)$ is $\overline{M}(w)$.
\end{definition}

Observe that unicycles are invariant under rotation: $\widehat{M}(\tau^n(w)) \cong \widehat{M}(w)$ for every $n$.

\subsubsection{Bicycles}

Note that, for any word $w$, the $\bt$ module $\overline{M}(w) \oplus \overline{M}(w^{\star})$ is self-dual, as the summands are each other's dual. We now detail the pairing on a $\bt$ module of this form.

\begin{lemma} \label{lem: pairing bicycle}
Let $w=u_{\ell-1}\cdots u_0$ be a word and consider the $\bt$ modules $\overline{M}(w)$, with basis $\{a_0, \ldots, a_{\ell-1}\}$, and $\overline{M}(w^{\star})$, with basis $\{b_0,\ldots ,b_{\ell-1}\}$. Let $(M,F,V)=\overline{M}(w) \oplus \overline{M}(w^{\star})$. Define the pairing
$$\lambda(a_i,b_j)=\delta_{i,j},$$ 
where $\delta_{i,j}$ is the Kronecker delta function, and extend this symplectically. Then $\widehat{M}=(M,F,V,\lambda)$ is a polarized $\bt$ module.
\end{lemma}
\begin{proof}
It suffices to check that $\lambda(F(a_i),b_j)=\lambda(a_i,V(b_j))^p$ for each $i,j$. Note that both sides of this equation vanish unless $j=i+1$. For $j=i+1$, there are two possibilities:
\begin{enumerate}
    \item Suppose $u_i=\word{f}$, so that $F(a_i)=a_{i+1}$. In this case, we have $u_i^{\star}=\word{v}$, which implies that $V(b_{i+1})=b_i$. Hence, 
$$\lambda(F(a_i),b_{i+1}) = \lambda(a_{i+1},b_{i+1})=1=\lambda(a_i,b_i)^p = \lambda(a_i,V(b_{i+1}))^p.$$
    \item Suppose $u_i=\word{v}$, so that $F(a_i)=V(b_{i+1})=0$. Then \[\lambda(F(a_i),b_{i+1})=\lambda(a_i,V(b_{i+1}))^p=0.\qedhere\]
 
\end{enumerate}
\end{proof}

\begin{definition} \label{def: bicycle}

For a primitive, non-self-dual word $w$, the polarized $\bt$ module  $\widehat{M}(w,w^{\star})$ in \cref{lem: pairing bicycle} is called a \emph{bicycle}. The underlying $\bt$ module of $\widehat{M}(w,w^{\star})$ is $\overline{M}(w) \oplus \overline{M}(w^{\star})$.
\end{definition}

Observe that bicycles are invariant under rotation: $\widehat{M}(w,w^{\star}) \cong \widehat{M}(\tau(w),\tau(w)^{\star})$. Bicycles are also invariant under dualizing $w$: we have $\widehat{M}(w,w^{\star}) \cong \widehat{M}(w^{\star},w)$.

\subsubsection{Structure theorem}

We can now state the structure theorem for polarized $\bt$ modules.

\begin{proposition}[Oort, {\cite[9.9]{OortStrat}}] \label{prop: Oort}
Any polarized $\bt$ module $\widehat{M}=(M,F,V,\lambda)$ decomposes as follows
$$\widehat{M} = \bigoplus_{w \in W_1} \widehat{M}(w) \oplus \bigoplus_{w \in W_2} \widehat{M}(w,w^{\star}).$$
Here, $W_1$ is a multiset of primitive self-dual words and $W_2$ is a multiset of primitive non-self-dual words. The decomposition is unique up to rotation, dualizing the words in $W_2$, and permuting the summands.
\end{proposition}

\begin{corollary} \label{cor:Oort}
If $\widehat{M}=(M,F,V,\lambda)$ is indecomposable as a polarized $\bt$ module, then $\widehat{M}$ is isomorphic to either a unicycle or a bicycle.
\end{corollary}

\begin{remark}\label{rmk:pol-complement}
For our purposes, a Kraft-style decomposition theorem is essential as it allows us to identify the indecomposable polarized $\bt$ modules. But one may also prove a weaker semisimplicity result for polarized $\bt$ modules. 

Let $\widehat{M}=(M,F,V,\lambda)$ be a polarized $\bt$ module. If $\widehat{N}=(N,F\vert_N,V\vert_N,\lambda\vert_{N\times N})$ is a polarized $\bt$ module, for a $\bt$ submodule $N$ of $M$, we call $\widehat{N}$ a \emph{polarized $\bt$ submodule} of $M$. The following lemma is a reframing of \cite[Lemma~3.4]{Achter-Pries} in this language. 
\begin{lemma}[Achter--Pries, {\cite[Lemma~3.4]{Achter-Pries}}]
    Let $P = \setp{m \in M}{\lambda_M(m,N) = 0}$ be the symplectic complement 
    of $N$ in $M$. Then $\widehat{P} = (P,F_M\vert_P,V_M\vert_P,\lambda_M\vert_{P\times P})$ is a polarized $\bt$ module and $\widehat{M} = \widehat{N} \oplus \widehat{P}$.
\end{lemma}
    
\end{remark}

\section{Decomposing Unitary \texorpdfstring{${\bt}$}{} modules}\label{sec:unitary bt1 modules}

The goal of this section is to establish an analogue
of \cref{prop: Kraft,prop: Oort} for \emph{unitary} $\bt$ modules. This is achieved in \cref{thm: structure of unitary BT1s}.

\subsection{Unitary \texorpdfstring{$\boldsymbol{\bt}$}{} modules}\label{subsec:unitary bt1 modules}

In what follows, if $A,B$ are $\kk$-vector spaces, $h:A \to B$ is a linear map, and $C$ is a subspace of $A$, then we define $C[h] \coloneqq \ker(h|_C).$

\begin{definition} \label{def: unitary BT1}
A \emph{unitary $\bt$ module} is a quintuple $\widetilde{M}=(M,F,V,\lambda,M=M_0 \oplus M_1)$ such that
\begin{enumerate}
    \item $(M,F,V,\lambda)$ is a polarized $\bt$ module;
    \item $F$ and $V$ are homogeneous of degree $1$ with respect to the decomposition $M=M_0 \oplus M_1$ (i.e. they exchange $M_0$ and $M_1$); and \label{property:homogeneous}
    \item $M_0$ and $M_1$ are Lagrangian with respect to $\lambda$ 
    (this is called the \emph{Rosati involution condition}). \label{property:RIC}
\end{enumerate}
Following \cite[(4.6)]{Moonengsas}, we define $(a,b) \coloneqq (\text{dim}_{\kk}(M_1[F]), \text{dim}_{\kk}(M_0[F]))$ 
to be the \emph{signature} of the unitary $\bt$ module $(M,F,V,\lambda,M=M_0 \oplus M_1)$. We will write $q \coloneqq a+b$, so that $\dim_\kk M = 2q$.
\end{definition}

\begin{remark} \label{rmk: iota}

Properties \eqref{property:homogeneous} and \eqref{property:RIC} in \cref{def: unitary BT1} are equivalent to equipping $M$ with an action $\iota: \ff_{p^2} \to \mathrm{End}_{\kk}(M)$ of $\ff_{p^2}$ such that it satisfies the \emph{Rosati involution condition} \[\lambda(\iota(\alpha)m,m') = \lambda(m,\iota(\alpha^p)m')\]
for $\alpha \in \ff_{p^2}$ and $m,m' \in M$. With this condition we have
\[
M_0 = \setp{m \in M}{\iota(\alpha)m = \alpha m} \hspace{1cm} \text{and} \hspace{1cm}
M_1 = \setp{m \in M}{\iota(\alpha)m = \alpha^p m}.
\]
Therefore, a unitary $\bt$ module $\widetilde{M}=(M,F,V,\lambda,M=M_0 \oplus M_1)$ may equivalently be expressed as the quintuple $\widetilde{M}=(M,F,V,\lambda,\iota)$, where $\iota$ is the action of $\ff_{p^2}$ corresponding to the decomposition $M = M_0 \oplus M_1$. 
\end{remark}

Just like how isomorphism classes of polarized $\bt$ modules of $\kk$-dimension $2q$ correspond to Ekedahl--Oort strata of $\aa_q$, isomorphism classes of unitary $\bt$ modules of signature $(a,b)$ correspond to Ekedahl--Oort strata of the unitary Shimura variety $\M(a,b)$. 

We say that a $\bt$ module $(M,F,V)$, or a polarized $\bt$ module $(M,F,V,\lambda)$, \emph{admits unitary structure} if there exists a unitary $\bt$ module $(M,F,V,\lambda,M=M_0 \oplus M_1)$. Note that two different unitary $\bt$ modules can have the same underlying (polarized) $\bt$ module. For instance, there are superspecial unitary $\bt$ modules of every signature. There are also pairs of non-isomorphic unitary $\bt$ modules with the same signature and the same underlying polarized $\bt$ module.

Given a unitary $\bt$ module $\widetilde{M}=(M,F,V,\lambda,M=M_0 \oplus M_1)$, we can form its \emph{conjugate unitary $\bt$ module} $\widetilde{M}^*$ by exchanging $M_0$ and $M_1$. That is, define 
$$\widetilde{M}^*  \coloneqq  (M,F,V,\lambda,M=M_1 \oplus M_0).$$
From the definition, it is clear that $\widetilde{M}^{**}=\widetilde{M}$ and $\widetilde{M}^*$ has signature $(b,a)$.

In this section, we will prove a Kraft-style decomposition theorem for unitary $\bt$ modules, similar to Oort's result for polarized $\bt$ modules (\cref{prop: Oort}). This will allow us to identify the indecomposable unitary $\bt$ modules. But one may state and prove a weaker semisimplicity result for unitary $\bt$ modules as well, similar to \cref{rmk:pol-complement} for polarized $\bt$ modules. 

\begin{remark}\label{rmk:uni-complement}
Let $\widetilde{M}=(M,F,V,\lambda,M = M_0 \oplus M_1)$ be a unitary $\bt$ module. For a (polarized) $\bt$ submodule $N$ of $M$, define $N_i \coloneqq N \cap M_i$ for $i=0,1$. 
If $\widetilde{N}=(M,F\vert_N,V\vert_N,\lambda\vert_{N \times N},N = N_0 \oplus N_1)$ is a unitary $\bt$ module, we call $\widetilde{N}$ a \emph{unitary $\bt$ submodule} of $M$.

Consider the symplectic complement $P = \setp{m \in M}{\lambda(m,N) = 0}$ of $N$ in $M$, and define \[P_0  \coloneqq  P \cap M_0 = \setp{m \in M_0}{\lambda(m,N_1) = 0} \quad \text{and} \quad P_1  \coloneqq  P \cap M_1 = \setp{m \in M_1}{\lambda(m,N_0) = 0}.\] 
Then $\widetilde{P} = (P,F\vert_P,V\vert_P,\lambda\vert_{P\times P},P = P_0 \oplus P_1)$ is a unitary $\bt$ submodule of $\widetilde{M}$ and $\widetilde{M} = \widetilde{N} \oplus \widetilde{P}$ as unitary $\bt$ modules.
    
\end{remark}

In \cite[(4.9)]{Moonengsas}, Moonen constructs a `standard object' for each isomorphism class of $\bt$ modules with given additional endomorphisms, of which unitary $\bt$ modules are a special case. In particular, he exhibits a basis $B_0 \coloneqq \{e_{0,1}, \ldots ,e_{0,q}\}$ of $M_0$ and a basis $B_1 \coloneqq \{e_{1,1}, \ldots ,e_{1,q}\}$ of $M_1$ such that the basis $B \coloneqq B_0\cup B_1$ of $M$ is preserved by $F$ and $V$.

We begin by constructing, in \cref{subsec:unitary-unicycles,subsec:unitary-bicycles}, the unitary analogues of the polarized unicycles and bicycles constructed in \cref{lem: pairing unicycle,lem: pairing bicycle}, respectively.
Our main result in this section, \cref{thm: structure of unitary BT1s}, is that the underlying polarized $\bt$ module structure of indecomposable unitary $\bt$ modules will either be a unicycle or bicycle or a \emph{Serre tensor construction} of them respectively. The Serre tensor construction is explained in \cref{section:STC}.

\subsection{Unitary unicycles}\label{subsec:unitary-unicycles}

\begin{lemma} \label{lem: unitary M(w)}
Let $w$ be a primitive word of even length. Then $(M,F,V)=\overline{M}(w)$ admits unitary structure if and only if $w=yy^{\star}$ for some word $y$ of odd length. 
\end{lemma}

\begin{proof}
We begin with the ``only if'' statement. Note that $w$ must be self-dual for $\overline{M}(w)$ to be polarizable. 
Hence \cref{lem: self-dual word} yields that $w=yy^{\star}$ for some word $y$. It remains to be shown that $q \coloneqq \ell(y)$ is odd. 
This comes essentially from the Rosati involution condition. More precisely, let $e_{0,1} \in B_0$. Because of the condition $\lambda(F(x),y) = \lambda(x,V(y))^p$ and the assumption that $yy^\star$ is primitive, it follows that $e_{0,1}$ can only pair non-trivially with the element of $M$ that is `antipodal' to $e_{0,1}$, i.e. the element that is $q$ steps away from $e_{0,1}$ in the Kraft diagram. However, any element of $M$ that is an even number of steps away from $e_{0,1}$ is an element of $M_0$, and the Rosati involution condition prescribes that $M_0$ is Lagrangian with respect to $\lambda$. Thus $q$ must be odd. 

For the ``if'' statement, let $w=yy^{\star}$ with $y$ of odd length $q$. Form $\widehat{M}(w)$ as in \cref{lem: pairing unicycle}. Let
\[
M_0  \coloneqq  \text{span}_{\kk} \{a_j \; | \; j \text{ even}\} \hspace{1cm} \text{and} \hspace{1cm}
M_1  \coloneqq  \text{span}_{\kk} \{a_j \; | \; j \text{ odd} \}.
\]

By construction, $F$ and $V$ interchange the summands $M_0$ and $M_1$. Finally, the fact that $M_0$ and $M_1$ are Lagrangian with respect to $\lambda$ follows from the fact that $q$ is odd, which ensures that $a_j$ lies in a different summand than $a_{j\pm q}$. \end{proof}

Note that, since $q$ is odd in the context of \cref{lem: unitary M(w)}, when defining the pairing $\lambda$, one may pick the constant $c$ in \cref{lem: pairing unicycle} to simply satisfy $c^p=-c$. This then implies that $c^{p^q}=-c$. 

Assume $yy^{\star}$ is primitive and $y$ has odd length. For any $n>0$, the construction in \cref{lem: unitary M(w)}, combined with \cref{rmk: M((w)^n) = M(w)^n}, allows one to equip $\overline{M}((yy^\star)^n)$ with unitary structure.

\begin{definition} \label{def: unitary unicycle}
For a word $w=(yy^{\star})^n$ with $y$ a word of odd length and $yy^{\star}$ primitive, we let $\widetilde{M
}(w)$ denote the unitary $\bt$ module constructed in \cref{lem: unitary M(w)}. In particular, the underlying $\bt$ module of $\widetilde{M}(w)$ is $\overline{M}(w)$. When $w$ is primitive, we call $\widetilde{M}(w)$ a \emph{unitary unicycle}.
\end{definition}

Note that unitary unicycles are indecomposable as unitary $\bt$ modules, since they are indecomposable as $\bt$ modules.

\begin{remark} \label{rmk: don't rotate unicycles!}
It is important to note that $\widetilde{M}(w)$ is \emph{not} invariant under rotation, because this exchanges $M_0$ and $M_1$. In fact, rotating the word $w$ amounts to conjugating $\widetilde{M}(w)$: 
\[ \widetilde{M}(\tau(w)) \cong \widetilde{M}(w)^*.\]
Since rotating twice preserves $M_0$ and $M_1$, we have that $\widetilde{M}(\tau^{2n}(w)) \cong \widetilde{M}(w)$ for every $n$.
\end{remark}

\begin{remark} \label{rmk: unicycle signature}
Let $w=yy^{\star}=u_{q-1} \cdots u_0 u_{q-1}^{\star} \cdots u_0^{\star}$ be primitive and assume $\ell(y)=q$ is odd. We can read off the signature of $\widetilde{M}(w)$ as follows. The number $b=\dim_{\kk}(M_0[F])$ equals the number of $a_{2j}$ that are killed by $F$. 

When $2j<q$, observe that $F(a_{2j})=0$ is equivalent to $u_{2j}^\star=\texttt{v}$. For $2j>q$, we have $F(a_{2j})=0$ if and only if $u_{2j-q}=\texttt{v}$. We conclude 
\begin{align*}
    a&= \#\{0\leq j \leq \lfloor q/2 \rfloor \mid u_{2j} = \word{v} \} + \#\{1 \leq j \leq \lfloor q/2 \rfloor  \mid  u_{2j-1}=\word{f}\}, \\
    b&= \#\{1\leq j \leq \lfloor q/2 \rfloor \mid u_{2j-1}=\word{v} \} + \#\{0\leq j \leq \lfloor q/2 \rfloor \mid u_{2j}=\word{f}\}.
\end{align*}
\end{remark}

\begin{example} \label{example:unicycle diagram}
The word $\word{ffvffvvfvv}$ with Kraft diagram shown in \cref{fig: ffvffvvfvv} is a unitary unicycle of the form $yy^\star$, where $y = \word{ffvff}$. 

\begin{figure}[H]
    \centering
\[\begin{tikzcd}[cramped]
	& a_0 & a_1 & a_2 & a_3 & \\
	a_9 &&&&& a_4 \\
	& a_8 & a_7 & a_6 & a_5
	\arrow["F", from=1-2, to=1-3]
	\arrow["V"', curve={height=12pt}, from=1-2, to=2-1]
	\arrow["F", from=1-3, to=1-4]
	\arrow["V"', from=1-5, to=1-4]
	\arrow["F", curve={height=-12pt}, from=1-5, to=2-6]
	\arrow["V"', curve={height=12pt}, from=2-1, to=3-2]
	\arrow["F", curve={height=-12pt}, from=2-6, to=3-5]
	\arrow["F", from=3-3, to=3-2]
	\arrow["V"', from=3-3, to=3-4]
	\arrow["V"', from=3-4, to=3-5]
\end{tikzcd}\]     \caption{The Kraft diagram for $\word{ffvffvvfvv}$}
    \label{fig: ffvffvvfvv}
\end{figure}
We conclude that the signature of $\widetilde{M}(\word{ffvffvvfvv})$ is $(3,2)$ by applying \cref{rmk: unicycle signature} and noting that we have $u_0, u_1, u_3, u_4 = \word{f}$, and $u_2 = \word{v}$.
\end{example}

\subsection{Unitary bicycles and Serre tensor constructions}\label{subsec:unitary-bicycles}

In an analogous fashion to the preceding subsection, we now construct unitary $\bt$ modules whose underlying polarized $\bt$ module is of the form $\widehat{M}(w,w^\star)$.

\begin{lemma} \label{lem: unitary M(w)+M(w)-c}
Let $w$ be a word of even length. Then $(M,F,V)=\overline{M}(w) \oplus \overline{M}(w^{\star})$ admits unitary structure. 
\end{lemma}
\begin{proof}
Write $w=u_{q-1} \ldots u_0$, where $q$ is even. As in \cref{lem: pairing bicycle}, let $\{a_0, \ldots , a_{q-1}\}$ be the basis of $\overline{M}(w)$ and let $\{b_0, \ldots ,b_{q-1}\}$ be the basis of $\overline{M}(w^{\star})$. Let $\lambda$ be as defined in \cref{lem: pairing bicycle} and set 
\begin{align*}
    M_0&=\text{span}_{\kk} \{a_j\mid j \text{ even}\} \oplus \text{span}_{\kk} \{b_j \mid j \text{ odd}\}, \\
    M_1&=\text{span}_{\kk} \{a_j \mid j \text{ odd}\} \oplus \text{span}_{\kk} \{b_j \mid j \text{ even}\}.
\end{align*}
Then one readily verifies that $F$ and $V$ exchange $M_0$ and $M_1$ and moreover $M_0$ and $M_1$ are Lagrangian with respect to $\lambda$.
\end{proof}

\begin{definition} \label{def: widetilde M(w,wc)}
For a word $w$ of even length, we let $\widetilde{M}(w,w^{\star})$ denote the unitary $\bt$ module constructed in \cref{lem: unitary M(w)+M(w)-c}. In particular, the underlying $\bt$ module of $\widetilde{M}(w,w^{\star})$ is $\overline{M}(w)\oplus \overline{M}(w^{\star})$.
\end{definition}

The following special cases of $\widetilde{M}(w,w^\star)$ will be of primary interest:

\begin{definition} \label{def: (Serre) bicycles}
Let $\widetilde{M}(w,w^{\star})$ be a unitary $\bt$ module. Then 
\begin{enumerate}[label=$\mathrlap{\oplus}\otimes$]
    \item if $w$ is a non-self-dual-primitive word of even length, we refer to $\widetilde{M}(w,w^{\star})$ as a \emph{unitary bicycle};
    \item if $w$ is a self-dual-primitive word whose length is divisible by $4$, we refer to $\widetilde{M}(w,w^{\star})$ as a \emph{Serre unicycle}; and 
    \item if $w=zz$ for a primitive word $z$ of odd length (necessarily non-self-dual) we refer to $\widetilde{M}(w,w^{\star})$ as a \emph{Serre bicycle}. 
\end{enumerate}

\end{definition}

See \cref{example: bicycles} for a demonstration of each type of unitary $\bt$ module described above.

\begin{remark} \label{rmk: Serre tensor products}
If $w$ is a self-dual-primitive word whose length is divisible by $4$, then $\widehat{M}(w)$ is a unicycle that does not admit unitary structure (see \cref{lem: unitary M(w)}). In that case $\widetilde{M}(w,w^{\star})$ is isomorphic to the \emph{Serre tensor construction} of the unicycle $\widehat{M}(w)$; see \cref{prop: STC unicycle}. 

Similarly, if $w=zz$ and $z$ is a non-self-dual-primitive word of odd length, then $\widehat{M}(z,z^{\star})$ is a bicycle that does not admit unitary structure (since the Kraft diagram has a loop of odd length), and $\widetilde{M}(w,w^{\star})$ is the Serre tensor construction of the bicycle $\widehat{M}(z,z^{\star})$. See \cref{section:STC} for details on the Serre tensor construction and how it applies to our setting. 
\end{remark}

\begin{lemma} \label{lem: (Serre) bicycle indecomposable}
In all cases of \cref{def: (Serre) bicycles}, $\widetilde{M}(w,w^{\star})$ is indecomposable as a unitary $\bt$ module.
\end{lemma}
\begin{proof}
First note that unitary bicycles are indecomposable as polarized $\bt$ modules, so in particular they are indecomposable as unitary $\bt$ modules.

Now, let $\widetilde{M}(w,w^{\star})$ be a Serre unicycle. Then the underlying $\bt$ module of $\widetilde{M}(w,w^{\star})$ is isomorphic to $\overline{M}(w) \oplus \overline{M}(w)$. Since $w$ is primitive, any non-trivial factor of $\widetilde{M}(w,w^{\star})$ would be isomorphic to $\overline{M}(w)$ as a $\bt$ module. But such a $\bt$ module does not admit unitary structure by \cref{lem: unitary M(w)}.

Finally, we consider the case when $\widetilde{M}(w,w^{\star})$ is a Serre bicycle. Then the underlying polarized $\bt$ module of $\widetilde{M}(w,w^{\star})$ is isomorphic to $\widehat{M}(z,z^{\star}) \oplus \widehat{M}(z,z^{\star})$ and both summands are indecomposable as polarized $\bt$ modules. Hence any non-trivial factor of $\widetilde{M}(w,w^{\star})$ would be isomorphic to $\widehat{M}(z,z^{\star})$ as a polarized $\bt$ module. But this is not possible, since the loop $\overline{M}(z)$ has odd length and so we cannot form subspaces $M_0$ and $M_1$ that are exchanged by $F$ and $V$. 
\end{proof}

\begin{remark} \label{rmk: don't rotate bicycles!}
Similar to \cref{rmk: don't rotate unicycles!}, unitary bicycles are also not invariant under rotation, but are invariant under double rotation. Moreover, dualizing $w$ has, up to isomorphism, the same effect as rotation:
$$\widetilde{M}(\tau(w),\tau(w)^{\star}) \cong \widetilde{M}(w,w^{\star})^* \cong \widetilde{M}(w^{\star},w),$$
so that $\widetilde{M}(\tau(w)^{\star},\tau(w)) \cong \widetilde{M}(w,w^{\star})$. Note the distinction between $\star$ (dual) and $*$ (conjugate) in the previous equation.

On the other hand, if $\widetilde{M}(w,w^{\star})$ is a Serre unicycle or a Serre bicycle, then it is self-conjugate and, therefore, invariant under rotation and under dualizing $w$.
\end{remark}

\begin{remark} \label{rmk: bicycle signature}
We compute the signature of a unitary $\bt$ module of the form $\overline{M}(w)\oplus \overline{M}(w^{\star})$ with $w=u_{q-1} \cdots u_0$ in the following way. By definition, $b=\dim_{\kk} (M_0[F])$ equals the number of $a_{2j}$ and $b_{2j-1}$ that are killed by $F$. We observe that $F(a_{2j})=0$ is equivalent to $u_{2j}=\word{v}$. Similarly, $F(b_{2j-1})=0$ is equivalent to $u_{2j-1}^{\star}=\word{v}$, which implies $u_{2j-1}=\word{f}$. We conclude
\begin{align*}
    a&= \#\{ 1\leq j \leq q/2 \; | \; u_{2j-1}=\word{v} \}+ \# \{0\leq j < q/2 \; | \; u_{2j}=\word{f}\},
     \\
    b&= \# \{0\leq j < q/2 \; | \; u_{2j}=\word{v}\} + \#\{ 1\leq j \leq q/2 \; | \; u_{2j-1}=\word{f} \} .
\end{align*}

If $\widetilde{M}(w,w^{\star})$ is a Serre unicycle or a Serre bicycle then, by \cref{def: (Serre) bicycles}, we must have $a=b$. We deduce that Serre unicycles can only occur in parallel signature $(a,a)$ with $a$ even, while Serre bicycles can only occur in parallel signature $(a,a)$ with $a$ odd.  

\end{remark}

\begin{example}\label{example: bicycles}
We demonstrate the above definitions of unitary $\bt$ modules $\widetilde{M}(w,w^{\star})$:  
\begin{enumerate}[label=(\roman*)]
    \item Consider $w = \word{ffffvf}$, a primitive and non-self-dual word of even length. Then $\widetilde{M}(w,w^\star)$ 
    is the unitary bicycle shown below. Note that applying \cref{rmk: bicycle signature}, we uncover the signature $(4,2)$.

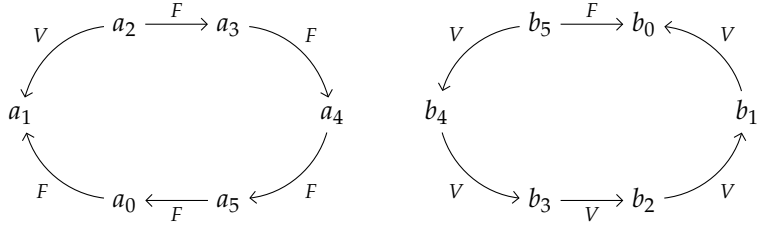
\[\begin{tikzcd}[cramped]
	& a_2 & a_3 &&& b_5 & b_0 \\
	a_1 &&& a_4 & b_4 &&& b_1 \\
	& a_0 & a_5 &&& b_3 & b_2
	\arrow["F", from=1-2, to=1-3]
	\arrow["V"', curve={height=12pt}, from=1-2, to=2-1]
	\arrow["F", curve={height=-12pt}, from=1-3, to=2-4]
	\arrow["F", from=1-6, to=1-7]
	\arrow["V"', curve={height=12pt}, from=1-6, to=2-5]
	\arrow["V"', curve={height=12pt}, from=2-8, to=1-7]
	\arrow["F", curve={height=-12pt}, from=2-4, to=3-3]
	\arrow["V"', curve={height=12pt}, from=3-7, to=2-8]
	\arrow["F", curve={height=-12pt}, from=3-2, to=2-1]
	\arrow["F", from=3-3, to=3-2]
	\arrow["V"', curve={height=12pt}, from=2-5, to=3-6]
	\arrow["V"', from=3-6, to=3-7]
\end{tikzcd}\]   
    \item Consider $w = \word{ffvv}$, a primitive and self-dual word with length divisible by $4$. Then $\widetilde{M}(w,w^\star)$ is a Serre unicycle, and, by applying \cref{rmk: bicycle signature}, we find that the signature is $(2,2)$.

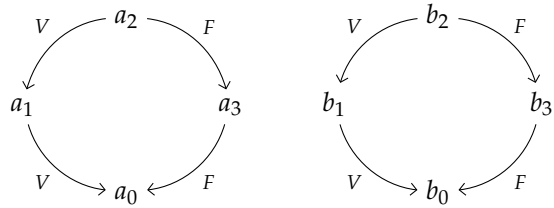
\[\begin{tikzcd}[cramped]
	& a_2 &&& b_2 \\
	a_1 && a_3 & b_1 && b_3 \\
	& a_0 &&& b_0
	\arrow["V"', curve={height=12pt}, from=1-2, to=2-1]
	\arrow["F", curve={height=-12pt}, from=1-2, to=2-3]
	\arrow["V"', curve={height=12pt}, from=1-5, to=2-4]
	\arrow["F", curve={height=-12pt}, from=1-5, to=2-6]
	\arrow["V"', curve={height=12pt}, from=2-1, to=3-2]
	\arrow["F", curve={height=-12pt}, from=2-3, to=3-2]
	\arrow["V"', curve={height=12pt}, from=2-4, to=3-5]
	\arrow["F", curve={height=-12pt}, from=2-6, to=3-5]
\end{tikzcd}\]

 \item Lastly, we consider $w=zz$, where $z=\word{ffv}$ is a primitive and non-self-dual word of odd length. Then $\widetilde{M}(w,w^\star)$ is a Serre bicycle in signature $(3,3)$.

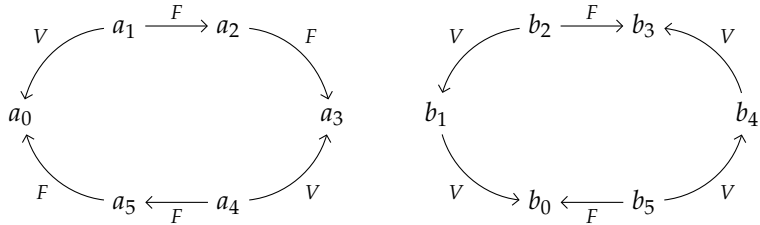
\[\begin{tikzcd}[cramped]
	& a_1 & a_2 &&& b_2 & b_3 \\
	a_0 &&& a_3 & b_1 &&& b_4 \\
	& a_5 & a_4 &&& b_0 & b_5
	\arrow["F", from=1-2, to=1-3]
	\arrow["V"', curve={height=12pt}, from=1-2, to=2-1]
	\arrow["F", curve={height=-12pt}, from=1-3, to=2-4]
	\arrow["F", from=1-6, to=1-7]
	\arrow["V"', curve={height=12pt}, from=1-6, to=2-5]
	\arrow["V"', curve={height=12pt}, from=2-8, to=1-7]
	\arrow["V"', curve={height=12pt}, from=3-3, to=2-4]
	\arrow["V"', curve={height=12pt}, from=3-7, to=2-8]
	\arrow["F", curve={height=-12pt}, from=3-2, to=2-1]
	\arrow["F", from=3-3, to=3-2]
	\arrow["V"', curve={height=12pt}, from=2-5, to=3-6]
	\arrow["F", from=3-7, to=3-6]
\end{tikzcd}\] 
\end{enumerate} 

\end{example}

\subsection{A structure theorem for unitary \texorpdfstring{$\boldsymbol{\bt}$}{} modules}\label{subsec:indecompuni}

The goal of this section is to prove a structure theorem for unitary $\bt$ modules, akin to \cref{prop: Kraft,prop: Oort}. More precisely, we will prove that unitary $\bt$ modules decompose as products of unitary unicycles, unitary bicycles, Serre unicycles, and Serre bicycles. See \cref{thm: structure of unitary BT1s} for the precise statement.

Recall that \cite[(4.9)]{Moonengsas} gives a basis $B=B_0 \cup B_1$ that is preserved by $F$ and $V$ and such that $M_i = \text{span}_{\kk} B_i$ for $i\in \{0,1\}$. This gives a decomposition of $\bt$ modules
\begin{equation} \label{eq: decomp B basis}
(M,F,V)\cong \bigoplus_{w\in W} \overline{M}(w),
\end{equation}
where each $\overline{M}(w)$ is spanned by elements of $B$. Note that the words $w$ in this decomposition need not be primitive and the decomposition need not be unique. Observe that all words $w\in W$ must have even length, since the elements of $B_0$ and $B_1$ alternate in the Kraft diagram. 

\begin{lemma} \label{lem: shorten words}
Let $(M,F,V,\lambda,M=M_0 \oplus M_1)$ decompose as in Equation~\eqref{eq: decomp B basis}. Then we may assume the words $w\in W$ are either primitive or of the form $w=zz$, where $z$ is primitive of odd length.
\end{lemma}
\begin{proof}
Consider the isomorphism $\phi:\overline{M}(w^n) \overset{\sim}{\to} \overline{M}(w)^n$ as is detailed in \cref{rmk: M((w)^n) = M(w)^n}. When $w$ has even length, $\phi$ preserves the summands $M_0$ and $M_1$. Therefore it provides bases $B'_i$ of $M_i$ such that $B'=B'_0 \cup B'_1$ is preserved by $F$ and $V$. As for the pairing, $\phi$ maps $\lambda$ to the pairing $\lambda_{\phi}(-,-)=\lambda(\phi^{-1}(-) , \phi^{-1}(-))$. Since $\phi$ is an isomorphism and $\lambda$ is a symplectic pairing, $\lambda_{\phi}$ is a symplectic pairing. Furthermore, since $\phi$ preserves $M_0$ and $M_1$, they are Lagrangian with respect to $\lambda_{\phi}$. We conclude that we can apply $\phi$ without affecting the unitary structure.

Repeatedly applying isomorphisms like $\phi$ to $w\in W$ with even length eventually leads to a basis $B'$ and decomposition as in Equation~\eqref{eq: decomp B basis} where every word $w$ is either primitive or of the form $w=zz$, where $z$ is primitive of odd length. In the latter case, applying $\phi$ would not preserve the unitary structure and hence we cannot simplify further. \end{proof}

We assume from here on out that all the words $w$ in Equation~\eqref{eq: decomp B basis} are either primitive or of the form $w=zz$, where $z$ is primitive of odd length. 

\begin{lemma} \label{lem: M(w) and M(w-c) both}
Let $w$ be a word such that $\overline{M}(w)$ does not admit unitary structure and suppose that $\overline{M}(w)$ occurs in Equation~\eqref{eq: decomp B basis} with multiplicity $m$. Then $\overline{M}(w^{\star})$ also occurs with multiplicity $m$.
\end{lemma}
\begin{proof}
Since $\overline{M}(w)$ does not admit unitary structure, it must be that any pairing $\lambda$ that makes $(M,F,V,\lambda,M=M_0 \oplus M_1)$ a unitary $\bt$ module is degenerate when restricted to $\overline{M}(w)$. Let $\lambda$ be any such pairing and let $e_{i,j} \in \overline{M}(w)$ be an element that pairs trivially with all of $\overline{M}(w)$. Let $e_{i',j'} \in M$ be such that $\lambda(e_{i,j},e_{i',j'})\neq 0$. Then it follows from the condition $\lambda(F(x),y) = \lambda(x,V(y))^p$ that $e_{i',j'}$ must lie on a copy of $\overline{M}(w^{\star})$. 
\end{proof}

\begin{remark}
If $w$ is non-self-dual, then \cref{lem: M(w) and M(w-c) both} follows immediately from \cref{prop: Oort}. However, \cref{lem: M(w) and M(w-c) both} also covers the case when $w=yy^{\star}$ for $y$ of even length, so that $\overline{M}(w)$ is polarizable but does not admit unitary structure.
\end{remark}

Observe that the summands $\overline{M}(zz)$ in Equation~\eqref{eq: decomp B basis} are not polarizable, since $z$ has odd length and is therefore non-self-dual. Moreover, \cref{lem: unitary M(w)} shows that, for $w$ primitive, $\overline{M}(w)$ admits unitary structure exactly when $w=yy^\star$ with $\ell(y)$ odd.

We conclude that there is a decomposition of $\bt$ modules
\begin{equation} \label{eq: B decomposition 2}
    (M,F,V) \cong \bigoplus_{w\in W'_1} \overline{M}(w) \oplus \bigoplus_{w\in W'_2} \left(\overline{M}(w) \oplus \overline{M}(w^{\star}) \right),
\end{equation}
where each $\overline{M}(w)$ is spanned by elements of $B$. Here $W'_1$ is a multiset of primitive words $w=yy^\star$ with $\ell(y)$ odd, and $W'_2$ is a multiset of words that either are primitive, but not of the form $w=yy^\star$ with $\ell(y)$ odd, or are of the form $w=zz$, with $z$ primitive of odd length.

A priori, Equation~\eqref{eq: B decomposition 2} only gives a decomposition of $\bt$ modules and does not take the pairing into account. When $(M,F,V,\lambda,M=M_0\oplus M_1)$ is a unitary $\bt$ module, the pairing $\lambda$ is unique up to an isomorphism of unitary $\bt$ modules (see \cite[Theorem~6.7]{Moonengsas} or \cite[Theorem~(9.4)]{OortStrat}). Thus, it suffices to exhibit one suitable pairing.

\begin{lemma} \label{lem: construct pairing}
Let $(M,F,V,\lambda,M=M_0 \oplus M_1)$ decompose as in Equation~\eqref{eq: B decomposition 2}. Then we may assume $\lambda$ is of the following form:
\begin{enumerate}
    \item For $w\in W'_1$, the pairing $\lambda$ on $\overline{M}(w)$ is as described in \cref{lem: pairing unicycle};
    \item For $w\in W'_2$, the pairing $\lambda$ on $\overline{M}(w) \oplus \overline{M}(w^{\star})$ is as described in \cref{lem: pairing bicycle}.
\end{enumerate}

\end{lemma}
\begin{proof}
Since the pairing is unique up to isomorphism of unitary $\bt$ modules, it suffices to show that $\lambda$ can be of the given form. It follows from \cref{lem: pairing unicycle,lem: pairing bicycle} that $(M,F,V,\lambda)$ is a polarized $\bt$ module. Moreover, \cref{lem: unitary M(w),lem: unitary M(w)+M(w)-c} provide that $M_0$ and $M_1$ are Lagrangian with respect to $\lambda$.
\end{proof}

We can now state and prove our structure theorem for unitary $\bt$ modules.

\begin{theorem} \label{thm: structure of unitary BT1s}
Any unitary $\bt$ module $\widetilde{M}=(M,F,V,\lambda,M=M_0\oplus M_1)$ decomposes as a product of unitary unicycles, unitary bicycles, Serre unicycles, and Serre bicycles. Formally, we have
\begin{equation*} \label{eq: unitray structure thm}
\widetilde{M} \cong \bigoplus_{w \in W_1} \widetilde{M}(w) \oplus \bigoplus_{w\in W_2} \widetilde{M}(w,w^{\star}) \oplus \bigoplus_{w\in W_3} \widetilde{M}(w,w^{\star}) \oplus \bigoplus_{w \in W_4} \widetilde{M}(w,w^{\star}),
\end{equation*}
where 
\begin{enumerate}[label=$\mathrlap{\oplus}\otimes$]
    \item $W_1$ is a multiset of primitive self-dual words $w=yy^{\star}$ where $y$ has odd length (unitary unicycles);
    \item $W_2$ is a multiset of primitive non-self-dual words of even length (unitary bicycles);
    \item $W_3$ is a multiset of primitive self-dual words $w=yy^{\star}$ where $y$ has even length (Serre unicycles);
    \item $W_4$ is a multiset of words $w=zz$ where $z$ is primitive of odd length (Serre bicycles).
\end{enumerate}
The decomposition is unique up to the action of $\tau^2$ on words in $W_1$ and $W_2$, the action of $w \mapsto \tau(w)^{\star}$ on words in $W_2$, the action of $\tau$ on $W_3$ and $W_4$, and the action of $w \mapsto w^{\star}$ on $W_4$.
\end{theorem}
\begin{proof}
Applying \cref{lem: construct pairing} to Equation~\eqref{eq: B decomposition 2} gives the decomposition
$$\widetilde{M} \cong \bigoplus_{w\in W'_1} \widetilde{M}(w) \oplus \bigoplus_{w\in W'_2} \widetilde{M}(w,w^{\star}).$$
By definition $W_1'=W_1$. On the other hand, $W'_2$ is a multiset of words that are either primitive, but not of the form $w=yy^\star$ with $\ell(y)$ odd, or of the form $w=zz$, with $z$ primitive of odd length.
It is clear that the case $w=zz$ corresponds to the multiset $W_4$. Furthermore, if $w\in W_2'$ is primitive and self-dual, it must be of the form $w=yy^{\star}$ by \cref{lem: self-dual word} and $\ell(y)$ must be even, so that $w\in W_3$. The remaining $w\in W_2'$ lie in $W_2$.

For the uniqueness statement, see \cref{rmk: don't rotate unicycles!} for $W_1$ and \cref{rmk: don't rotate bicycles!} for $W_2$, $W_3$, and $W_4$.
\end{proof}

\begin{corollary} \label{cor: indecomposables}
If $\widetilde{M}$ is indecomposable as a unitary $\bt$ module, then $\widetilde{M}$ is isomorphic to a unitary unicycle, a unitary bicycle, a Serre unicycle, or a Serre bicycle.
\end{corollary}

\begin{remark}\label{rmk: sign-fix-indecomp}
    Consider integers $q > 0$ and $0 \leq b \leq q$. \cref{thm: structure of unitary BT1s} allows us to list the indecomposable types occurring for a given signature $(q-b,b)$. This depends on the parity of $q$. 

Case I. \emph{$q$ is odd}.

In this case, the indecomposable unitary $\bt$ modules of signature $(q-b,b)$ are of the form $\widetilde{M}(w)$, where $w$ is a primitive self-dual word of length $2q$ in which $\word{v}$ occurs in $b$ places with even index. In other words, when $q$ is odd, the only indecomposable type possible is that of a unitary unicycle.

Case II. \emph{$q$ is even}.

In this case, the indecomposable unitary $\bt$ modules of signature $(q-b,b)$ are of the form $\widetilde{M}(w,w^\star)$ with $b$ equaling the sum of the number of $\word{v}$'s in even places and the number of $\word{f}$'s in odd places of $w$. There are three indecomposable types for $q$ even. 

For any even $q>2$, the indecomposable type of a unitary bicycle occurs. In this case, the word $w$ is a non-self-dual-primitive word of even length. The other two indecomposable types occur in the case of parallel signature $(q/2,q/2)$, depending on the parity of $q/2$.
    \begin{enumerate}[label=$\rhd$,leftmargin=1.75em]
        \item When $q\equiv 0 \bmod{4}$, the indecomposable type of a Serre unicycle occurs. In this case, the word $w$ is a self-dual-primitive word whose length is divisible by $4$.
        \item When $q\equiv 2 \bmod{4}$, the indecomposable type of a Serre bicycle occurs. In this case, the word $w=zz$ for a primitive word $z$ of odd length. 
    \end{enumerate}

    When $q = 2$, indecomposables only appear in signature $(1,1)$, where there is a unique one with the indecomposable type of a Serre bicycle.
    
    See \cref{example:unicycle diagram,example: bicycles} for specific examples of each indecomposable type.

    In \cref{appendix:indecomp-count}, we give the number of unitary unicycles, unitary bicycles, Serre unicycles and Serre bicycles that occur with signature $(q-b,b)$.
\end{remark}

\subsection{Epilogue: The Serre Tensor Construction}\label{section:STC}

In this section, we discuss the Serre tensor construction of an abelian variety (and of a $p$-divisible group) and its relation to $\bt$ modules. For more details, see \cite{stc-amir,cm-lifting,conrad-gross-zagier,Serre2002Appendix}.

Consider a polarized abelian variety $(A,\lambda)$ over $\kk$, and let $K$ be an imaginary quadratic field such that $p$ is inert in $K$. Then, the \emph{Serre tensor construction} of $A$, denoted $\mathcal{O}_K \otimes_{\zz} A$, is an abelian variety over $\kk$ with a natural action $\iota$ of $\mathcal{O}_K$; see \cite[Appendix]{Serre2002Appendix}, \cite[Section 1]{stc-amir} or \cite[Proposition~1.7.4.5]{cm-lifting}.

Let $h$ be the trace pairing $(\alpha,\beta) \mapsto \mathrm{Tr}(\alpha\overline{\beta})$, which is a positive definite hermitian form on $\mathcal{O}_K$. Then \cite[Theorem A]{stc-amir} allows us to conclude that $(\mathcal{O}_K \otimes_{\zz} A,h \otimes \lambda,\xi)$ is a polarized abelian variety with an action $\iota$ of $\mathcal{O}_K$. Additionally, by \cite[Corollary 7.9]{conrad-gross-zagier}, we have an isomorphism of $p$-torsion group schemes $\ff_{p^2} \otimes_{\ff_p} A[p] \cong (\mathcal{O}_K \otimes_{\zz} A)[p]$. 

Let $\widehat{D} = (D,F,V,\lambda)$ be the polarized $\bt$ module corresponding to $A[p]$ under the \Dd equivalence. Then the isomorphism $\ff_{p^2} \otimes_{\ff_p} A[p] \cong (\mathcal{O}_K \otimes_{\zz} A)[p]$, under the \Dd{} equivalence,  produces the following unitary \Dd{} module (in the sense of \cref{rmk: iota}): 
\[\ff_{p^2} \otimes_{\ff_p} \widehat{D} \coloneqq (\ff_{p^2} \otimes_{\ff_p} D,\id \otimes\, F, \id \otimes\, V,\widetilde{\lambda},\iota),\] where
\begin{enumerate}
    \item $\iota$ is the natural action of $\ff_{p^2}$ on $\ff_{p^2} \otimes_{\ff_p} D$; and
    \item $\widetilde{\lambda} \coloneqq \widetilde{\mathrm{Tr}} \otimes \lambda$
    is a pairing on $\ff_{p^2} \otimes_{\ff_p} D$, where $\widetilde{\mathrm{Tr}}$ is the trace pairing on $\ff_{p^2}$ defined as
    \[\widetilde{\mathrm{Tr}}(\alpha,\beta) = \mathrm{Tr}(\alpha\beta^p).\]
\end{enumerate}

\begin{definition}\label{def:stc-bt1}
    Consider $\widehat{N} = (N,F,V,\lambda)$, a polarized $\bt$ module. Let $\phi_0,\phi_1: \ff_{p^2} \hookrightarrow \kk$ be the canonical $\ff_p$-linear embeddings. For $i = 0,1$, let $N_i$ be the $\kk$-vector space $N$ equipped with an action of $\ff_{p^2}$ via $\phi_i$.

    Define operators $F_{\mathsf{STC}}$ and $V_{\mathsf{STC}}$ on $\mathsf{STC}(N) \coloneqq N_0 \oplus N_1$ as follows:
    \[F_{\mathsf{STC}}(n_0,n_1) = (F(n_1),F(n_0)) \quad \text{and} \quad V_{\mathsf{STC}}(n_0,n_1) = (V(n_1),V(n_0)),\]
    define an action $\iota_{\mathsf{STC}}$ of $\ff_{p^2}$ on $N_0 \oplus N_1$ as follows: 
\[\iota_{\mathsf{STC}}: \ff_{p^2} \to \mathrm{End}_{\kk}(N_0 \oplus N_1),\; \alpha \mapsto \left[(n_0,n_1) \mapsto (\phi_0(\alpha)n_0,\phi_1(\alpha) n_1)\right],\]
and finally define a pairing $\lambda_{\mathsf{STC}}$ on $N_0 \oplus N_1$ as follows:
\[\lambda_{\mathsf{STC}}: (N_0 \oplus N_1) \times (N_0 \oplus N_1) \to \kk,\; ((n_0,n_1),(n_0',n_1')) \mapsto \lambda(n_0,n_1') - \lambda(n_0',n_1).\]

The quintuple \[\mathsf{STC}(\widehat{N}) \coloneqq (\mathsf{STC}(N),F_{\mathsf{STC}},V_{\mathsf{STC}},\lambda_{\mathsf{STC}},\mathsf{STC}(N) = N_0 \oplus N_1)\] is called the \emph{Serre tensor construction} of $\widehat{N}$. 
\end{definition}

\begin{proposition}
    Consider a $\bt$ module $\widehat{N} = (N,F,V,\lambda)$. Then its Serre tensor construction $\mathsf{STC}(\widehat{N})$ is a unitary $\bt$ module of signature $(q,q)$, where $q = \dim_{\kk} N$.
\end{proposition}
\begin{proof}
    This is clear from the definition.
\end{proof}

The following proposition justifies calling $\mathsf{STC}(\widehat{N})$ a Serre tensor construction.
\begin{proposition}\label{prop:stc-bt1-const}
    Consider a $\bt$ module $\widehat{N} = (N,F,V,\lambda)$ and its Serre tensor construction $\mathsf{STC}(\widehat{N})$. Given the quintuple $\ff_{p^2} \otimes_{\ff_p} \widehat{N} \coloneqq (\ff_{p^2} \otimes_{\ff_p} N,\id \otimes\, F, \id \otimes\, V,\widetilde{\lambda},\iota)$ , where
$\iota$ is the natural action of $\ff_{p^2}$ on $\ff_{p^2} \otimes_{\ff_p} N$ and $\widetilde{\lambda} = \widetilde{\mathrm{Tr}} \otimes \lambda$, the following isomorphism of $\kk$-vector spaces
\[\ff_{p^2} \otimes_{\ff_p} N \to N_0 \oplus N_1: \alpha \otimes n \mapsto (\phi_0(\alpha)n, \phi_1(\alpha)n)\]
commutes with the given additional structures on $\ff_{p^2} \otimes_{\ff_p} N$ and $N_0 \oplus N_1$ respectively.
\end{proposition}
\begin{proof}
    This may be verified directly, using \cref{rmk: iota} to translate between the $\iota$ and the $\oplus$ descriptions. 
\end{proof}

The underlying polarized $\bt$ module of a Serre tensor construction has a simple description.

\begin{lemma}\label{lemma: stc-inj}
For a polarized $\bt$ module $\widehat{N} = (N,F,V,\lambda)$, we have an isomorphism of polarized $\bt$ modules for the underlying polarized $\bt$ module of its Serre tensor construction $\STC(\widehat{N})$,
\[(\STC(N),F_{\STC},V_{\STC},\lambda_{\STC}) \cong \widehat{N} \oplus \widehat{N}\]
Moreover, if $\widehat{N}'$ is another polarized $\bt$ module such that there is an isomorphism $\STC(\widehat{N}) \cong \STC(\widehat{N}')$ of unitary $\bt$ modules, then $\widehat{N} \cong \widehat{N}'$ as polarized $\bt$ modules. That is, the Serre tensor construction is injective on isomorphism classes of polarized $\bt$ modules.
\end{lemma}
\begin{proof}
By \cref{prop:stc-bt1-const}, we work instead with $\ff_{p^2} \otimes_{\ff_p} \widehat{N}$. Since one has $\ff_{p^2} \cong \ff_p \oplus \ff_p$, therefore
\[\ff_{p^2} \otimes_{\ff_p} N \cong (\ff_p \oplus \ff_p) \otimes_{\ff_p} N \cong (\ff_p \otimes_{\ff_p} N) \oplus (\ff_p \otimes_{\ff_p} N) \cong N \oplus N.\]
Necessarily, under this isomorphism, $\id \otimes\, F$ and $\id \otimes\, V$ correspond to $F \oplus F$ and $V \oplus V$. Hence, the isomorphism above is an isomorphism of $\bt$ modules $(\ff_{p^2} \otimes_{\ff_p} N,\id \otimes\, F, \id \otimes\, V) \cong (N \oplus N, F\oplus F, V\oplus V)$. 

Since the polarization is unique, up to isomorphism (see \cite[Theorem~6.7]{Moonengsas} or \cite[Theorem~(9.4)]{OortStrat}), we conclude that we have an isomorphism of polarized $\bt$ modules
\[(\STC(N),F_{\STC},V_{\STC},\lambda_{\STC}) \cong (\ff_{p^2} \otimes_{\ff_p} N,\id \otimes\, F,\id \otimes\, V,\widetilde{\lambda}) \cong \widehat{N} \oplus \widehat{N}\]

Consider now a polarized $\bt$ module $\widehat{N}'$ such that there is an isomorphism $\STC(\widehat{N}) \cong \STC(\widehat{N}')$ of unitary $\bt$ modules. By our arguments above, we obtain an isomorphism $\widehat{N} \oplus \widehat{N} \cong \widehat{N}' \oplus \widehat{N}'$. Relying on the uniqueness statement in \cref{prop: Oort}, we conclude $\widehat{N} \cong \widehat{N}'$ as polarized $\bt$ modules.
\end{proof}

The following two propositions explain the terminology \emph{Serre unicycle} and \emph{Serre bicycle}, respectively.

\begin{proposition} \label{prop: STC unicycle}
Let $w$ be a primitive, self-dual word, of length divisible by $4$. Then $\mathsf{STC}(\widehat{M}(w)) \cong \widetilde{M}(w,w^\star)$. That is, Serre unicycles arise as Serre tensor constructions of unicycles.
\end{proposition}
\begin{proof}
We apply \cref{def:stc-bt1} to the polarized $\bt$ module $\widehat{N}=\widehat{M}(w)$. The word $w$ can be written as 
\[w=u_{2q-1} \cdots u_0=\tau^q(w^\star)=u_{q-1}^\star \cdots u_0^\star u_{2q-1}^\star \cdots u_q^\star\]
with $2q = \ell(w)$ and $q$ even. Given $0\leq j \leq 2q-1$, we let $0\leq \hat{\jmath} \leq 2q-1$ be the integer satisfying $\hat{\jmath} \equiv j+q \pmod{2q}$, i.e., $\hat{\jmath}$ is either $j+q$ or $j-q$. 

Following \cref{def: unicycle}, we have $\widehat{N}=(N,F,V,\lambda)$, where $N=\mathrm{span}_\kk \{a_0, \ldots, a_{2q-1}\}$,
\[F(a_{j}) = \begin{cases} a_{j+1} &\hbox{if $u_j=\word{f},$} \\ 0 &\hbox{if $u_j=\word{v},$} \end{cases} \hspace{2cm} V(a_{j+1}) = \begin{cases} a_{j} &\hbox{if $u_j=\word{v},$} \\ 0 &\hbox{if $u_j=\word{f}$,} \end{cases} 
\]
and $\lambda$ is as in \cref{lem: pairing unicycle}.

Let $\{i, \hat{\imath}\} = \{0,1\}$. Now, applying \cref{def:stc-bt1}, we begin by formally defining the $\kk$-vector spaces $N_i=\mathrm{span}_\kk \{n_{i,0}, \ldots ,n_{i,2q-1}\}$ and set $\mathsf{STC}(N)=N_0\oplus N_1$, the underlying vector space of $\mathsf{STC}(\widehat{N})$.

By \cref{def:stc-bt1}, the actions of $F_{\mathsf{STC}}$ and $V_{\mathsf{STC}}$ are: 
\[
F_{\mathsf{STC}}(n_{i,j}) = \begin{cases} n_{\hat{\imath},j+1} &\hbox{if $u_j=\word{f},$} \\ 0 &\hbox{if $u_j=\word{v},$} \end{cases} \hspace{2cm} V_{\mathsf{STC}}(n_{i,j+1}) = \begin{cases} n_{\hat{\imath},j} &\hbox{if $u_j=\word{v},$} \\ 0 &\hbox{if $u_j=\word{f}$.} \end{cases}
\]

Together with the polarization $\lambda_{\STC}$, this construction yields the unitary $\bt$ module 
\[
\mathsf{STC}(\widehat{M}(w)) = (\mathsf{STC}(N), F_{\mathsf{STC}}, V_{\mathsf{STC}}, \lambda_{\mathsf{STC}}, \mathsf{STC}(N)=N_0\oplus N_1).  
\]

On the other hand, consider the construction $\widetilde{M}(w,w^\star)=(M,\widetilde{F},\widetilde{V},\widetilde{\lambda},M=M_0\oplus M_1)$ from Definition \ref{def: widetilde M(w,wc)} (see also \cref{lem: unitary M(w)+M(w)-c}). The underlying vector space is given by
\[
M \coloneqq \mathrm{span}_\kk \{a_0, \ldots, a_{2q-1},b_0, \ldots, b_{2q-1} \}.
\]
The actions of the morphisms $\widetilde{F}$ and $\widetilde{V}$ on $\{a_0, \ldots , a_{2q-1}\}$ are the same as the actions of $F$ and $V$ on $N=\mathrm{span}_\kk \{a_0, \ldots , a_{2q-1}\}$, whereas their actions on $\{b_0,\ldots ,b_{2q-1}\}$ are given by
\[
\widetilde{F}(b_j) = \begin{cases} b_{j+1} &\hbox{if $u_j^\star=\word{f},$} \\ 0 &\hbox{if $u_j^\star=\word{v},$} \end{cases} \hspace{2cm} \widetilde{V}(b_{j+1}) = \begin{cases} b_{j} &\hbox{if $u_j^\star=\word{v},$} \\ 0 &\hbox{if $u_j^\star=\word{f}$.} \end{cases}
\]
Note that $u_j^\star=u_{\hat{\jmath}}$ since $w$ is self-dual and primitive. Following \cref{lem: unitary M(w)+M(w)-c}, we define subspaces
\begin{align*}
    M_0&=\text{span}_{\kk} \{a_j\mid j \text{ even}\} \oplus \text{span}_{\kk} \{b_j \mid j \text{ odd}\}, \\
    M_1&=\text{span}_{\kk} \{a_j \mid j \text{ odd}\} \oplus \text{span}_{\kk} \{b_j \mid j \text{ even}\}.
\end{align*}
Together with the polarization $\widetilde{\lambda}$ from \cref{lem: pairing bicycle}, this construction yields the unitary $\bt$ module 
\[
\widetilde{M}(w)=(M, \widetilde{F}, \widetilde{V},\widetilde{\lambda}, M=M_0 \oplus M_1).
\]

We now exhibit an isomorphism between the constructions $\mathsf{STC}(\widehat{M}(w))$ and $\widetilde{M}(w)$. For this, we define the following isomorphism of the underlying vector spaces:
\begin{align*}
    \psi: \mathsf{STC}(N) &\to M \\
    n_{i,j} &\mapsto \begin{cases} 
    a_j &\hbox{if $i+j$ is even,} \\
    b_{\hat{\jmath}} &\hbox{if $i+j$ is odd}.
    \end{cases}
\end{align*}

One then verifies, using $u_j^\star=u_{\hat{\jmath}}$, that $\psi \circ F_\STC = \widetilde{F} \circ \psi$ and $\psi \circ V_\STC = \widetilde{V} \circ \psi$.

Hence, $\psi$ gives rise to an isomorphism between $(\STC(N),F_{\STC},V_\STC)$ and $(M,\widetilde{F},\widetilde{V})$. We observe that $\psi(N_i)=M_i$ for $i\in \{0,1\}$ since $q$ is even.

Since the polarization is unique, up to isomorphism (see \cite[Theorem~6.7]{Moonengsas} or \cite[Theorem~(9.4)]{OortStrat}), we conclude that the unitary $\bt$ modules $\STC(\widehat{M}(w))$ and $\widetilde{M}(w,w^\star)$ are isomorphic. 
\end{proof}

\begin{proposition} \label{prop: STC bicycle}
Let $z$ be a primitive word of odd length. Then $\STC(\widehat{M}(z,z^\star)) \cong \widetilde{M}(zz,z^\star z^\star)$. That is, Serre bicycles arise as Serre tensor constructions of bicycles.
\end{proposition}

\begin{proof}
This can be verified with the proof method of \cref{prop: STC unicycle}.
\end{proof}

\begin{remark}
    We contextualize the results of this section, specifically the indecomposable types introduced so far, in terms of the forgetful functors
    \begin{align*}
        \widetilde{N} = (N,F,V,\lambda,M = M_0 \oplus M_1) &\mapsto (N,F,V,\lambda) \eqcolon \widehat{U}(\widetilde{N})\\
        \widehat{L} = (L,F,V,\lambda) &\mapsto (L,F,V) \eqcolon \overline{U}(\widehat{L})
    \end{align*}
which, respectively, takes a unitary $\bt$ module $\widetilde{N}$ to its underlying polarized $\bt$ module $\widehat{U}(\widetilde{N})$, and takes a polarized $\bt$ module $\widehat{L}$ to its underlying (undecorated) $\bt$ module $\overline{U}(\widehat{L})$.

Let the composition of these functor still be denoted as $\overline{U}$, i.e., the underlying (undecorated) $\bt$ module of a unitary $\bt$ module $\widetilde{N}$ is denoted as $\overline{U}(\widetilde{N})$.

Given an indecomposable unitary $\bt$ module $\widetilde{N}$, either the polarized $\bt$ module $\widehat{U}(\widetilde{N})$ remains indecomposable, or it is decomposable.

When $\widehat{U}(\widetilde{N})$ is indecomposable, either $\overline{U}(\widetilde{N})$ is also indecomposable, in which case $\widehat{U}(\widetilde{N})$ is a \emph{unicycle} and $\widetilde{N}$ is a \emph{unitary unicycle}, or $\overline{U}(\widetilde{N})$ is decomposable, in which case $\widehat{U}(\widetilde{N})$ is a \emph{bicycle} and $\widetilde{N}$ is a \emph{unitary bicycle}.

When $\widehat{U}(\widetilde{N})$ is decomposable, it follows from \cref{thm: structure of unitary BT1s,prop: STC unicycle,prop: STC bicycle} that $\widetilde{N} \cong \STC(\widehat{L})$ for some indecomposable polarized $\bt$ module $\widehat{L}$. Either $\overline{U}(\widehat{L})$ is also indecomposable, in which case $\widetilde{N}$ is a \emph{Serre unicycle}, or $\overline{U}(\widehat{L})$ is decomposable, in which case $\widetilde{N}$ is a \emph{Serre bicycle}.
\end{remark}

\section{Words and Weyl group cosets} \label{sec: Weyl}

In \cref{thm: structure of unitary BT1s}, Ekedahl--Oort strata of $\M(a,b)$ are described in terms of their associated words in the alphabet $\{\word{f},\word{v}\}$. In this section, we provide an algorithm, much like a dictionary, that describes these Ekedahl--Oort strata in terms of Weyl group cosets, as was first outlined in \cite{Moonengsas}. While the word construction is a natural way to describe and understand the decompositions of $\bt$ modules, the Weyl group coset characterization has the benefit of making geometric properties of the strata, such as the dimension and the topological closure, more apparent.

\subsection{Preliminaries} \label{sec: Weyl intro}

Let $q=a+b$. Recall from \cref{sec:background Weyl} that $\mathbf{W}(a,b)$ denotes the set of minimal length representatives of the Weyl group cosets in $\mathbf{W}_{(a,b)} \setminus \mathfrak{S}_{q}$. 

By the work of \cite[(3.5)]{Moonengsas} and \cite[Section 3]{EOstrata1}, there is a procedure to compute the $\gamma\in \mathbf{W}(a,b)$, corresponding to a unitary $\bt$ module $\widetilde{M}=(M,F,V,\lambda,M=M_0\oplus M_1)$ of signature $(a,b)$:

\begin{enumerate}[label=$\bullet$]
    \item Let $W_\bullet$ be a final filtration of $M$, i.e. a filtration stable under $F$ and $V^{-1}$ that contains a space of each dimension between $1$ and $2(a+b)=2q$;
    \item Define $C_{\bullet} = W_{\bullet} \cap M_0$ and define, for $s \in \{1, \ldots, q\}$, the function $\eta(s) \coloneqq \dim_k(C_{1,s}[F])$;
    \item Let $z_1< \cdots < z_b$ be the integers where $\eta$ jumps, i.e. $\eta(z_m)=\eta(z_m-1)+1$ for every $1\leq m \leq b$; 
    \item Then $\gamma_{z_1,\ldots ,z_b} \coloneqq (b, b+1, \ldots, z_b) \cdots (2,3,\ldots ,z_2)(1,2,\ldots ,z_1) \in \mathbf{W}(a,b)$ corresponds to $\widetilde{M}$.
\end{enumerate}
The goal of \cref{sec: word Weyl} is to describe the indices $z_1 < \cdots < z_b$ in terms of the words in the alphabet $\{\word{f,v}\}$ that are used in \cref{thm: structure of unitary BT1s}.

\subsection{The ``Word to Weyl'' algorithm} \label{sec: word Weyl}

Let $\widetilde{M}=(M,F,V,\lambda,M_0 \oplus M_1)$ be a unitary $\bt$ module of signature $(a,b)$ and let $\overline{M}=(M,F,V)$ be its underlying $\bt$ module. 
 
For $i\in \{0,1\}$, let $B_i=\{e_{i,j} \; | \; 1 \leq j \leq q\}$ be a basis of $M_i$ such that $F$ and $V$ preserve $B=B_0 \cup B_1$.

By \cref{lem: shorten words}, there is a decomposition
$$\overline{M} = \bigoplus_{w \in W} \overline{M}(w)$$
with the following properties:
\begin{enumerate}
    \item Every factor $\overline{M}(w)$ is spanned by elements of $B$ (which implies that every $w\in W$ must have even length); \label{property B}
    \item Every word $w\in W$ is either primitive or of the form $w=zz$, where $z$ is primitive of odd length. \label{property prim}
\end{enumerate}
In addition, we may and do assume, after rotating words if necessary, the following condition: 
\begin{enumerate} \setcounter{enumi}{2}
    \item For each $w=u_{\ell-1} \cdots u_0 \in W$, the letter $u_0$ signifies either $F(e_{0,j})=e_{1,j'}$ or $V(e_{1,j'})=e_{0,j}$ for some $j,j'$. \label{property 01} 
\end{enumerate}

Under these assumptions, we associate a permutation $\gamma(W) \in \mathbf{W}(a,b)$ to the multiset $W$. Recall the rotation operator $\tau$ defined by $\tau(u_{\ell-1} \cdots u_0) = u_0 u_{\ell-1} \cdots u_1$. We denote by $\preceq$ the lexicographical order that can compare any two words of equal length in the alphabet $\{\word{f},\word{v}\}$ using the rule $\word{f} \preceq \word{v}$. Note that when we order the words lexicographically, we work left-to-right, but when we are reading the words as maps, we apply the functions right-to-left. 

\begin{definition} \label{def: gamma(W)}
Let $W$ be a multiset of words with $L_W  \coloneqq  \text{LCM}_{w\in W}\{\ell(w)\}$. Define the multiset $$\widetilde{W} \coloneqq  \left\{\left(\tau^{2j}(w)\right) ^{L_W/\ell(w)} \; | \; w \in W, 0\leq j < \ell(w)/2 \right\}$$
and set $q_W \coloneqq \# \widetilde{W}$. Let $\widetilde{\omega}_1 \preceq \cdots \preceq \widetilde{\omega}_{q_W}$ be a lexicographical ordering of the elements (with multiplicity) 
of $\widetilde{W}$. Then define
\begin{align*}
    \setp{z_1, \ldots ,z_b}{z_1 <\cdots < z_b}
    & \coloneqq \setp{s}{\widetilde{\omega}_s \text{ has $\word{v}$ as its right-most letter}}; \text{ and} \\
    \gamma(W)& \coloneqq  \gamma_{z_1,\ldots, z_b} = (b, b+1, \ldots, z_b) \cdots (2,3,\ldots ,z_2)(1,2,\ldots ,z_1) \in \mathbf{W}(a,b).
\end{align*}
\end{definition}

\begin{remark} \label{rmk: gamma(W) canonical}
The multiset $\widetilde{W}$ records all the double rotations of words in $W$, with multiplicity. The words are raised to the power $L_W/\ell(w)$ to arrange that all elements of $\widetilde{W}$ are of equal length, so that the lexicographical order can compare any two distinct elements of $\widetilde{W}$. When $W$ contains multiple elements that are in the same $\tau^2$-orbit, then $\widetilde{W}$ is not a set and therefore the lexicographical ordering of $\widetilde{W}$ is not canonical. However, this amounts to the fact that we may permute words that are equal, so it is clear that the definition of $\gamma(W)$ is nonetheless canonical.

See \cref{eg: fv fffvvv} for the process of computing $\widetilde{W}$ and ordering the elements of this multiset.
\end{remark}

The following theorem is a ``word to Weyl'' algorithm, as it converts the word-based description from \cref{thm: structure of unitary BT1s} to the Weyl group coset representative of a given Ekedahl--Oort stratum.

\begin{theorem} \label{thm: word Weyl}
Let $\widetilde{M}$ be a unitary $\bt$ module of signature $(a,b)$ whose underlying $\bt$ module admits the decomposition $\overline{M} = \bigoplus _{w\in W} \overline{M}(w)$ satisfying Properties \eqref{property B}, \eqref{property prim} and \eqref{property 01} above. Let $\gamma(W)$ be as in \cref{def: gamma(W)}. Then $\gamma(W)$ is the minimal length representative of the coset of $(\mathfrak{S}_a \times \mathfrak{S}_b) \setminus \mathfrak{S}_q$, where $q=q_W$, that corresponds to the Ekedahl--Oort stratum of $\widetilde{M}$ via \cite[Theorem~6.7]{Moonengsas}.
\end{theorem}
\begin{proof}
We follow the steps outlined in \cref{sec: Weyl intro} for computing $z_1,\ldots ,z_b$. The first step is to form a final filtration of $M$, for which we use \cite[Section 3.5]{PriesUlmerBT1Fermat}. Let $n$ be the number of distinct words. Following their notation, we write
$$\overline{M} = \bigoplus_{i=1}^n \overline{M}(w_i)^{m_i},$$
where $\ell_i$ is the length of the word $w_i=u_{i,\ell_i-1} \ldots u_{i,0}$ 
and $m_i$ is its multiplicity.  
The vector space $M$ has a basis $\{e_{i,j,k} \; | \; 1\leq i \leq n, \; 0\leq j< \ell_i, 1\leq k \leq m_i \}$, on which $F$ and $V$ act as follows:
\begin{align*}
    F(e_{i,j,k})&= \begin{cases} e_{i,j+1,k} &\hbox{if $u_{i,j}=\word{f}$,} \\ 0 &\hbox{if $u_{i,j}=\word{v}$,} \end{cases} \hspace{2cm}
    V(e_{i,j+1,k})= \begin{cases} e_{i,j,k} &\hbox{if $u_{i,j}=\word{v}$,} \\ 0 &\hbox{if $u_{i,j}=\word{f}$}. \end{cases}
\end{align*}

Now, for $1\leq i \leq n$, define $w_{i,0} \coloneqq w_i^{L_W/\ell_i}$ and $$w_{i,j} \coloneqq (\tau^j(w_i))^{L_W/\ell_i}=(u_{i,j-1} \cdots u_{i,0}u_{i,\ell_i-1} \cdots u_{i,j})^{L_W/\ell_i}$$ 
for $0< j < \ell_i$. Note that, by construction, every $w_{i,j}$ has the length $L_W$. Let $\Omega$ be the set of all $w_{i,j}$. For each $\omega\in \Omega$, we now define 
$$N_{\omega}  \coloneqq  \text{span} \left\{ e_{i,j,k} \; | \; w_{i,j} \preceq \omega  \right\}.$$

By \cite[Section 3.5]{PriesUlmerBT1Fermat}, the canonical filtration of $M$ is given by $\{N_\omega \; | \; \omega \in \Omega\}$. From the construction it is clear that $N_\omega \subseteq N_{\omega'}$ if and only if $\omega\preceq \omega'$.

The next step in \cref{sec: Weyl intro} is to intersect the final filtration with $M_0$. It suffices to intersect the canonical filtration with $M_0$. To this end, note that Property~\eqref{property 01} above ensures that $e_{i,0,k} \in M_0$. Furthermore, since $F$ and $V$ interchange $M_0$ and $M_1$, we see that $e_{i,j,k}\in M_0$ if and only if $j$ is even. Hence,
$$N_\omega^0 \coloneqq N_{\omega} \cap M_0 = \text{span} \{e_{i,j,k} \; | \; w_{i,j} \preceq \omega \text{ and } 2\mid j\}.$$

Let $q=q_W$ from \cref{def: gamma(W)}. The function $\eta:\{1,\ldots , q\} \to \{1, \ldots ,q\}$ will now be determined by its values at the numbers $s_\omega: =\dim(N_\omega^0)$ since it is irrelevant how the canonical filtration is refined to a final filtration. For this, recall the elements $\widetilde{\omega}_1 \preceq \cdots \preceq \widetilde{\omega}_q$ of $\widetilde{W}$ in \cref{def: gamma(W)}. Note that these elements correspond to the elements $\omega_{i,j} \in \Omega$ with $j$ even. We then observe:
\begin{align*}
    \dim\left(\text{span} \{e_{i,j,k} \; | \; w_{i,j} = \omega  \text{ and }  2 \mid j\} \right) &= \# \{1\leq s \leq q \; | \; \widetilde{\omega}_s=\omega\},  \\
    s_\omega=\dim\left( N_\omega^0\right)&= \#\{1\leq s\leq q \; | \; \widetilde{\omega}_s \preceq \omega \}.
\end{align*}

Equivalently, $s_{\omega}=\max\{s \; | \; \widetilde{\omega}_s \preceq \omega\}$.

The last step is to compute the dimension of $N_\omega^0[F]$ for each $\omega\in \Omega$. For this, we note that $F(e_{i,j,k})=0$ is equivalent to $u_{i,j}=\word{v}$, which is in turn equivalent to $w_{i,j}$ ending with $\word{v}$. Thus we conclude that, for each $\omega\in \Omega$, we have
\begin{align*}
    \eta(s_\omega) = \dim(N_\omega^0[F]) &= \#\{1 \leq s \leq q \; | \; \widetilde{\omega}_s \preceq \omega \text{ and } \widetilde{\omega}_s \text{ ends with $\word{v}$}\} \\
    &= \#\{1\leq s \leq s_\omega \; | \; \widetilde{\omega}_s \text{ ends with $\word{v}$}\}.
\end{align*}
Finally, the property $\eta(s)\leq \eta(s+1)\leq \eta(s)+1$ dictates the further behavior of the function $\eta$:
$$\eta(s)=\#\{1\leq t \leq s \; | \; \widetilde{\omega}_t \text{ ends with $\word{v}$}\}.$$

Thus, $\eta$ jumps at the values $\setp{z_1, \ldots ,z_b}{z_1 <\cdots < z_b} 
= \{1\leq s \leq q \; | \; \widetilde{\omega}_s \text{ ends with $\word{v}$}\}.$ \end{proof}

\begin{remark}\label{remark:multiset W}
By \cref{thm: structure of unitary BT1s}, every unitary $\bt$ module decomposes as 
\[
\widetilde{M} = \bigoplus_{w\in W_1} \widetilde{M}(w) \oplus \bigoplus_{w\in W_2} \widetilde{M}(w,w^\star).
\]

We now outline how to derive the multiset $W$ used for \cref{thm: word Weyl} from this decomposition. From the definitions of $\widetilde{M}(w)$ and $\widetilde{M}(w,w^\star)$ we obtain a decomposition of the underlying $\bt$ module
\[
\overline{M} = \bigoplus_{w\in W_1} \overline{M}(w) \oplus \bigoplus_{w\in W_2} \left( \overline{M}(w) \oplus \overline{M}(w^\star) \right).
\]
However, appealing to the construction in \cref{lem: unitary M(w)+M(w)-c}, we observe that the summands $\overline{M}(w^\star)$ do not satisfy Property~\eqref{property 01} above. To this end, we replace them by $\overline{M}(\tau(w^\star))$, which do satisfy Property~\eqref{property 01}. We conclude that 
\[
W=W_1 \cup W_2 \cup \{\tau(w^\star) \; | \; w\in W_2\}.
\]
\end{remark}

\begin{example} \label{eg: fv fffvvv}
Consider the unitary $\bt$ module $\widetilde{M}(\word{vf}) \oplus \widetilde{M}(\word{fffvvv})$ with $W=\{\word{vf},\word{fffvvv}\}$. We compute $L_W=6$ and, after lexicographically ordering,
we obtain for the multiset $\widetilde{W}$ 
\begin{center}
    \begin{tabular}{cccc}
      $\underbracket[.75pt]{\,\word{fffvvv}\,}_{\Circled[inner ysep=6pt]{1}}$\;, & $\underbracket[.75pt]{\,\word{fvvvff}\,}_{\Circled[inner ysep=6pt]{2}}$\;, & $\underbracket[.75pt]{\,\word{vfvfvf}\,}_{\Circled[inner ysep=6pt]{3}}$\;, & $\underbracket[.75pt]{\,\word{vvfffv}\,}_{\Circled[inner ysep=6pt]{4}}$. \\
    \end{tabular}
\end{center}
By examining which elements above end in $\word{v}$, one sees $\gamma(W)=\gamma_{1,4}= (2,3,4)\in \mathbf{W}(2,2)$.

On the other hand, considering instead $\widetilde{M}(\word{fv}) \oplus \widetilde{M}(\word{ffvvvf})$ with $W=\{\word{fv},\word{ffvvvf}\}$ gives
\begin{center}
    \begin{tabular}{cccc}
      $\underbracket[.75pt]{\,\word{ffvvvf}\,}_{\Circled[inner ysep=6pt]{1}}$\;, & $\underbracket[.75pt]{\,\word{fvfvfv}\,}_{\Circled[inner ysep=6pt]{2}}$\;, & $\underbracket[.75pt]{\,\word{vfffvv}\,}_{\Circled[inner ysep=6pt]{3}}$\;, & $\underbracket[.75pt]{\,\word{vvvfff}\,}_{\Circled[inner ysep=6pt]{4}}$ \\
    \end{tabular}
\end{center}
for the multiset $\widetilde{W}$. 

This results in $\gamma(W)=\gamma_{2,3}=(2,3)(1,2)=(3,2,1)\in \mathbf{W}(2,2)$. 

These are the 
two non-isomorphic unitary $\bt$ modules in signature $(2,2)$ that have the same underlying polarized $\bt$ module.
\end{example}

\begin{example} \label{ex:bicycle word to weyl}
Consider the unitary $\bt$ module $\widetilde{M}(\word{ffffvf}, \word{vvvvfv}) $ with $W=\{\word{ffffvf},\word{vvvvvf}\}$, which we first introduced in \cref{example: bicycles}. Letting $w=\word{ffffvf}$, note that $W=\{w,\tau(w^\star)\}$ in accordance with \cref{remark:multiset W}. We compute $L_W=6$ and, after lexicographically ordering,
we obtain for the multiset $\widetilde{W}$
\begin{center}
    \begin{tabular}{cccccc}
      $\underbracket[.75pt]{\,\word{ffffvf}\,}_{\Circled[inner ysep=6pt]{1}}$\;, & $\underbracket[.75pt]{\,\word{ffvfff}\,}_{\Circled[inner ysep=6pt]{2}}$\;, & $\underbracket[.75pt]{\,\word{vfffff}\,}_{\Circled[inner ysep=6pt]{3}}$\;, & $\underbracket[.75pt]{\,\word{vfvvvv}\,}_{\Circled[inner ysep=6pt]{4}}$\;, & $\underbracket[.75pt]{\,\word{vvvfvv}\,}_{\Circled[inner ysep=6pt]{5}}$\;, & $\underbracket[.75pt]{\,\word{vvvvvf}\,}_{\Circled[inner ysep=6pt]{6}}.$\\
    \end{tabular}
\end{center}

By examining which elements above end in $\word{v}$, we obtain that $\gamma(W) = \gamma_{4,5} \in \mathbf{W}(4,2).$ 
\end{example}

\section{Tautological lifts and Newton strata}\label{sec: Tautological lifts and Newton strata}

We now focus on the interaction between the Ekedahl--Oort stratification and the \emph{Newton stratification} of the unitary Shimura variety $\M(a,b)$. Recall from \cref{subsec:eo-newton-def} that the Newton stratification is based on the isogeny class of the $p$-divisible group of the parameterized abelian varieties. The unique closed Newton stratum of $\M(a,b)$ is the supersingular locus, and a $\kk$-point $(A, \lambda, \iota, \xi)$ is contained in the supersingular locus if and only if $A$ is a supersingular abelian variety.

The additional structure we consider on the abelian varieties parameterized by $\mathcal{M}(a,b)$ imposes additional structure on the corresponding $p$-divisible groups, and thereby the corresponding p-adic \Dd{} modules via the \Dd{} equivalence; this gives us the notion of a \emph{unitary} $p$-divisible group and \emph{unitary} $p$-adic \Dd{} module respectively.

\cref{thm: structure of unitary BT1s} describes the structure of unitary $\bt$ modules, which correspond to Ekedahl--Oort strata of unitary Shimura varieties. In this section, we construct, from a unitary $\bt$ module $\widetilde{M}$, an explicit unitary $p$-adic \Dd module called the \emph{tautological lift} of $\widetilde{M}$. 
 
We compute the Newton polygon of this lift and conclude that the given Ekedahl--Oort stratum intersects the resulting Newton stratum. This process relies on the word-based characterization of $\widetilde{M}$ given in \cref{thm: structure of unitary BT1s}.

\subsection{Unitary \texorpdfstring{$\boldsymbol{p}$}{}-divisible groups and \texorpdfstring{$\boldsymbol{p}$}{}-adic \Dd{} modules}\label{subsec:unitary-p-adic-mod}

Let $\xW =\xW(\kk)$ be the ring of Witt vectors of $\kk$ and define $\mathbx{K} \coloneqq \text{Frac}(\xW)=\xW[1/p]$ to be the fraction field of $\xW$. When $\kk=\overline{\ff}_p$, we have $\xW=\breve{\zz}_p$ and $\mathbx{K}=\breve{\qq}_p$, the completion of the maximal unramified extension of $\qq_p$. 

Let $\varphi_0$, $\varphi_1$ be the two embeddings of $\mathcal{O}_K$ into $\breve{\mathbx{Q}}_p$, where $K$ is the imaginary quadratic field in the moduli problem defining $\M(a,b)$. Finally, as before, we have $q =a+b$.

\begin{definition}\label{def:pdiv group} A \emph{unitary $p$-divisible group of signature $(a,b)$} over $\kk$ is a triple $(X, \lambda_X, \iota_X)$, where

\begin{enumerate}
\item{$X$ is a $p$-divisible group over $\kk$ of height $2(a+b)=2q$ and dimension $q$; }
\item{$\lambda_X: X \rightarrow X^\vee$ is a $p$-principal polarization; 
}
\item{$\iota_X: \mathcal{O}_K \otimes_{\mathbx{Z}} \mathbx{Z}_p \rightarrow \mathrm{End}(X)$ is an action satisfying the signature $(a,b)$ condition
$$\mathrm{charpol}(\iota(r) \mid \mathrm{Lie}(X) ) = (T - \phi_0(r))^{a}(T-\phi_1(r))^b \in \breve{\mathbx{Z}}_p[T],$$
for all $r \in \mathcal{O}_K$; and \label{def:pdiv group unitary structure}
\item $\lambda_X$ has the following $\mathcal{O}_K$-linear compatibility condition with $\iota_X$ \[\lambda_X\circ \iota_X(r) = \iota_X(\overline{r})^\vee \circ \lambda_X,\quad \text{for all $r \in \mathcal{O}_K$.}\]
}
\end{enumerate}
\end{definition}
The \Dd{} equivalence (\cref{sec:preliminary DD theory}) respects the additional structure of unitary $p$-divisible groups coming from $\iota_X$, giving rise to $p$-adic \Dd{} modules with additional structure. We call these \emph{unitary} $p$-adic \Dd{} modules.

Recall that $\sigma = \xW(\frob)$ is the lift of the Frobenius automorphism of $\kk$ to $\xW$.

\begin{definition} \label{def: unitary p-adic Dd}
A \emph{unitary $p$-adic \Dd module} is a tuple $\widetilde{\xM} = (\xM, \xF,\xV, \Lambda, \xM = \xM_0 \oplus \xM_1)$, where
\begin{enumerate}
    \item $(\xM, \xF, \xV)$ is a $p$-adic \Dd module over $\kk$;
    \item $\Lambda: \xM \times \xM \to \xW$ is a perfect alternating $\xW$-bilinear pairing;
    \item $\Lambda(\xF(\xx),\xy) = \Lambda(\xx, \xV(\xy))^\sigma$ for every $\xx,\xy\in \xM$;
    \item $\xM_0$ and $\xM_1$ are $\xW$-submodules of $\xM$ of equal rank such that $\xM=\xM_0 \oplus \xM_1$;
    \item $\xF$ and $\xV$ are homogeneous of degree $1$ with respect to the decomposition $\xM=\xM_0\oplus \xM_1$; 
    \item $\xM_0$ and $\xM_1$ are both totally isotropic with respect to $\Lambda$.
\end{enumerate}
We define the \emph{signature} of $\widetilde{\xM}$ to be $(a,b) \coloneqq (\dim_{\kk} (\xM_0/\xF(\xM_1)) , \dim_{\kk} (\xM_1/\xF(\xM_0))).$ It follows that the $\xW$-rank of $\xM$ equals $2(a+b)=2q$.
\end{definition}

Reducing a unitary $p$-adic \Dd module modulo $p$ gives a unitary $\bt$ module. 
Given a unitary $p$-adic \Dd{} module $\widetilde{\xM}$ and unitary $\bt$ module $\widetilde{M}$ such that $\widetilde{\xM}/p\widetilde{\xM} \cong \widetilde{M}$, we say $\widetilde{\xM}$ is a \emph{lift} of $\widetilde{M}$. 
 
Since the $p$-divisible group of an abelian variety determines both its Ekedahl--Oort stratum 
and its Newton stratum, constructing suitable unitary $p$-divisible groups lets us construct points in the intersection of an Ekedahl--Oort stratum and a Newton stratum of a unitary Shimura variety.

\subsection{Tautological lifts} \label{sec: TL}

The goal of this section is to construct suitable lifts of unitary $\bt$ modules and compute their Newton polygons.

In \cite[(2.6)]{Oort05simple}, Oort constructs a \emph{tautological lift}: given a $\bt$ module $\overline{M}$, a $p$-adic \Dd module $\overline{\xM}$ is constructed so that $\overline{\xM}/p\overline{\xM}$ is isomorphic to $\overline{M}$ as a $\bt$ module. In \cite[Remark 6.2]{Harashita_polygon}, this is called the \emph{standard lift}. We now generalize this construction to include a quasi-polarization and an action of $\zz_{p^2} \coloneqq \xW(\ff_{p^2})$, the ring of Witt vectors of $\ff_{p^2}$.

Let $\widetilde{M}=(M,F,V,\lambda,M=M_0\oplus M_1)$ be a unitary $\bt$ module of signature $(a,b)$. As in \cref{subsec:unitary bt1 modules}, let $B_i=\{e_{i,1}, \ldots ,e_{i,q}\}$  be a basis of $M_i$ such that $B=B_0 \cup B_1$ is preserved by $F$ and $V$. Fix a constant $c \in \kk^\times$ satisfying $c^p=-c$. For convenience, assume the pairing $\lambda$ is determined by
$$\lambda(e_{0,j},e_{1,j'}) =\begin{cases} c &\hbox{if $j+j'=q+1$,} \\ 0 &\hbox{else.} \end{cases}$$
This is the pairing given in \cite[(5.5)]{Moonengsas}. Note that this agrees with the pairing in \cref{lem: pairing unicycle}, but differs slightly from the pairing in \cref{lem: pairing bicycle}, replacing $\lambda$ by $c\lambda$.

For the definition below, we fix a constant $C \in \xW$ such that $C^\sigma=-C$ and $C \pmod p \equiv c$.

\begin{definition} \label{def: TL}
Let $\widetilde{M}$ be as above. We make the following definitions:
\begin{enumerate}
    \item let  
    $\xM_i  \coloneqq  \xW \{ \xe_{i,1} , \ldots ,\xe_{i,q}\}$ for $i\in \{0,1\}$, and set $\xM \coloneqq \xM_0 \oplus \xM_1$;
    \item whenever $F(e_{i,j})=e_{i',j'}$, set $\xF(\xe_{i,j}) = \xe_{i',j'}$ and $\xV(\xe_{i',j'}) = p\xe_{i,j}$;
    \item whenever $V(e_{i,j}) = e_{i',j'}$, set $\xV(\xe_{i,j}) = \xe_{i',j'}$ and $\xF(e_{i',j'}) = p\xe_{i,j}$;
    \item using the previous two rules, extend $\xF$ $\sigma$-linearly and extend $\xV$ $\sigma^{-1}$-linearly; and
    \item let $\Lambda: \xM \times \xM \to \xW$ be the alternating $\xW$-bilinear pairing obtained by bilinearly extending:
    $$\Lambda(\xe_{i,j},\xe_{i',j'}) =\begin{cases} C &\hbox{if $j+j'=q+1$ and $i<i'$,} \\
    -C &\hbox{if $j+j'=q+1$ and $i>i'$,}
    \\ 0 &\hbox{else.} \end{cases}$$
    In particular, $\xM_0$ and $\xM_1$ are Lagrangian with respect to $\Lambda$.
\end{enumerate}
We then define the \emph{tautological lift} of $\widetilde{M}$ to be $$\TL(\widetilde{M}) \coloneqq (\xM,\xF,\xV,\Lambda,\xM=\xM_0\oplus \xM_1).$$
\end{definition}

\begin{proposition} \label{prop: TL(M) is lift}
Assume $\widetilde{M}$ is a unitary $\bt$ module of signature $(a,b)$. Then $\TL(\widetilde{M})$ is a unitary $p$-adic \Dd module of signature $(a,b)$, and $\TL(\widetilde{M})/p\TL(\widetilde{M})$ is isomorphic to $\widetilde{M}$ as a unitary $\bt$ module.
\end{proposition} 
\begin{proof}
We begin by showing that $\TL(\widetilde{M})$ is a unitary $p$-adic \Dd module. It is clear that $\xM_0$ and $\xM_1$ are free   
$\xW$-modules of rank $q$, so that $\xM$ is a free  
$\xW$-module of rank $2q$. By construction, $\xF$ is $\sigma$-linear, $\xV$ is $\sigma^{-1}$-linear and $\xF \circ \xV = \xV \circ \xF=p$, so that $(\xM,\xF,\xV)$ is a $p$-adic \Dd module. Next, we observe that $\Lambda$ is a perfect pairing. By construction, it is also  
$\xW$-bilinear and alternating. The condition $\Lambda(\xF(\xx),\xy) = \Lambda(\xx, \xV(\xy))^\sigma$ then follows from the assumption $\lambda(F(x),y)=\lambda(x,V(y))^p$ and the definition of $C$. Finally, it is apparent in the construction that $\xF$ and $\xV$ act homogeneously on $\xM=\xM_0 \oplus \xM_1$ and that $\xM_0$ and $\xM_1$ are totally isotropic with respect to $\Lambda$. We conclude that $\TL(\widetilde{M})$ is a unitary $p$-adic \Dd module.

The isomorphism between $\TL(\widetilde{M})/p\TL(\widetilde{M})$ and $\widetilde{M}$ is explicitly given by $\xe_{i,j} \pmod {p} \mapsto e_{i,j}$. Under this map, since $c \equiv C \pmod{p}$, observe that $\xF, \xV$ and $\Lambda$ correspond to $F,V$ and $\lambda$ respectively. This recovers the structure of $\widetilde{M}$ as a unitary $\bt$ module.

Finally, the statement about the signatures follows from $\TL(\widetilde{M}) /p\TL(\widetilde{M}) \cong \widetilde{M}$, since 
$$\dim_{\kk} (\xM_0/\xF(\xM_1)) = \dim_{\kk}(M_0/F(M_1)) = q-\dim_{\kk}(F(M_1)) = \dim_{\kk}(M_1[F])=a.$$
Exchanging $\xM_0$ and $\xM_1$ gives the formula for $b$.
\end{proof}

If $\overline{M}$ is merely a $\bt$ module, we let $\TL(\overline{M})$ denote the $p$-adic \Dd module defined by only specifying $(\xM,\xF,\xV)$ in \cref{def: TL}. This agrees with the definition in \cite[(2.5)]{Oort05simple}. As is observed there, $\TL(\cdot)$ is not a functor. By construction, it does respect direct sums:
\begin{align}
    \TL(\widetilde{M} \oplus \widetilde{N}) &= \TL(\widetilde{M}) \oplus \TL(\widetilde{N}), \nonumber \\
    \TL(\overline{M} \oplus \overline{N}) &= \TL(\overline{M}) \oplus \TL(\overline{N}). \label{eq: TL oplus}
\end{align}

\subsection{Newton polygons}\label{sec:NewtonPolygon}

Let $\widetilde{M}=(M,F,V,\lambda,M=M_0\oplus M_1)$ be a unitary $\bt$ module and let $\TL(\widetilde{M})$ be its tautological lift. For the purposes of \cref{thm: intersection}, we now compute the \emph{Newton polygon} of $\TL(\widetilde{M})$. This is defined to be the Newton polygon of the underlying $p$-adic \Dd module $\overline{\xM}=(\xM, \xF, \xV)$. By the \Dd-Manin classification,   
$\overline{\xM} \otimes_{\xW} K$ decomposes as a $K[\xF]$-module as follows: 
$$\overline{\xM} \otimes_{\xW} K = \bigoplus_{i=1}^r \overline{\xM}_{m_i, n_i} ^{\lambda_i}.$$
Here $\overline{\xM}_{m_i,n_i}$ is the simple $K[\xF]$-module on which $\xF^{m_i+n_i}=p^{n_i}$. We call $\alpha_i=\frac{m_i}{m_i+n_i}$ the \emph{slope} of $\overline{\xM}_{m_i,n_i}$, we call $\lambda_i$ the \emph{multiplicity} of $\alpha_i$, and we call $\mu_i=\lambda_i(m_i+n_i)$ the \emph{length} of the slope $\alpha_i$. We assume $\alpha_1 < \cdots < \alpha_r$. Then the \emph{Newton polygon} is the polygon from $(0,0)$ to $(2q,q)$ obtained by concatenating the line segments of slope $\alpha_i$ and horizontal length $\mu_i$.

\begin{definition} \label{def: NP(W)}
    Let $w=u_{\ell-1} \cdots u_0$ be a word in the alphabet $\{\word{f},\word{v}\}$. Then we define its slope
    $$\alpha(w)  \coloneqq  \frac{\#\{i \; | \; u_i=\texttt{v}\}}{\ell}.$$
    Now, let $W$ be a multiset of words. We define $\NP(W)$ to be the Newton polygon obtained by letting each $w \in W$ contribute the slope $\alpha(w)$ with length $\ell(w)$.
\end{definition}

\begin{proposition} \label{prop: NP of TL}
Let $\widetilde{M}$ be a unitary $\bt$ module whose underlying $\bt$ module $\overline{M}$ decomposes as $\overline{M} = \bigoplus_{w \in W} \overline{M}(w).$ Then the tautological lift $\TL(\widetilde{M})$ has Newton polygon $\NP(W)$. 
\end{proposition}
\begin{proof}
Write $\overline{\xM}$ for the underlying $p$-adic \Dd module of $\TL(\widetilde{M})$. By Equation~\eqref{eq: TL oplus},
\[\overline{\xM} = \bigoplus_{w \in W} \TL(\overline{M}(w)),\; \text{ and } \quad \overline{\xM} \otimes_{\xW} K  = \bigoplus_{w \in W} \left(\TL(\overline{M}(w)) \otimes_{\xW} K \right).\] 
The slopes of $\TL(\overline{M}(w))$ are calculated in \cite[(2.6)]{Oort05simple}; for completeness, we also calculate them here. It is clear from \cref{def: TL} that on $\TL(\overline{M}(w))$, with $w=u_{\ell-1} \cdots u_0$, we have $\xF^\ell = p^{\#\{i \; | \; u_i=\word{v}\}}$. Thus each $w\in W$  contributes the slope $\alpha(w)$ with length $\ell(w)$. We conclude that the Newton polygon of $\TL(\widetilde{M})$ is $\NP(W)$, as desired.
\end{proof}

\begin{definition} \label{def: balanced}
We shall say a word $w=u_{\ell-1} \cdots u_0$ 
is \emph{balanced} if it contains the same number of occurrences of $\word{f}$ and $\word{v}$: $\#\{i \; | \; u_i=\word{f}\} = \#\{i\; | \; u_i=\word{v}\} = \frac{\ell}{2}$.
\end{definition}

Note that every self-dual word is balanced, but not every balanced word is self-dual; i.e., for a word, being self-dual is a strictly stronger condition than being balanced. 

\begin{corollary} \label{cor: TL ss}
Let $\widetilde{M}$ be a unitary $\bt$ module whose underlying $\bt$ module $\overline{M}$ decomposes as $\overline{M} = \bigoplus_{w \in W} \overline{M}(w).$ Then $\TL(\widetilde{M})$ is supersingular if and only if every $w \in W$ is balanced.
\end{corollary}
\begin{proof}
By \cref{prop: NP of TL}, $\TL(\widetilde{M})$ is supersingular if and only if $\alpha(w)=\frac{1}{2}$ for each $w \in W$. By \cref{def: NP(W)}, this happens precisely when $w$ is balanced.
\end{proof}

\begin{example}
We compute the Newton polygons of all tautological lifts in signature $(2,2)$. In signature $(2,2)$, 
all words that are balanced are associated with the Newton polygon of slope $\frac{1}{2}$, joining $(0,0)$ to $(8,4)$, illustrated as $\nu_1$ below. 
The Newton polygon associated to $\gamma_{2,4}$, with  $W=\{\word{ff}, \word{fv},\word{vf},\word{vv}\}$, has three segments, with slopes $0,\frac12,$ and $ 1$, as shown with the polygon $\nu_2$. The remaining ordinary Newton polygon is drawn as $\nu_3$, with slopes 0 and 1. 
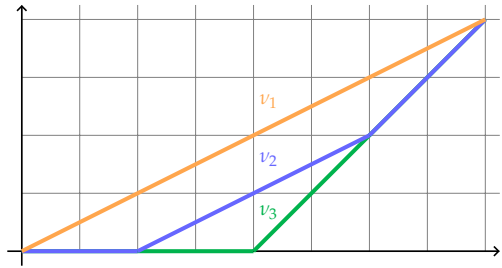
\begin{center}
\resizebox{0.4\textwidth}{!}{
\begin{tikzpicture}

\draw[thin,color=gray] (0,0) grid (8.25,4.25);   
    \draw[thick,->] (-0.25,0) -- (8.25,0);
    \draw[thick,->] (0,-0.25) -- (0,4.25);

   \draw[line width=.7mm,color=green!70!blue] (0,0) -- (4,0);
   \draw[line width=.7mm,color=green!70!blue] (4,0) -- (8,4);

    \draw[line width=.7mm,color=blue!60] (0,0) -- (2,0);
    \draw[line width=.7mm,color=blue!60] (2,0) -- (6,2); 
    \draw[line width=.7mm,color=blue!60] (6,2) -- (8,4);    
    \draw[line width=.7mm,color=orange!70!] (0,0) -- (8,4);

    \node[color=orange!70!,xshift=7.5pt,yshift=2.5pt] at (4,2.5) {$\nu_1$};
    \node[color=blue!60,xshift=7.5pt,yshift=2.5pt] at (4,1.5) {$\nu_2$};
    \node[color=green!70!blue,xshift=7.5pt,yshift=5pt] at (4,0.5) {$\nu_3$};
\end{tikzpicture}}
\end{center}  
\end{example}

\subsection{Intersections between Ekedahl--Oort strata and Newton strata}

We now employ all the results we have obtained for Dieudonn\'{e} modules to prove a result about the intersection between the Ekedahl--Oort stratification and the Newton stratification.

In what follows, we assume, following Equation~\eqref{eq: B decomposition 2}:
\begin{enumerate}
    \item $W_1$ is a multiset of self-dual-primitive words whose length is not divisible by $4$; and 
    \item $W_2$ is a multiset of words of even length that are either primitive, but not of the form $w=yy^{\star}$ where $y$ has odd length, or of the form $w=zz$ with $z$ primitive of odd length.
\end{enumerate}
We set \[q \coloneqq \displaystyle\sum_{w \in W_1} \ell(w)/2 + \sum_{w\in W_2} \ell(w)\] and 
\begin{align*}
    b& \coloneqq  \sum_{u_{\ell-1} \cdots u_0 \in W_1} \#\{ 1\leq j \leq \ell/2 \; | \; u_{2j} = \word{v} \} \\[0.5em] &\qquad + \sum_{u_{\ell-1} \cdots u_0 \in W_2} \left( \#\{1\leq j \leq  \ell/2  \; | \; u_{2j} = \word{v}\} + \#\{ 0 \leq j \leq   \ell/2  \; | \; u_{2j-1}=\word{f}\} \right).
\end{align*}

Moreover, we define $a \coloneqq q-b$. Finally, following \cref{remark:multiset W}, we define the following multiset of words:
\begin{equation} \label{eq: def W}
    W \coloneqq  W_1 \cup W_2 \cup \{\tau(w^\star) \; | \; w \in W_2\}.
\end{equation}

This is the multiset that will be used when applying \cref{thm: word Weyl}.

For a permutation $\gamma \in \mathbf{W}(a,b)$, we denote by $\M(a,b)_{\gamma}$ the Ekedahl--Oort stratum of $\M(a,b)$ corresponding to $\gamma$ via \cite[Theorem~6.7]{Moonengsas}. Likewise, given a Newton polygon $\mathcal{N}$, we denote by $\M(a,b)^{\mathcal{N}}$ the Newton stratum of $\M(a,b)$ corresponding to the Newton polygon $\mathcal{N}$. Recall also the definition of $\gamma(W)$ in \cref{def: gamma(W)} and the definition of $\NP(W)$ in \cref{def: NP(W)}.

\begin{theorem} \label{thm: intersection}
Let the notation be as above. The intersection 
$\M(a,b)_{\gamma(W)} \cap \M(a,b)^{\NP(W)}$ is non-empty.
\end{theorem}
\begin{proof}
Using \cref{def: unicycle,def: (Serre) bicycles}, we construct the unitary $\bt$ module 
\begin{equation} \label{eq: Mtilde W1 W2}
    \widetilde{M}  \coloneqq  \bigoplus_{w\in W_1} \widetilde{M}(w) \oplus \bigoplus_{w \in W_2} \widetilde{M}(w,w^\star).
\end{equation}

It follows from \cref{lem: unitary M(w),lem: unitary M(w)+M(w)-c} that this is a unitary $\bt$ module. It can be verified using \cref{rmk: unicycle signature,rmk: bicycle signature} that $\widetilde{M}$ has signature $(a,b)$. Moreover, \cref{thm: word Weyl} provides that $\widetilde{M}$ corresponds to the Ekedahl--Oort stratum $\M(a,b)_{\gamma(W)}$ via \cite[Theorem~6.7]{Moonengsas}.

We now construct the tautological lift $\TL(\widetilde{M})$, following \cref{def: TL}. By \cref{prop: TL(M) is lift}, $\TL(\widetilde{M})$ is a unitary $p$-adic \Dd module whose mod-$p$ reduction is isomorphic to $\widetilde{M}$ as a unitary $\bt$ module. Furthermore, \cref{prop: NP of TL} provides that the Newton polygon of $\TL(\widetilde{M})$ is $\NP(W)$. 

Through the Dieudonn\'{e} equivalence, $\TL(\widetilde{M})$ gives rise to a unitary $p$-divisible group $(X,\lambda_X,\iota_X)$. Finally, by \cite[Theorem 1.6 (2)]{ViehmannWedhorn2013}, there exists a point $(A,\lambda,\iota,\xi)$ in $\mathcal{M}(a,b)$ such that the unitary $p$-divisible group $(A[p^\infty], \lambda_{A[p^\infty]}, \iota_{A[p^\infty]})$ is isomorphic to $(X,\lambda_X,\iota_X)$.

By \cref{prop: TL(M) is lift,thm: word Weyl}, this point lies in $\M(a,b)_{\gamma(W)}$, and by \cref{prop: NP of TL}, this point lies in $\M(a,b)^{\NP(W)}$. Thus, the point lies in the intersection $\M(a,b)_{\gamma(W)} \cap \M(a,b)^{\NP(W)}$, which is therefore non-empty. \end{proof}

\begin{remark}
    Note that the unitary structure on the $p$-divisible groups provides the ``$\mathcal{D}$-structure'' described in \cite{ViehmannWedhorn2013}. 
\end{remark}

\begin{corollary}\label{corollary:balancedssintersection1}
Assume every word in $W_2$ is balanced. Then $\M(a,b)_{\gamma(W)} \cap \M(a,b)^{\textnormal{ss}}$ is non-empty, where $\M(a,b)^{\textnormal{ss}}$ the supersingular locus of $\M(a,b)$.
\end{corollary}
\begin{proof}
This follows immediately from \cref{thm: intersection,cor: TL ss}, together with the observation that every word of the form $yy^\star$ is balanced.
\end{proof}

If $\widetilde{M}$ is indecomposable and $q$ is odd, then $W_1$ contains a single, balanced word and $W_2$ is empty. Hence the Ekedahl--Oort stratum of $\widetilde{M}$ intersects the supersingular locus. 
Similarly, when $W_1$ is empty and $W_2=\{w\}$ with $w$ self-dual, then $w$ is balanced and $\M(a,b)_{\gamma(W)}$ intersects the supersingular locus.
This leads to the following corollary.

\begin{corollary}\label{corollary:balancedssintersection2} $\M(a,b)_{\gamma(W)} \cap \M(a,b)^{\textnormal{ss}}$ is non-empty if $\gamma(W)$ corresponds to a unitary unicycle or Serre unicycle.  
\end{corollary}

\begin{remark}
The converse of \cref{corollary:balancedssintersection1} is not true. In \cite[Lemma 43]{GroenLupoianParker2026}, the authors prove the existence of a supersingular principally polarized abelian variety $A$ of dimension $5$ whose $p$-torsion group scheme corresponds to the polarized $\bt$ module $\widehat{M}(\word{fffvv},\word{vvvff})$. Applying the Serre tensor construction described in \cref{section:STC} to $A$ yields an abelian variety $\mathcal{O}_K \otimes_{\zz} A$, which admits an action by $\mathcal{O}_K$ of signature $(5,5)$. 
This construction gives rise to a point in $\M(5,5)$. Since the Serre tensor construction of a supersingular abelian variety is supersingular, this point lies in $\M(5,5)^{\textnormal{ss}}$. On the other hand, by \cref{prop: STC bicycle}, its $p$-torsion group scheme corresponds to the unitary $\bt$ module $\widetilde{M}(\word{fffvvfffvv},\word{vvvffvvvff})$. By Equation~\eqref{eq: Mtilde W1 W2}, we have $W_1=\emptyset$ and $W_2=\{\word{fffvvfffvv} \}$. Then Equation~\eqref{eq: def W} yields
\[
W= \{\word{fffvvfffvv}, \word{vvffvvvffv}\}.
\]
Note that using $W_2=\{ \word{vvvffvvvff}\}$ gives the same result. It is clear that the word $\word{fffvvfffvv}$ is not balanced and \cref{def: NP(W)} prescribes that the Newton polygon of the tautological lift has slopes $2/5$ and $3/5$, both with length $10$.  
Thus we have constructed a point in $\M(a,b)_{\gamma(W)} \cap \M(a,b)^{\textnormal{ss}}$ even though the words in $W_2$ are not balanced.

We can use the ``word to Weyl'' algorithm of \cref{sec: word Weyl} to identify the Weyl group coset representative that this corresponds to. To do this, we lexicographically order the double rotations of each component $\{\word{fffvvfffvv}, \word{vvffvvvffv}\}$, which yields the ordered list
\begin{center}
    \begin{tabular}{ccccc}
      $\underbracket[.75pt]{\,\word{fffvvfffvv}\,}_{\Circled[inner ysep=6pt]{1}}$\;, & $\underbracket[.75pt]{\,\word{ffvvfffvvf}\,}_{\Circled[inner ysep=6pt]{2}}$\;, & $\underbracket[.75pt]{\,\word{ffvvvffvvv}\,}_{\Circled[inner ysep=6pt]{3}}$\;, & $\underbracket[.75pt]{\,\word{fvvfffvvff}\,}_{\Circled[inner ysep=6pt]{4}}$\;, & $\underbracket[.75pt]{\,\word{fvvvffvvvf}\,}_{\Circled[inner ysep=6pt]{5}}$\;, \\[2.5em]
      $\underbracket[.75pt]{\,\word{vfffvvfffv}\,}_{\Circled[inner ysep=6pt]{6}}$\;, & $\underbracket[.75pt]{\,\word{vffvvvffvv}\,}_{\Circled[inner ysep=6pt]{7}}$\;, & $\underbracket[.75pt]{\,\word{vvfffvvfff}\,}_{\Circled[inner ysep=6pt]{8}}$\;, & $\underbracket[.75pt]{\,\word{vvffvvvffv}\,}_{\Circled[inner ysep=6pt]{9}}$\;, & $\underbracket[.75pt]{\,\word{vvvffvvvff}\,}_{\Circled[inner ysep=6pt]{10}}$ \\
    \end{tabular}
\end{center}

Thus, $\gamma(W) = \gamma_{1,3,6,7,9} \in \mathbf{W}(5,5)$. 
\end{remark}

\appendix
\section{Counting Indecomposables}\label{appendix:indecomp-count}

For integers $q>0$ and $0 \leq b\leq q$, let $\widetilde{\ind}(q-b,b)$ be the set of isomorphism classes of indecomposable unitary $\bt$ modules with signature $(q-b,b)$ and therefore of $\kk$-dimension $2q$. $\widetilde{\ind}(q-b,b)$ corresponds to the indecomposable Ekedahl--Oort strata of the unitary Shimura variety $\mathcal{M}(q-b,b)$.

By \cref{cor: indecomposables,rmk: sign-fix-indecomp}, we have the following characterization:
\[\widetilde{\ind}(q-b,b) = \begin{cases}
    \widetilde{\uni}(q-b,b) & \text{when $q$ is odd,}\\[0.5em]%
    \widetilde{\bi}(q-b,b) & \text{when $q$ is even and $b \neq q/2$,}\\[0.5em]%
    \widetilde{\bi}(q/2,q/2) \sqcup \widetilde{\uni}_{\serre}(q/2,q/2) & \text{when $q\equiv 0 \bmod{4}$ and $b = q/2$,}\\[0.5em]%
    \widetilde{\bi}(q/2,q/2) \sqcup \widetilde{\bi}_{\serre}(q/2,q/2) & \text{when $q\equiv 2 \bmod{4}$ and $b = q/2$.}
\end{cases}\]

where $\widetilde{\uni}(q-b,b),\,\widetilde{\bi}(q-b,b),\,\widetilde{\uni}_{\serre}(q/2,q/2)$ and $\widetilde{\bi}_{\serre}(q/2,q/2)$ are the sets of isomorphism classes of unitary unicycles (\cref{def: unitary unicycle}), unitary bicycles, Serre unicycles, and Serre bicycles (\cref{def: (Serre) bicycles}) respectively. Isomorphisms here respect the unitary structure.

In this appendix, we give a complete answer to the following combinatorial question:

\begin{question}\label{ques:indecomp-count}
    Consider the unitary Shimura variety $\M(q-b,b)$ and let $T$ be an indecomposable type occurring in this signature. How many Ekedahl--Oort strata of $\M(q-b,b)$ have isomorphism type $T$? 
    
    Equivalently:
    
    What are the cardinalities of the sets $\widetilde{\uni}(q-b,b),\,\widetilde{\bi}(q-b,b),\,\widetilde{\uni}_{\serre}(q/2,q/2)$ and $\widetilde{\bi}_{\serre}(q/2,q/2)$?
\end{question}

We summarize the results that answer \cref{ques:indecomp-count}. First, define,
\begin{align*}
   \widetilde{\ind}(q) &\coloneqq \bigsqcup_{0 \leq b \leq q} \widetilde{\ind}(q-b,b), & \widetilde{\uni}(q) &\coloneqq \bigsqcup_{0 \leq b \leq q} \widetilde{\uni}(q-b,b), & \widetilde{\bi}(q) &\coloneqq \bigsqcup_{0 \leq b \leq q} \widetilde{\bi}(q-b,b);
\end{align*}
which collect isomorphism classes of all indecomposable $\bt$ modules, unitary unicycles and bicycles, respectively, of $\kk$-dimension $2q$. The signature-agnostic counts --- $\# \widetilde{\uni}(q)$, $\# \widetilde{\bi}(q)$ and $\# \widetilde{\ind}(q)$ --- are determined in \cref{cor: uni-uni-count,cor: uni-bi-count,thm: uni-indecomp-count} respectively. 

The refined signature-specific counts for unitary unicycles $\# \widetilde{\uni}(q-b,b)$ and unitary bicycles $\widetilde{\bi}(q-b,b)$ are determined in \cref{cor: uni-uni-count-final,cor: uni-bi-count-final} respectively, while the counts for Serre unicycles $\# \widetilde{\uni}_{\serre}(q-b,b)$ and Serre bicycles $\widetilde{\bi}(q-b,b)$  are both determined in \cref{prop: uni-serre-count}. 

\subsection{Preliminaries}
In this section, we record various notions that are used throughout in what follows.

Our counting arguments will rely crucially on the following general counting principle.
\begin{lemma}\label{lemma: mobius-count}
    Let $n > 0$ be an integer, and let $\mathscr{D}_{\mathbf{P}}(n)$ be a subset of $\mathscr{D}(n)$, the set of positive divisors of $n$, satisfying some property $\mathbf{P}$ such that $n \in \mathscr{D}_{\mathbf{P}}(n)$. Assume you have two families of sets $\set{\mathsf{S}_d}_{d \in \mathscr{D}_{\mathbf{P}}(n)}$ and $\set{\mathsf{X}_d}_{d \in \mathscr{D}_{\mathbf{P}}(n)}$ with
    \[\mathsf{S}_n \cong \bigsqcup_{d \in \mathscr{D}_{\mathbf{P}}(n)} \mathsf{X}_d.\]
    Then,
    \[\#\mathsf{X}_n = \sum_{d \in \mathscr{D}_{\mathbf{P}}(n)} \mu(n/d)\,\#\mathsf{S}_{d},\quad \text{where $\mu$ is the M{\"o}bius function.}\]
\end{lemma}
\begin{proof}
    Let the notation be as in the statement of the lemma. The given partition yields
    \[\#\mathsf{S}_n = \sum_{d \in \mathscr{D}_{\mathbf{P}}(n)} \#\mathsf{X}_d.\]
    The claim follows from applying the M{\"o}bius inversion formula on the poset $\mathscr{D}_{\mathbf{P}}(n)$.
\end{proof}

\begin{remark}\label{rmk: co-propert divisors}
Given $\mathscr{D}_{\mathbf{P}}(n)$ containing $n$, we may also consider the opposite poset \[\mathscr{D}_{\textnormal{co-}\mathbf{P}}(n) \coloneqq \setp{d \in \mathscr{D}(n)}{n/d \in \mathscr{D}_{\mathbf{P}}(n)}\]
containing $1$. The map $d \mapsto n/d$ establishes an order-reversing bijection between $\mathscr{D}_{\mathbf{P}}(n)$ and $\mathscr{D}_{\textnormal{co-}\mathbf{P}}(n)$.

Using the notation in \cref{lemma: mobius-count}, we may rewrite the sum as follows:
    \[\#\mathsf{X}_n = \sum_{d \in \mathscr{D}_{\mathbf{P}}(n)} \mu(n/d)\,\#\mathsf{S}_{d} = \sum_{d \in \mathscr{D}_{\textnormal{co-}\mathbf{P}}(n)} \mu(d)\,\#\mathsf{S}_{n/d}.\]

In what follows, the possible properties are:
\begin{enumerate}[label=$\bullet$]
    \item $\mathbf{P}=\emptyset$, i.e., $\mathscr{D}_{\mathbf{P}}(n) = \mathscr{D}(n)$. In this case, $\mathscr{D}_{\textnormal{co-}\mathbf{P}}(n) = \mathscr{D}(n)$.
    \item $\mathbf{P}=\text{odd or even}$, where $\mathscr{D}_{\odd}(n)$ (resp. $\mathscr{D}_{\even}(n)$) collects positive odd (resp. even) divisors of $n$.
\end{enumerate}
\end{remark}

What we will effectively count to answer \cref{ques:indecomp-count} will be words words in the alphabet $\set{\word{f},\word{v}}$ subject to some constraints. We establish some basic notation and results for these words.

\begin{definition}\label{def: basic-comb-set}
    Fix an integer $\ell > 0$. Let $\mathsf{W}_\ell$ be the set of all words of length $\ell$ in $\set{\word{f},\word{v}}$. This set comes equipped with an action of the rotation operator $\tau$. For any $w \in \mathsf{W}_\ell$, one has $\tau^{\ell}(w) = w$.

    For any $w \in \mathsf{W}_{\ell}$, let its \emph{$\tau$-period} be the smallest positive integer $d$ such that $\tau^d(w) = w$. Define 
    \[\mathsf{P}_{\ell,d} = \setp{z \in \mathsf{W}_{\ell}}{\text{$z$ has $\tau$-period $=d$}}.\]
    
    The set of primitive words of length $\ell$ is precisely $\mathsf{P}_{\ell,\ell}$. For ease of notation, define $\mathsf{P}_{\ell} \coloneqq \mathsf{P}_{\ell,\ell}$.

    When $\ell$ is even, within $\mathsf{P}_{\ell}$, we have the following important subset of self-dual words
    \[\mathsf{U}_{\ell} \coloneqq \setp{w \in \mathsf{P}_\ell}{\text{$w$ is self-dual}} = \setp{yy^\star \in \mathsf{P}_{\ell}}{y \in \mathsf{W}_{\ell/2}};\]
    the latter equality follows from \cref{lem: self-dual word}. We denote the set of non-self-dual-words, as $\mathsf{B}_{\ell} \coloneqq \mathsf{P}_{\ell}\setminus \mathsf{U}_{\ell}$. 
    
    When $\ell$ is odd, necessarily, $\mathsf{U}_{\ell} = \emptyset$ and $\mathsf{B}_{\ell} = \mathsf{P}_\ell$.
\end{definition}

\begin{lemma} \label{lemma: word-count}
    One has $\# \mathsf{W}_\ell = 2^{\ell}$.
\end{lemma}

\begin{remark}\label{rmk: prim-partition}
    For words in $\mathsf{W}_{\ell}$, the maximal possible $\tau$-period is $\ell$. Assume $d$ occurs as a $\tau$-period of some word in $\mathsf{W}_{\ell}$, by minimality of $d$ and $\ell$ being the maximal such number, we readily deduce that $d\mid \ell$. Moreover, observe that any $z \in \mathsf{P}_{\ell,d}$ may be written as $z = h^{\ell/d}$ for some primitive word $h$ of length $d$. Therefore, the following
    \[\Psi:\mathsf{P}_d \to \mathsf{P}_{\ell,d}:h \mapsto h^{\ell/d}\]
    is a bijection. Hence, we have a partition
    \[\mathsf{W}_{\ell} = \bigsqcup_{d\in \mathscr{D}(\ell)}\mathsf{P}_{\ell,d} \cong \bigsqcup_{d \in \mathscr{D}(\ell)}\mathsf{P}_d.\]
\end{remark}

The following serves as a template on how \cref{lemma: mobius-count,rmk: co-propert divisors} will be applied subsequently.

\begin{proposition}\label{prop: prim-count}
One has $\displaystyle \# \mathsf{P}_{\ell} = \sum_{d\in \mathscr{D}(\ell)}\mu(d)\, 2^{\ell/d}$.
\end{proposition}
\begin{proof}
By \cref{rmk: prim-partition}, we have $\displaystyle\mathsf{W}_{\ell} \cong \bigsqcup_{d\in \mathscr{D}(\ell)} \mathsf{P}_d$, where $\mathsf{P}_d$ is the set of primitive words of length $d$. 

By \cref{lemma: mobius-count,lemma: word-count}, we obtain
    \[\# \mathsf{P}_{\ell} = \sum_{d\in \mathscr{D}(\ell)}\mu(\ell/d)\, \# \mathsf{W}_{d} = \sum_{d\in \mathscr{D}(\ell)}\mu(\ell/d)\, 2^{d} = \sum_{d\in \mathscr{D}(\ell)}\mu(d)\, 2^{\ell/d},\]
    where the latter equality follows from \cref{rmk: co-propert divisors}.
\end{proof}

\subsection{Prologue: Indecomposable \texorpdfstring{$\boldsymbol{\bt}$}{} Modules} 
Let $\ind(m)$ be the set of isomorphism classes of indecomposable $\bt$ modules of $\kk$-dimension $m > 0$. The rotation operator $\tau$ acts on $\mathsf{P}_{m}$ freely with exact order $m$, and therefore by \cref{cor: Kraft} we have 
\[\#\ind(m) = \frac{1}{m}\# \mathsf{P}_m.\]

Therefore, by \cref{prop: prim-count}, we obtain $\displaystyle\#\ind(m) = \frac{1}{m}\sum_{d\in \mathscr{D}(m)}\mu(d)\, 2^{m/d}$.

\subsection{Indecomposable Polarized \texorpdfstring{$\boldsymbol{\bt}$}{} Modules}
For an integer $g>0$, let $\widehat{\ind}(g)$ be the set of isomorphism classes of indecomposable polarized $\bt$ modules of $\kk$-dimension $2g$. By \cref{cor:Oort}, we have
\[\widehat{\ind}(g) = \widehat{\uni}(g) \sqcup \widehat{\bi}(g),\]
where $\widehat{\uni}(g)$ and $\widehat{\bi}(g)$ are the sets of isomorphism classes of unicycles (\cref{def: unicycle}) and bicycles (\cref{def: bicycle}) respectively. Isomorphisms here respect the polarization.

\subsubsection{Unicycles}\label{subsubsec:uni-count}
By \cref{prop: Oort}, $\widehat{\uni}(g)$ is in bijection with the $\tau$-orbits on $\mathsf{U}_{2g}$, the set of primitive self-dual words of length $g$. The rotation operator $\tau$ acts on $\mathsf{U}_{2g}$ freely with exact order $2g$ and therefore
\[\#\widehat{\uni}(g) = \frac{1}{2g}\# \mathsf{U}_{2g}.\]

To determine $\#\mathsf{U}_{2g}$, we introduce the following useful notions.

\begin{definition}\label{def:dual-prim}
    For a word $w = u_{\ell - 1}\cdots u_1u_0$, define the \emph{dual-rotation operator} $\tau^\star$ as \[\tau^\star(u_{\ell - 1}\cdots u_1u_0) = u_0^\star u_{\ell - 1}\cdots u_1.\]

For any $w \in \mathsf{W}_{\ell}$, let its \emph{$\tau^\star$-period} be the smallest positive integer $t$ such that $(\tau^\star)^t(w) = w$. Define 
    \[\mathsf{Y}_{\ell,t} = \setp{y \in \mathsf{W}_{\ell}}{\text{$y$ has $\tau^\star$-period $= t$}}\]

For any word $w \in \mathsf{W}_{\ell}$, observe that the maximal possible $\tau^\star$-period is $2\ell$, since $(\tau^\star)^{\ell}(w) = w^\star$, and therefore $(\tau^\star)^{2\ell}(w) = w$. For ease of notation, define $\mathsf{Y}_{\ell} \coloneqq \mathsf{Y}_{\ell,2\ell}$ and call its elements \emph{dual primitive}.
\end{definition}

\begin{remark}\label{rmk: dual-prim-partition}
For words in $\mathsf{W}_{\ell}$, the maximal possible $\tau^\star$-period is $2\ell$. Assume $t$ occurs as a $\tau^\star$-period of some word in $\mathsf{W}_{\ell}$, by minimality of $t$ and $2\ell$ being the maximal such number, we readily deduce that $t\mid 2\ell$. 

    Assume there exists $w \in \mathsf{W}_\ell$ with $\tau^\star$-period $t\mid \ell$. Therefore $(\tau^\star)^\ell(w) = (\tau^\star)^t(w) = w$. But recall that we must have $(\tau^\star)^\ell(w) = w^\star$, giving us $w = w^\star$, a contradiction.

    Hence a $\tau^\star$-period $t$ of an element in $\mathsf{W}_\ell$ is such that $t\mid 2\ell$ and $t\nmid \ell$. Thus, necessarily, $t = 2d$ with $2d \nmid \ell$. Equivalently, $2d \mid 2\ell$ and $2\nmid (\ell/d)$. Therefore, $d \mid \ell$ and $\ell/d$ is odd. Hence, we have the following partition 
    \[\mathsf{W}_{\ell} = \bigsqcup_{d\in \mathscr{D}_{\textnormal{co-}\odd}(\ell)} \mathsf{Y}_{\ell,2d}.\]
    where, following \cref{rmk: co-propert divisors}, $\mathscr{D}_{\textnormal{co-}\odd}(\ell) = \setp{d \in \mathscr{D}(\ell)}{\ell/d \in \mathscr{D}_{\odd}(\ell)}$.
\end{remark}

\begin{lemma}\label{lemma: dual-word-bij}
For any $d \in \mathscr{D}_{\textnormal{co-}\odd}(\ell)$, we have the following commutative diagram consists of bijections
\[
    \begin{tikzcd}[row sep=large]
        \mathsf{Y}_d \arrow[r, "\Phi_1"] \arrow[d, "\Phi_0"'] & \mathsf{U}_{2d} \arrow[d, "\Phi_2"] \\
        \mathsf{Y}_{\ell,2d} \arrow[r, "\Phi_3"'] & \mathsf{U}_{2\ell,2d}
    \end{tikzcd}
    \]
    where $\mathsf{U}_{2\ell,2d} = \setp{w \in \mathsf{W}_{2\ell}}{w \text{ is self-dual and has $\tau$-period $=2d$}}$ 
    and the maps are
    \begin{align*}
        \Phi_1: h &\mapsto hh^\star & \Phi_2: yy^\star &\mapsto (yy^\star)^{\ell/d} & \Phi_3 : z &\mapsto zz^\star & \Phi_0 &\coloneqq \Phi_3^{-1}\circ\Phi_2\circ\Phi_1.
    \end{align*}
\end{lemma}
\begin{proof}
The bijectivity of $\Phi_1$ and $\Phi_3$ follows immediately. Once we show $\Phi_2$ is well-defined, its bijectivity also follows directly. For a primitive self-dual word $yy^\star$ of length $2d$, it is clear that $\Phi_2(yy^\star) = (yy^\star)^{\ell/d}$ has $\tau$-period $2d$; we are left to justify that $(yy^\star)^{\ell/d}$ is self-dual. By assumption, $\ell/d$ is odd; write $\ell/d = 2k + 1$ for some $k$. Hence $(yy^\star)^{\ell/d} = xx^\star$, where $x = (yy^\star)^k y$, and thus $(yy^\star)^{\ell/d}$ is self-dual and $\Phi_2(yy^\star) \in \mathsf{U}_{2\ell,2d}$.

Therefore $\Phi_2$ is a bijection. We may additionally make $\Phi_0$ explicit; we have $\Phi_0(h) = (hh^\star)^k h$.
\end{proof}

\begin{proposition}\label{prop: pol-uni-set-count}
One has $\displaystyle \# \mathsf{U}_{2g} = \sum_{d\in \mathscr{D}_{\odd}(g)}\mu(d)\, 2^{g/d}$.
\end{proposition}
\begin{proof}
       By \cref{rmk: dual-prim-partition,lemma: dual-word-bij}, we may partition $\mathsf{W}_{g}$ as \[\mathsf{W}_{g} = \bigsqcup_{d\in \mathscr{D}_{\textnormal{co-}\odd}(g)} \mathsf{Y}_{g,2d} \cong \bigsqcup_{d\in \mathscr{D}_{\textnormal{co-}\odd}(g)} \mathsf{Y}_{d} \cong \bigsqcup_{d\in \mathscr{D}_{\textnormal{co-}\odd}(g)} \mathsf{U}_{2d}.\]
    
    By \cref{lemma: mobius-count,lemma: word-count}, we obtain
    \[\# \mathsf{U}_{2g} = \sum_{d\in \mathscr{D}_{\textnormal{co-}\odd}(g)}\mu(g/d)\, \# \mathsf{W}_{d} = \sum_{d\in \mathscr{D}_{\textnormal{co-}\odd}(g)}\mu(g/d)\,  2^{d} = \sum_{d\in \mathscr{D}_{\odd}(g)}\mu(d)\,  2^{g/d},\]
    where the latter equality follows from \cref{rmk: co-propert divisors}.
    \end{proof}

\begin{corollary}\label{cor: poli-uni-count}
One has $\displaystyle \# \widehat{\uni}(g) = \frac{1}{2g}\sum_{d\in \mathscr{D}_{\odd}(g)}\mu(d)\, 2^{g/d}$.
\end{corollary}

\subsubsection{Bicycles}
By \cref{prop: Oort}, $\widehat{\bi}(g)$ is in bijection with the orbits on $\mathsf{B}_g = \mathsf{P}_g \setminus \mathsf{U}_g$, the set of primitive non-self-dual words of length $g$, under both $\tau$ and dualizing $(\cdot)^\star:w \mapsto w^\star$. The map $\tau$ acts freely with exact order $g$ since $\mathsf{B}_g$ consists of primitive words, and $(\cdot)^\star$ is an involution commuting with the action of $\tau$ and acts freely on the $\tau$-orbits of $\mathsf{B}_g$ since $\mathsf{B}_g$ consists of non-self-dual words. Therefore,
\[\#\widehat{\bi}(g) = \frac{1}{2g}\# \mathsf{B}_{g}.\]

\begin{proposition}\label{prop: pol-bi-set-count}
    One has 
    \[\displaystyle \#\widehat{\bi}(g) = \begin{cases}
        \displaystyle \frac{1}{2g}\sum_{d\in \mathscr{D}(g)}\mu(d)\, 2^{g/d} & \text{when $g$ is odd,}\\[2em]
        \displaystyle \frac{1}{2g}\left(\sum_{d\in \mathscr{D}(g)}\mu(d)\, 2^{g/d} - \sum_{d\in \mathscr{D}_{\odd}(g)}\mu(d)\, 2^{g/2d}\right) & \text{when $g$ is even.}
    \end{cases}\]
\end{proposition}
\begin{proof}
We determine $\# \mathsf{B}_g$. For $g$ odd, all words of length $g$ are non-self-dual. Therefore, by \cref{prop: prim-count},
    \[\#\mathsf{B}_g = \# \mathsf{P}_{g} = \sum_{d\in \mathscr{D}(g)}\mu(d)\, 2^{g/d}.\]
    
    When $g$ is even, then $\mathsf{B}_{g} = \mathsf{P}_g \setminus \mathsf{U}_{g}$. Therefore, by \cref{prop: prim-count,prop: pol-uni-set-count}, we have
\begin{align*}
    \#\mathsf{B}_g = \# \mathsf{P}_{g} - \#\mathsf{U}_{g} &= \sum_{d\in \mathscr{D}(g)}\mu(d)\, 2^{g/d} - \sum_{d\in \mathscr{D}_{\odd}(g/2)}\mu(d)\, 2^{g/2d}\\[1em] &= \sum_{d\in \mathscr{D}(g)}\mu(d)\, 2^{g/d} - \sum_{d\in \mathscr{D}_{\odd}(g)}\mu(d)\, 2^{g/2d},
\end{align*}
where the latter equality follows since $g$ being even means $g$ and $g/2$ have the same odd divisors.
\end{proof}

\subsection{Indecomposable Unitary \texorpdfstring{$\boldsymbol{\bt}$}{} Modules}
For integers $q>0$ and $0 \leq b\leq q$, recall $\widetilde{\uni}(q-b,b)$ and $\widetilde{\bi}(q-b,b)$, the sets of isomorphism classes of unitary unicycles and unitary bicycles in signatures $(q-b,b)$, and $\widetilde{\uni}_{\serre}(q/2,q/2)$ and $\widetilde{\bi}_{\serre}(q/2,q/2)$, the sets of isomorphism classes of Serre unicycles and unitary bicycles. In this section, we determine their cardinalities.

\subsubsection{Unitary Unicycles} Throughout this section, assume $q>0$ to be odd. Of sole interest is $\widetilde{\uni}(q-b,b)$, the set of isomorphism classes of unitary unicycles. Before we focus on a fixed signature, to ease into our computations, let us first consider the set $\widetilde{\uni}(q)$, which collects the isomorphism classes of all unitary unicycles of $\kk$-dimension $2q$. 

By definition, $\widetilde{\uni}(q)$ is in bijection with the $\tau^2$-orbits on $\mathsf{U}_{2q}$, the set of self-dual words of length $2q$. The operator $\tau^2$ acts on $\mathsf{U}_{2q}$ freely with exact order $q$ and therefore
\[\#\widetilde{\uni}(q) = \frac{1}{q}\# \mathsf{U}_{2q}.\]
Since $q$ is odd, as a consequence of \cref{prop: pol-uni-set-count}, we obtain the following corollary. 

\begin{corollary}\label{cor: uni-uni-count}
    One has $\displaystyle\#\widetilde{\uni}(q) = \frac{1}{q}\sum_{d\in \mathscr{D}(q)}\mu(d)\, 2^{q/d}$.
\end{corollary}
The counts in \cref{cor: poli-uni-count,cor: uni-uni-count} illustrate the fact that there are two ways to give a unitary structure to a unicycle of dimension $2q$, for $q$ odd, either with signature $(q - b,b)$ or $(b,q-b)$.

The following notion will be necessary to determine $\#\widetilde{\uni}(q-b,b)$.
\begin{definition} \label{def: bodd}
    Consider a word of the form $ww^\star$, not necessarily primitive, where $\ell = \ell(w)$ is odd. Writing $w = u_{\ell-1}\cdots u_1u_0$, define
    \[\delta_1(ww^\star) \coloneqq \#\{1\leq j \leq \lfloor \ell/2 \rfloor \mid u_{2j-1}=\word{v} \} + \#\{0\leq j \leq \lfloor \ell/2 \rfloor \mid u_{2j}=\word{f}\}.\]
\end{definition}

For a fixed $0 \leq b \leq q$, define
\[\mathsf{U}^{\delta_1 = b}_{2q} \coloneqq \setp{yy^\star \in \mathsf{U}_{2q}}{\delta_1(yy^\star) = b}.\]

By \cref{thm: structure of unitary BT1s,rmk: unicycle signature}, the set $\widetilde{\uni}(q-b,b)$ is in bijection with the $\tau^2$-orbits on $\mathsf{U}^{\delta_1 = b}_{2q}$. The free action of $\tau^2$, on the set, with exact order $q$ gives us
\[\#\widetilde{\uni}(q-b,b) = \frac{1}{q}\# \mathsf{U}^{\delta_1 = b}_{2q}.\]

For any $\ell >0$ odd and $c \geq 0$, define $\mathsf{W}_{\ell}(c) = \setp{w \in \mathsf{W}_\ell}{\delta_1(ww^\star) = c}$.

\begin{lemma}\label{lemma: word-count-odd-wt}
    One has $\displaystyle\# \mathsf{W}_{\ell}(c) = \binom{\ell}{c}.$
\end{lemma}
\begin{proof}
    Consider the word $w_0=\word{fvfv}\cdots \word{fvf} \in \mathsf{W}_{\ell}$ and observe that $\delta_1(w_0 w_0^\star)=0$. Constructing a word $w\in \mathsf{W}_{\ell}(c)$ amounts to choosing $c$ positions of $w_0$ to invert a letter, for which there are $\displaystyle\binom{\ell}{c}$ possibilities.
\end{proof}

Using the bijection $\Phi_1$ in \cref{lemma: dual-word-bij}, one readily sees $\mathsf{U}^{\delta_1 = b}_{2q}$ is in bijection with
\[\mathsf{Y}_{q}(b) \coloneqq \mathsf{Y}_q \cap \mathsf{W}_{q}(b) = \{y \in \mathsf{Y}_q\ \vert\ \delta_1(yy^\star) = b\}.\]

\begin{proposition}\label{prop: uni-uni-set-count}
One has $\displaystyle \# \mathsf{U}^{\delta_1 = b}_{2q} = \sum_{d\in \mathscr{D}(\gcd(q,b))}\mu(d) \binom{q/d}{b/d}$.
\end{proposition}

\begin{proof}
By \cref{rmk: dual-prim-partition}, the set $\mathsf{W}_{q}(b)$ may be partitioned into subsets \[\mathsf{Y}_{q,2d}(b) \coloneqq \mathsf{Y}_{q,2d} \cap \mathsf{W}_{q}(b) = \{z \in \mathsf{Y}_{q,2d}\ \vert\ \delta_1(zz^\star) = b\}\]
for $d\mid q$. Note that $q/d$ is odd since $q$ is odd. 

Recall the bijection $\Phi_0$ from \cref{lemma: dual-word-bij}. For $z \in \mathsf{Y}_{q,2d}(b)$, the element $h = \Phi_0^{-1}(z) \in \mathsf{Y}_d$ is such that $zz^\star = (hh^\star)^{q/d}$. Therefore,
\[b = \delta_1(zz^\star) = \frac{q}{d}\cdot \delta_1(hh^\star),\]
giving $\delta_1(hh^\star) = bd/q$, which allows us to conclude that $(q/d)\mid b$ and that $h \in \mathsf{Y}_{d}(bd/q)$. 

It is then straightforward to see that 
\[\Phi_0^{-1}:\mathsf{Y}_{q,2d}(b) \to \mathsf{Y}_{d}(bd/q)\]
is a bijection. Composing with $\Phi_1$, we obtain the bijection $\mathsf{Y}_{q,2d}(b) \cong \mathsf{U}_{2d}^{\delta_1 = bd/q}$.

Define $\mathscr{D}_b(q) = \setp{d \in \mathscr{D}(q)}{(q/d)\mid b}$. Thus, 
\[\mathsf{W}_q(b) = \bigsqcup_{d\in \mathscr{D}_b(q)}\mathsf{Y}_{q,2d}(b) \cong  \bigsqcup_{d\in \mathscr{D}_b(q)}\mathsf{U}_{2d}^{\delta_1 = bd/q},\quad \text{and therefore } \#\mathsf{W}_q(b) = \sum_{d\in \mathscr{D}_b(q)}\#\mathsf{U}_{2d}^{\delta_1 = bd/q}.\]
By \cref{lemma: mobius-count,lemma: word-count-odd-wt}, we obtain
    \[\# \mathsf{U}_{2q}^{\delta_1 = b} = \sum_{d\in \mathscr{D}_{b}(q)}\mu(q/d)\, \# \mathsf{W}_{q}(bd/q) = \sum_{d\in \mathscr{D}_{b}(q)}\mu(q/d)\, \binom{d}{bd/q}.\]
    Observe now that
\[\mathscr{D}_{\textnormal{co-}b}(q) = \setp{d \in \mathscr{D}(q)}{q/d \in \mathscr{D}_b(q)} = \mathscr{D}(q) \cap \mathscr{D}(b) = \mathscr{D}(\gcd(q,b)),\]
and therefore, by \cref{rmk: co-propert divisors}, we obtain
    \[\# \mathsf{U}_{2q}^{\delta_1 = b} = \sum_{d\in \mathscr{D}(\gcd(q,b))}\mu(d) \binom{q/d}{b/d}.\qedhere\]
\end{proof}

\begin{corollary}\label{cor: uni-uni-count-final}
    One has \[\displaystyle \#\widetilde{\uni}(q-b,b) = \frac{1}{q}\sum_{d\in \mathscr{D}(\gcd(q,b))}\mu(d) \binom{q/d}{b/d}.\]
\end{corollary}

\subsubsection{Serre Unicycles \& Bicycles}

Let $q > 0$ be even. For $q \equiv 0 \bmod{4}$, of interest is $\widetilde{\uni}_{\serre}(q/2,q/2)$, the set of isomorphism classes of Serre unicycles; and for $q \equiv 2 \bmod{4}$, of interest is $\widetilde{\bi}_{\serre}(q/2,q/2)$, the set of isomorphism classes of Serre bicycles. 

Recall the notion of the Serre tensor construction of a polarized $\bt$ modules from \cref{section:STC}.

\begin{proposition}\label{prop: uni-serre-count}
    One has 
    \begin{align*}
        \#\widetilde{\uni}_{\serre}(q/2,q/2) &= \frac{1}{q}\sum_{d\in\mathscr{D}_{\odd}(q)}\mu(d) 2^{q/2d}, \qquad \text{when $q \equiv 0 \bmod{4}$,}\\[0.65em]
        \#\widetilde{\bi}_{\serre}(q/2,q/2) &= \frac{1}{q}\sum_{d\in\mathscr{D}(q/2)}\mu(d) 2^{q/2d}, \qquad \text{when $q \equiv 2 \bmod{4}$.}
    \end{align*}
\end{proposition}
\begin{proof}
    By \cref{lemma: stc-inj} and \cref{prop: STC unicycle,prop: STC bicycle}$, \widetilde{\uni}_{\serre}(q/2,q/2)$ and $\widetilde{\bi}_{\serre}(q/2,q/2)$ are in bijection with $\widehat{\uni}(q/2)$ and $\widehat{\bi}(q/2)$, respectively, via the Serre tensor construction. 

    Assume $q\equiv 0 \bmod{4}$, i.e., $q/2$ is even. The result for $\widetilde{\uni}_{\serre}(q/2,q/2)$ follows from \cref{cor: poli-uni-count} along with the observation that $q$ and $q/2$ have the same odd divisors. Assume $q\equiv 2 \bmod{4}$, i.e., $q/2$ is odd. The result for $\widetilde{\bi}_{\serre}(q/2,q/2)$ follows from \cref{prop: pol-bi-set-count}.
\end{proof}

\subsubsection{Unitary Bicycles I. Counting by Dimension}
Assume $q > 0$ to be even, and recall $\widetilde{\bi}(q)$, the set of isomorphism classes of all unitary bicycles of $\kk$-dimension $2q$.

By \cref{rmk: don't rotate bicycles!}, $\widetilde{\bi}(q)$ is in bijection with the orbits on $\mathsf{B}_q$, the set of primitive non-self-dual words of length $q$, under the action $w \mapsto \tau(w)^\star$.
This is a free action on $\mathsf{B}_{q}$ with exact order $q$ and therefore
\[\#\widetilde{\bi}(q) = \frac{1}{q}\# \mathsf{B}_{q}.\]
As a consequence of \cref{prop: pol-bi-set-count}, since $q$ is even, we obtain the following corollary.

\begin{corollary}\label{cor: uni-bi-count}
    One has  $\displaystyle\#\widetilde{\bi}(q) =  \frac{1}{q}\left(\sum_{d\in \mathscr{D}(q)}\mu(d)\, 2^{q/d} - \sum_{d\in \mathscr{D}_{\odd}(q)}\mu(d)\, 2^{q/2d}\right).$
\end{corollary}

\subsubsection{Counting Indecomposables by Dimension}
Recall the set $\widetilde{\ind}(q)$ for any integer $q>0$, which collects the isomorphism classes of all indecomposable unitary $\bt$ modules of $\kk$-dimension $2q$. We have,
\[\widetilde{\ind}(q) = \begin{cases}
    \widetilde{\uni}(q) & \text{when $q$ odd,}\\[1em]
    \widetilde{\bi}(q) \sqcup \widetilde{\uni}_{\serre}(q/2,q/2) & \text{when $q \equiv 0 \bmod{4}$,}\\[1em]
    \widetilde{\bi}(q) \sqcup  \widetilde{\bi}_{\serre}(q/2,q/2) & \text{when $q \equiv 2 \bmod{4}$.}
\end{cases}
\]

\begin{theorem}\label{thm: uni-indecomp-count}
    One has $\displaystyle\#\widetilde{\ind}(q) = \frac{1}{q}\sum_{d\in \mathscr{D}(q)}\mu(d)\, 2^{q/d}$.
\end{theorem}
\begin{proof}
When $q$ is odd, the statement is exactly the statement of \cref{cor: uni-uni-count}.  When $q \equiv 0 \bmod{4}$, the statement follows from \cref{prop: uni-serre-count,cor: uni-bi-count}. Assume $q \equiv 2 \bmod{4}$. Using \cref{prop: uni-serre-count,cor: uni-bi-count}, one has
\begin{align*}
    \# \widetilde{\ind}(q) &= \# \widetilde{\bi}(q) + \# \widetilde{\bi}_{\serre}(q/2,q/2)\\[0.75em]
    &= \frac{1}{q}\left(\sum_{d\in \mathscr{D}(q)}\mu(d)\, 2^{q/d} - \sum_{d\in \mathscr{D}_{\odd}(q)}\mu(d)\, 2^{q/2d}\right) + \frac{1}{q}\sum_{d\in \mathscr{D}(q/2)}\mu(d) \, 2^{q/2d} = \frac{1}{q}\sum_{d\in \mathscr{D}(q)}\mu(d)\, 2^{q/d},
\end{align*}
since $q$ is even and $q/2$ is odd and therefore $\mathscr{D}_{\odd}(q) = \mathscr{D}_{\odd}(q/2) = \mathscr{D}(q/2)$.
\end{proof}

\subsubsection{Unitary Bicycles II. Counting by Signature}
Assume $q>0$ to be even and we turn our attention to $\widetilde{\bi}(q-b,b)$, the set of isomorphism classes of unitary bicycles with signature $(q-b,b)$. That is, we describe how unitary bicycles in $\widetilde{\bi}(q)$ are distributed across signatures.

The following notion will be necessary to determine our counts in this section.
\begin{definition}
    Consider a word $w = u_{\ell-1}\cdots u_1u_0$, where $\ell = \ell(w)$ is even. Define
    \[\delta_0(w) = \# \{0\leq j < \ell/2 \; | \; u_{2j}=\word{v}\} + \#\{ 1\leq j \leq \ell/2 \; | \; u_{2j-1}=\word{f}\}.\]
\end{definition}

Define $\mathsf{W}_\ell^{\delta_0 = c} \coloneqq  \setp{w \in \mathsf{W}_\ell}{\delta_0(w) = c}$, where $\ell > 0$ is even and $c\geq 0$. Similar to \cref{lemma: word-count-odd-wt}, we obtain the following lemma.

\begin{lemma}\label{lemma: word-count-even-wt}
    One has $\displaystyle\# \mathsf{W}^{\delta_0 = c}_{\ell,c} = \binom{\ell}{c}.$
\end{lemma}

For a fixed $0 \leq b \leq q$, define
\begin{align*}
    \mathsf{B}^{\delta_0 = b}_{q} &\coloneqq \mathsf{B}_q \cap \mathsf{W}^{\delta_0 = b}_{q}; & \mathsf{P}^{\delta_0 = b}_{q} &\coloneqq \mathsf{P}_q \cap \mathsf{W}^{\delta_0 = b}_{q};\text{ and} & \mathsf{U}^{\delta_0 = b}_{q} &\coloneqq \mathsf{U}_q \cap \mathsf{W}^{\delta_0 = b}_{q}
\end{align*}
where $\mathsf{B}_q,\mathsf{P}_q$ and $\mathsf{U}_q$ are as in \cref{def: basic-comb-set}. Therefore, we have $\mathsf{B}^{\delta_0 = b}_{q} = \mathsf{P}^{\delta_0 = b}_{q} \setminus \mathsf{U}^{\delta_0 = b}_{q}$.

By \cref{thm: structure of unitary BT1s,rmk: bicycle signature}, the set $\widetilde{\bi}(q-b,b)$ is in bijection with the orbits on $\mathsf{B}^{\delta_0 = b}_{q}$ under the operator $\tau_\star: w \mapsto \tau(w)^\star$; observe that $\tau_{\star}^2 = \tau^2$. Primitivity and non-self-duality of elements in $\mathsf{B}^{\delta_0 = b}_{q}$ allow us to conclude the action of $\tau_\star$ is free with exact order $q$. This gives us
\[\#\widetilde{\bi}(q-b,b) = \frac{1}{q}\# \mathsf{B}^{\delta_0 = b}_{q} = \frac{1}{q}(\# \mathsf{P}^{\delta_0 = b}_{q} - \# \mathsf{U}^{\delta_0 = b}_{q}).\]

In following propositions we determine $\# \mathsf{P}^{\delta_0 = b}_{q}$ and $\# \mathsf{U}^{\delta_0 = b}_{q}$.

\begin{proposition}\label{prop: uni-bi-prim=count}
    For $q$ even, one has 
    \[\# \mathsf{P}^{\delta_0 = b}_{q} = \begin{cases}
        \displaystyle\sum_{d \in \mathscr{D}_{\textnormal{co-}\even}(q)}\mu(d) \binom{q/d}{q/2d} + \sum_{d \in \mathscr{D}_{\textnormal{co-}\odd}(q)}\mu(d) 2^{q/d} & \text{when $b = q/2$,}\\[2em]
        \displaystyle\sum_{d\in\mathscr{D}_{\textnormal{w.co-}\even}(\gcd(q,b))}\mu(d) \binom{q/d}{b/d} & \text{when $b \neq q/2$,}
    \end{cases}\]
    where $\mathscr{D}_{\textnormal{w.co-}\even}(\gcd(q,b)) := \setp{d\in \mathscr{D}(\gcd(q,b))}{q/d\text{ is even}}$ and \emph{w.co-even} stands for ``weakly co-even''.
\end{proposition}

\begin{proof}
    By \cref{rmk: prim-partition}, the set $\mathsf{W}^{\delta_0 = b}_{q}$ may be partitioned into following subsets, indexed by $d\mid q$,
\[\mathsf{P}^{\delta_0 = b}_{q,d} = \mathsf{P}_{q,d} \cap \mathsf{W}^{\delta_0 = b}_{q} = \{z \in \mathsf{P}_{q,d}\ \vert\ \delta_0(z) = b\}.\]

We first focus on the contribution made by $\mathsf{P}^{\delta_0 = b}_{q,d}$ to the set $\mathsf{W}^{\delta_0 = b}_{q}$ when $d$ is odd. Fix $d \in \mathscr{D}_{\odd}(q)$.

We will show that $\mathsf{P}_{q,d} = \mathsf{P}_{q,d}^{\delta_0 = q/2}$, i.e., $\delta_0(z) = q/2$ for any $z \in \mathsf{P}_{q,d}$, and therefore $\mathsf{P}_{q,d}^{\delta_0 = b} = \emptyset$ when $b \neq q/2$.

Consider $z \in \mathsf{P}_{q,d}$, i.e., $z$ is a word of length $q$ with $\tau$-period $d$, and recall the bijection $\Psi$ from \cref{rmk: prim-partition}. The element $h = \Psi^{-1}(z) \in \mathsf{P}_d$, the set of primitive words of length $d$, is such that $z = h^{q/d}$. Since $d$ is odd and $q$ is even, necessarily $q/d$ is even and we may write
\[z = (h^2)^{q/2d}\]
Since $\ell(h) = d$ is odd, one may readily deduce from the definition of $\delta_0$ 
that $\delta_0(h^2) = d$ and hence
\[\delta_0(z) = \frac{q}{2d}\cdot \delta_0(h^2) = \frac{q}{2d}\cdot d = \frac{q}{2}.\]
Thus, $z \in \mathsf{P}_{q,d}^{\delta_0 = q/2}$, and we conclude $\mathsf{P}_{q,d} = \mathsf{P}_{q,d}^{\delta_0 = q/2}$ and, via $\Psi$, $\mathsf{P}_{q,d}^{\delta_0 = q/2} \cong \mathsf{P}_d$.

Therefore, for any $d \in \mathscr{D}_{\odd}(q)$, we conclude
\[\# \mathsf{P}^{\delta_0 = b}_{q,d} = \begin{cases}
    0 & \text{if $b \neq q/2$,}\\[0.5em]
    \#\mathsf{P}_d & \text{if $b = q/2$.}
\end{cases}\]

Assume now $b \neq q/2$. Our arguments above tell us that no odd divisors $d$ of $q$ contribute in this case. 

Therefore, consider $\mathsf{P}^{\delta_0 = b}_{q,d}$ for $d \in \mathscr{D}_{\even}(q)$. An almost exact argument as the one given in \cref{prop: uni-uni-set-count} for $\Phi_0^{-1}$ allows us to conclude that $\Psi^{-1}$ from \cref{rmk: prim-partition} induces the bijection $\mathsf{P}^{\delta_0 = b}_{q,d} \cong \mathsf{P}_d^{\delta_0 = bd/q}$, the set of primitive words in $\mathsf{W}_d^{\delta_0 = bd/q}$, and necessarily $(q/d)\mid b$. Hence,

\[\mathsf{W}^{\delta_0 = b}_{q} = \bigsqcup_{d \in \mathscr{D}_{b,\even}(q)}\mathsf{P}^{\delta_0 = b}_{q,d} \cong  \bigsqcup_{d \in \mathscr{D}_{b,\even}(q)}\mathsf{P}_d^{\delta_0 = bd/q,}\]
where $\mathscr{D}_{b,\even}(q) = \setp{d \in \mathscr{D}(q)}{d \text{ is even and }(q/d)\mid b}$. Therefore,
By \cref{lemma: mobius-count,lemma: word-count-even-wt}, we obtain
    \[\# \mathsf{P}^{\delta_0 = b}_{q} = \sum_{d\in \mathscr{D}_{b,\even}(q)}\mu(q/d)\, \# \mathsf{W}^{\delta_0 = bd/q}_{d} = \sum_{d\in \mathscr{D}_{b,\even}(q)}\mu(q/d)\, \binom{d}{bd/q}.\]
    Observe now that
\[\mathscr{D}_{\textnormal{co-}b,\even}(q) = \setp{d \in \mathscr{D}(q)}{q/d \in \mathscr{D}_{b,\even}(q)} = \setp{d \in \mathscr{D}(\gcd(q,b))}{q/d\text{ is even}} \eqqcolon \mathscr{D}_{\text{w.co-}\even}(\gcd(q,b)).\]

Hence, when $b \neq q/2$, by \cref{rmk: co-propert divisors}, we obtain 
    $\displaystyle\# \mathsf{P}^{\delta_0 = b}_{q} = \sum_{d\in \mathscr{D}_{\text{w.co-}\even}(\gcd(q,b))}\mu(d) \binom{q/d}{b/d}$.

Assume now $b = q/2$. In the case when $d \in \mathscr{D}_{\odd}(q)$, we have already argued that $\#\mathsf{P}^{\delta_0 = q/2}_{q,d} = \# \mathsf{P}_d$.

Consider any $z \in \#\mathsf{P}^{\delta_0 = q/2}_{q,d}$ for $d$ even, and recall once again the bijection $\Psi$ from \cref{rmk: prim-partition}. The element $h = \Psi^{-1}(z) \in \mathsf{P}_d$ is such that $z = h^{q/d}$. Since $\ell(h) = d$ is even, one may consider $\delta_0(h)$ and observe that
\[\frac{q}{2} = \delta_0(z) = \frac{q}{d}\cdot \delta_0(h).\]
Therefore $h \in \mathsf{P}_d^{\delta_0 = d/2}$ and hence one deduces that $\Psi$ induces the bijection $\mathsf{P}^{\delta_0 = d/2}_{d} \cong \mathsf{P}^{\delta_0 = q/2}_{q,d}$. 

Thus, 
\begin{align*}
\mathsf{W}^{\delta_0 = q/2}_{q} &= \bigsqcup_{d \in \mathscr{D}(q)}\mathsf{P}^{\delta_0 = q/2}_{q,d} \cong \bigsqcup_{d \in \mathscr{D}_{\even}(q)}\mathsf{P}^{\delta_0 = d/2}_{d}\; \sqcup \bigsqcup_{d \in \mathscr{D}_{\odd}(q)}\mathsf{P}_{d}.
\end{align*}

We will employ the classical M{\"o}bius inversion on the following two arithmetic functions defined on $\mathscr{D}(q)$:
\begin{align*}
    \kappa_{\mathsf{W}}(d) &= \begin{cases}
       \# \mathsf{W}^{\delta_0 = d/2}_{d} & \text{if $d$ is even,}\\[0.25em] 
       \# \mathsf{W}_{d} & \text{if $d$ is odd,} 
    \end{cases} &
    \kappa_{\mathsf{P}}(d) &= \begin{cases}
       \# \mathsf{P}^{\delta_0 = d/2}_{d} & \text{if $d$ is even,}\\[0.25em] 
       \# \mathsf{P}_{d} & \text{if $d$ is odd.} 
    \end{cases}
\end{align*}
Taking cardinalities in the bijection above gives us $\displaystyle\kappa_{\mathsf{W}}(q) = \sum_{d \in \mathscr{D}(q)} \kappa_{\mathsf{P}}(d)$. By M{\"obius} inversion, we obtain
\[\# \mathsf{P}^{\delta_0 = q/2}_{q} = \kappa_{\mathsf{P}}(q) = \sum_{d \in \mathscr{D}(q)} \mu(q/d)\,\kappa_{\mathsf{W}}(d) =  \sum_{d \in \mathscr{D}_{\even}(q)} \mu(q/d)\,\# \mathsf{W}^{\delta_0 = d/2}_{d} + \sum_{d \in \mathscr{D}_{\odd}(q)} \mu(q/d)\,\#\mathsf{W}_{d}.\]
By \cref{lemma: word-count,lemma: word-count-even-wt}, we have
\[\# \mathsf{P}^{\delta_0 = q/2}_{q} = \sum_{d \in \mathscr{D}_{\even}(q)} \mu(q/d)\, \binom{d}{d/2} + \sum_{d \in \mathscr{D}_{\odd}(q)} \mu(q/d)\,2^d.\]
Therefore, by \cref{rmk: co-propert divisors}, we obtain
    \[\# \mathsf{P}^{\delta_0 = q/2}_{q} = \sum_{d \in \mathscr{D}_{\text{co-}\even}(q)} \mu(d)\, \binom{q/d}{q/2d} + \sum_{d \in \mathscr{D}_{\text{co-}\odd}(q)} \mu(d)\,2^{q/d}.\qedhere\]
\end{proof}

\begin{proposition}\label{prop: uni-bi-self-dual=count}
$\# \mathsf{U}^{\delta_0 = b}_{q} = 0$ when $q \equiv 0 \bmod{4}$ and $b\neq q/2$ or when $q \equiv 2\bmod{4}$ and $b$ odd, and
    \[\# \mathsf{U}^{\delta_0 = b}_{q} = \begin{cases}
        \displaystyle\sum_{d \in \mathscr{D}_{\odd}(q)} \mu(d)\, 2^{q/2d} & \text{when $q \equiv 0 \bmod{4}$ and $b = q/2$},\\[2em]
        \displaystyle\sum_{d \in \mathscr{D}_{\odd}(\gcd(q,b))} \mu(d)\, \binom{q/2d}{b/2d} & \text{when $q \equiv 2 \bmod{4}$ and $b$ is even}.
    \end{cases}\]
\end{proposition}

\begin{proof}
Consider the ambient set $\mathsf{U}_q$, the set of primitive self-dual words $yy^\star$ with $\ell(y) = q/2$.

Assume $q \equiv 0 \bmod{4}$, i.e., $q/2$ is even. Therefore any word $yy^\star \in \mathsf{U}_q$ is such that $\ell(y) = q/2$ is even. By \cref{thm: structure of unitary BT1s,rmk: bicycle signature} the $\tau$-orbits on $\mathsf{U}_q$ are in bijection with $\widetilde{\uni}_{\serre}(q/2,q/2)$ and $\mathsf{U}_q = \mathsf{U}_q^{\delta_0 = q/2}$. The free action of $\tau$ on $\mathsf{U}_q$ with exact order $q$ gives us
\[\#\widetilde{\uni}_{\serre}(q/2,q/2) = \frac{1}{q}\# \mathsf{U}_q = \frac{1}{q}\#\mathsf{U}^{\delta_0 = q/2}_{q}.\]
Therefore, by \cref{prop: uni-serre-count}, we obtain
    \[\# \mathsf{U}_q^{\delta_0 = q/2} = \#\mathsf{U}_q  = q\cdot \#\widetilde{\uni}(q/2,q/2) = \sum_{d \in \mathscr{D}_{\odd}(q)} \mu(d)\, 2^{q/2d},\]
and necessarily $\#\mathsf{U}_{q}^{\delta_0 = b} = \emptyset$ for $b \neq q/2$.

Assume now $q \equiv 2 \bmod{4}$, i.e., $q/2$ is odd. Therefore any word $yy^\star \in \mathsf{U}_q$ with $\ell(y) = q/2$ odd. For such words, we may consider both $\delta_1(yy^\star)$ and $\delta_0(yy^\star)$. Inspecting both definitions, one may verify that
\[\delta_0(yy^\star) = 2\cdot \delta_1(yy^\star).\]
Therefore, the set $\mathsf{U}_q^{\delta_0 = b} = \emptyset$ when $b$ is odd and $\mathsf{U}_q^{\delta_0 = b} \cong \mathsf{U}_q^{\delta_1 = b/2}$ when $b$ is even. 

Hence, when $b$ is even, by \cref{prop: uni-uni-set-count}, we obtain
\[\#\mathsf{U}^{\delta_0 = b}_{q} = \#\mathsf{U}^{\delta_1 = b/2}_{q} = \sum_{d\in \mathscr{D}(\gcd(q/2,b/2))}\mu(d) \binom{q/2d}{b/2d}.\]
Since $q/2$ is odd, we have $\mathscr{D}(\gcd(q/2,b/2)) = \mathscr{D}_{\odd}(\gcd(q,b))$, allowing us to conclude
\[\#\mathsf{U}^{\delta_0 = b}_{q} = \sum_{d\in \mathscr{D}_{\odd}(\gcd(q,b))}\mu(d) \binom{q/2d}{b/2d}.\qedhere\]
\end{proof}

\begin{corollary}\label{cor: uni-bi-count-final}
One has $\displaystyle\#\widetilde{\mathsf{\bi}}(q-b,b)= \frac{1}{q}(\# \mathsf{P}^{\delta_0 = b}_{q} - \# \mathsf{U}^{\delta_0 = b}_{q})$, giving
\[\#\widetilde{\mathsf{\bi}}(q/2,q/2) = \begin{cases}
        \displaystyle\frac{1}{q}\left(\sum_{d \in \mathscr{D}_{\textnormal{co-}\even}(q)}\mu(d) \binom{q/d}{q/2d} + \sum_{d \in \mathscr{D}_{\textnormal{co-}\odd}(q)}\mu(d)\, 2^{q/d} - \sum_{d \in \mathscr{D}_{\odd}(q)}\mu(d)\, 2^{q/2d} \right) & \text{if $q \equiv 0 \bmod{4}$,}\\[1.5em]
        \displaystyle\frac{1}{q}\left(\sum_{d \in \mathscr{D}_{\textnormal{co-}\even}(q)}\mu(d) \binom{q/d}{q/2d} + \sum_{d \in \mathscr{D}_{\textnormal{co-}\odd}(q)}\mu(d)\, 2^{q/d}\right) & \text{if $q \equiv 2 \bmod{4}$;}
    \end{cases} \]
    and, for $b \neq q/2$,
    \[
\#\widetilde{\bi}(q-b,b) = \begin{cases} \displaystyle \frac{1}{q} \left(\sum_{d\in\mathscr{D}_{\textnormal{w.co-}\even}(\gcd(q,b))}\mu(d) \binom{q/d}{b/d} - \sum_{d \in \mathscr{D}_{\odd}(\gcd(q,b))} \mu(d)\, \binom{q/2d}{b/2d}\right) &\hbox{if $q\equiv 2 \bmod 4$ and $b$ is even} \\[1.5em]
\displaystyle \frac{1}{q} \sum_{d\in\mathscr{D}_{\textnormal{w.co-}\even}(\gcd(q,b))}\mu(d) \binom{q/d}{b/d} &\hbox{else.}
\end{cases}
\]
\end{corollary}
\begin{proof}
The expressions are obtained from the formulae for $\# \mathsf{P}^{\delta_0 = b}_{q}$ and $\# \mathsf{U}^{\delta_0 = b}_{q}$ in \cref{prop: uni-bi-prim=count,prop: uni-bi-self-dual=count} respectively. We justify the expression for $\#\widetilde{\mathsf{\bi}}(q/2,q/2)$. By \cref{prop: uni-bi-prim=count}, we have
    \[\# \mathsf{P}^{\delta_0 = q/2}_{q}=\sum_{d \in \mathscr{D}_{\textnormal{co-}\even}(q)}\mu(d) \binom{q/d}{q/2d} + \sum_{d \in \mathscr{D}_{\textnormal{co-}\odd}(q)}\mu(d)\, 2^{q/d};\]
    while, by \cref{prop: uni-bi-self-dual=count}, we have
    \[\# \mathsf{U}^{\delta_0 = q/2}_{q} = \begin{cases}
        \displaystyle\sum_{d \in \mathscr{D}_{\odd}(q)}\mu(d)\, 2^{q/2d} & \text{if $q \equiv 0 \bmod{4}$,}\\[0.25em]
        0 & \text{if $q \equiv 2 \bmod{4}$}.
    \end{cases}\]
    The latter expression for $q \equiv 2 \bmod{4}$ is obtained since $b = q/2$ is odd.

    Since $\displaystyle\#\widetilde{\mathsf{\bi}}(q/2,q/2)= \frac{1}{q}(\#\mathsf{P}^{\delta_0 = q/2}_{q} - \#\mathsf{U}^{\delta_0 = q/2}_{q})$, we obtain the given expression.
\end{proof}

\subsection{Some Formulae \& Tables of Values}

\begin{example}
    In this example, we record $\#\widetilde{\ind}(q-b,b)$ for $1 \leq b \leq 3$. We perform this computation by applying the formulae in \cref{cor: uni-uni-count-final,cor: uni-bi-count-final,prop: uni-serre-count}, along with the observation on which indecomposable types occur in which signature (\cref{rmk: sign-fix-indecomp}). 
    
    The outputs depend on some congruence conditions on $q$ with respect to $b$. Packaging the various possible cases using the floor function to, we obtain:
    \begin{align*}
        \#\widetilde{\ind}(q-1,1) &= 1, & \#\widetilde{\ind}(q-2,2) &= \left\lfloor \frac{q-1}{2} \right\rfloor,  & \#\widetilde{\ind}(q-3,3) &= \left\lfloor \frac{(q-1)(q-2)}{6} \right\rfloor.
    \end{align*}
\end{example}

Taking duals of words $w \mapsto w^\star$ induces a bijection $\widetilde{\ind}(q-b,b) \cong \widetilde{\ind}(b,q-b)$. We record in \cref{table:unitary-uni-bi-list,table:parallel-sig-list} the numerical value of $\#\widetilde{\ind}(q-b,b)$ for $2 \leq q \leq 20$ and $1 \leq b \leq q/2$. In the case $q = 2$, the only indecomposable unitary $\bt$ module occurs in signature $(1,1)$.

\begin{table}[h!]
\renewcommand{\arraystretch}{1.4}

\begin{footnotesize}
\centering
\begin{tabular}{|*{10}{>{\centering\arraybackslash}m{0.75cm}|}  } 
\hline
\diagbox[height=1.75\baselineskip]{$q$}{$b$} & $1$ & $2$ & $3$ & $4$ & $5$ & $6$ & $7$ & $8$ & $9$ \\
\hline
$3$ & $1$ & \none & \none & \none & \none & \none & \none & \none & \none 
 \\
 \hline
$4$ & $1$ & \none & \none & \none & \none &  \none & \none & \none & \none 
 \\
\hline
$5$ & $1$ & $2$ & \none & \none & \none & \none & \none & \none & \none 
 \\
 \hline
$6$ & $1$ & $2$ & \none & \none & \none &  \none & \none & \none & \none 
 \\
\hline
$7$ & $1$ & $3$ & $5$ & \none & \none & \none & \none & \none & \none 
 \\
\hline
$8$ & $1$ & $3$ & $7$ & \none & \none & \none & \none&  \none & \none 
 \\
\hline
$9$ & $1$ & $4$ & $9$ & $14$ & \none & \none & \none & \none & \none 
 \\
 \hline
 $10$ & $1$ & $4$ & $12$ & $20$ & \none& \none & \none  &  \none & \none 
 \\
\hline
$11$ & $1$ & $5$ & $15$ & $30$ & $42$ & \none & \none & \none & \none \\
 \hline
 $12$ & $1$ & $5$ & $18$ & $40$ & $66$ & \none & \none & \none & \none \\
\hline
$13$ & $1$ & $6$ & $22$ & $55$ & $99$ & $132$ & \none & \none & \none \\
 \hline
$14$ & $1$ & $6$ & $26$ & $70$ & $143$ & $212$ & \none &  \none & \none \\
\hline
$15$ & $1$ & $7$ & $30$ & $91$ & $200$ & $333$ & $429$ & \none & \none  \\
\hline
$16$ & $1$ & $7$ & $35$ & $112$ & $273$ & $497$ & $715$ & \none & \none  \\
\hline
$17$ & $1$ & $8$ & $40$ & $140$ & $364$ & $728$ & $1144$ & $1430$ & \none \\
 \hline
 $18$ & $1$ & $8$ & $45$ & $168$ & $476$ & $1026$ & $1768$ & $2424$& \none   \\
\hline
$19$ & $1$ & $9$ & $51$ & $204$ & $612$ & $1428$ & $2652$ & $3978$ & $4862$  \\
 \hline
$20$ & $1$ & $9$ & $57$ & $240$ & $775$ & $1932$ & $3876$ & $6288$ & $8398$ \\
\hline
\end{tabular}
\end{footnotesize}
\caption{$\#\widetilde{\ind}(q-b,b)$ for $3 \leq q \leq 20$ and $1 \leq b < q/2$.}
\label{table:unitary-uni-bi-list}
\end{table}

By \cref{rmk: sign-fix-indecomp}, the indecomposable types captured in $\widetilde{\ind}(q-b,b)$ for $b < q/2$ are unitary unicycles, when $q$ is odd, and unitary bicycles, when $q$ is even.

The indecomposable types captured in $\widetilde{\ind}(q/2,q/2)$ are unitary bicycles, Serre unicycles (when $q/2$ is even) and Serre bicycles (when $q/2$ is odd).

\begin{table}[h!]
\renewcommand{\arraystretch}{1.4}
\begin{footnotesize}
\begin{center}
\begin{tabular}{|>{\centering\arraybackslash}m{2cm}|*{10}{>{\centering\arraybackslash}m{1cm}|}}

\hline
$(q/2,q/2)$ & $(1,1)$ & $(2,2)$ & $(3,3)$ & $(4,4)$ & $(5,5)$ & $(6,6)$ & $(7,7)$ & $(8,8)$ & $(9,9)$ & $(10,10)$ \\
 \hline
\emph{Unitary Bicycles} &0 & 0 &  2 &  6 &  22 &  70 &  236 &  784 &  2672 & 9174\\
\emph{Serre Unicycles} & 1 &  0 & 1 & 0 & 3 & 0 & 9 & 0 & 28 & 0 \\
\emph{Serre Bicycles} &0 & 1 & 0 & 2 & 0 & 5 & 0 & 16 & 0 & 51\\
\hline
Total & 1 & 1 & 3 & 8 & 25 & 75 & 245 & 800 & 2700 & 9225\\
 \hline
\end{tabular}
\end{center}
\end{footnotesize}
\caption{$\#\widetilde{\ind}(q/2,q/2)$ for even $q$ where $1 \leq q/2 \leq 10$.}
\label{table:parallel-sig-list}
\end{table}

\bibliographystyle{alpha}
\bibliography{citations}

\end{document}